\documentclass[10pt]{elsarticle}
\usepackage{lmodern,microtype}
\usepackage{subcaption}
\usepackage{amssymb, amsthm, mathtools, physics}

\usepackage[
  colorlinks,
]{hyperref}

\usepackage[noabbrev,capitalise,nameinlink]{cleveref}
\usepackage[dvipsnames]{xcolor}
\usepackage{fullpage}

\newcommand{\fn}[1]{\MakeUppercase{#1}} 
\renewcommand{\vec}[1]{\boldsymbol{#1}} 
\newcommand{\fvec}[1]{\vec{\fn{#1}}}
\newcommand{\vc}[1]{\underline{#1}{}}
\newcommand{\mat}[1]{\underline{\underline{#1}}_{}{}}
\newcommand{\avg}[1]{\{\!\{ #1 \}\!\}}

\newcommand{\etal}{\emph{et al.}\ }
\newcommand{\df}{\coloneqq}
\newcommand{\eqdf}{=\vcentcolon}
\newcommand{\T}{\mathrm{T}}
\newcommand{\nvolnodes}{N_q}
\newcommand{\nfacnodes}{N_{q_f}^{(\zeta)}}
\newcommand{\ncons}{N_c}
\newcommand{\nelem}{N_e}
\newcommand{\nfac}{N_f}
\newcommand{\npoly}{{N_p^*}}
\newcommand{\ngeom}{N_{p_g}^*}

\newtheorem{lemma}{Lemma}
\newtheorem{theorem}{Theorem}
\theoremstyle{definition}
\newtheorem{definition}{Definition}
\newtheorem{assumption}{Assumption}
\theoremstyle{remark}
\newtheorem{remark}{Remark}

\crefname{assumption}{Assumption}{Assumptions}

\journal{Journal of Computational Physics}

\biboptions{sort&compress}

\makeatletter
\def\ps@pprintTitle{%
\let\@oddhead\@empty
\let\@evenhead\@empty
\def\@oddfoot{\reset@font\hfil\thepage\hfil}
\let\@evenfoot\@oddfoot
}
\makeatother

\AtBeginDocument{\hypersetup{citecolor=green!50!black,linkcolor=red!50!black,urlcolor=blue!50!black}}

\begin{document}

\begin{frontmatter}

\title{Efficient entropy-stable discontinuous spectral-element methods using tensor-product summation-by-parts operators on triangles and tetrahedra}
\author{\texorpdfstring{Tristan Montoya\corref{cor1}}{Tristan Montoya}}
\ead{tristan.montoya@mail.utoronto.ca}
\author{David W. Zingg}
\address{University of Toronto Institute for Aerospace Studies, 4925 Dufferin Street, Toronto, Ontario M3H 5T6, Canada}
\cortext[cor1]{Corresponding author, \href{https://orcid.org/0000-0002-4259-1449}{ORCiD: \texttt{0000-0002-4259-1449}}.}

\begin{abstract}
We present a new class of efficient and robust discontinuous spectral-element methods of arbitrary order for nonlinear hyperbolic systems of conservation laws on curved triangular and tetrahedral unstructured grids. Such discretizations employ a recently introduced family of sparse tensor-product summation-by-parts (SBP) operators in collapsed coordinates within an entropy-conservative modal formulation, which is rendered entropy stable when a dissipative numerical flux is used at element interfaces. The proposed algorithms exploit the structure of such SBP operators alongside that of the Proriol--Koornwinder--Dubiner polynomial basis used to represent the numerical solution on the reference triangle or tetrahedron, and a weight-adjusted approximation is employed in order to efficiently invert the local mass matrix for curvilinear elements. Using such techniques, we obtain an improvement in time complexity from $\mathcal{O}(p^{2d})$ to $\mathcal{O}(p^{d+1})$ relative to existing entropy-stable formulations using multidimensional SBP operators not possessing such a tensor-product structure, where $p$ is the polynomial degree of the approximation and $d$ is the number of spatial dimensions. The number of required entropy-conservative two-point flux evaluations between pairs of quadrature nodes is accordingly reduced by a factor ranging from 1.56 at $p=2$ to 4.57 at $p=10$ for triangles, and from 1.88 at $p=2$ to 10.99 at $p=10$ for tetrahedra. Through numerical experiments involving smooth solutions to the compressible Euler equations on isoparametric triangular and tetrahedral grids, the proposed methods using tensor-product SBP operators are shown to exhibit similar levels of accuracy for a given mesh and polynomial degree to those using multidimensional operators based on symmetric quadrature rules, with both approaches achieving order $p+1$ convergence with respect to the element size in the presence of interface dissipation as well as exponential convergence with respect to the polynomial degree. Furthermore, both operator families are shown to give rise to entropy-stable schemes which exhibit excellent robustness for test problems characteristic of under-resolved turbulence simulations. Such results suggest that the algorithmic advantages resulting from the use of tensor-product operators are obtained without compromising accuracy or robustness, enabling the efficient extension of the benefits of entropy stability to higher polynomial degrees than previously considered for triangular and tetrahedral elements.
\end{abstract}

\begin{keyword}
Entropy stability \sep summation-by-parts \sep discontinuous Galerkin \sep conservation laws
\MSC[2020] 65M12 \sep 65M60 \sep 65M70
\end{keyword}

\end{frontmatter}

\section{Introduction}\label{sec:intro}
Hyperbolic or advection-dominated nonlinear systems of conservation laws constitute a class of partial differential equations (PDEs) of considerable importance in numerous scientific and engineering disciplines, with applications ranging from aircraft design to climate modeling. As a result of their complex multiscale behaviour and propensity to develop discontinuities, such PDEs present a significant challenge in the design of efficient, automated, and robust numerical methods. High-order \emph{discontinuous spectral-element methods} (DSEMs)\footnote{In this work, we use the term \emph{spectral-element method} (SEM) to refer to any numerical method achieving high-order accuracy through the use of multiple degrees of freedom within a given element. Note that we do not necessarily require a collocated tensor-product formulation, as is sometimes implied by the use of such terminology.} have emerged as an attractive approach for this class of problem, having received considerable attention in recent years due to their performance on modern hardware (see, for example, Kl\"ockner \etal \cite{klockner_dg_gpu_09}, Abdi \etal \cite{abdi_gpu_cg_dg_17}, and Vermiere \etal \cite{vermeire_17}) resulting from their relatively high arithmetic intensity (i.e.\ the ratio of floating-point operations to memory accesses) and data locality. Moreover, such schemes are highly flexible in their support for local variation in element size as well as polynomial degree, facilitating the adaptive solution of complex problems in an efficient and automated manner, as exemplified in recent papers by Parsani \etal \cite{parsani_ssdc_21} and Mossier \etal \cite{mossier_p_adaptive_dg_22}.
\par
High-order methods such as DSEMs notoriously lack robustness when applied to strongly nonlinear problems, particularly in the presence of under-resolved scales or discontinuous solutions. While such issues are traditionally addressed through \emph{ad hoc} stabilization techniques such as over-integration, modal filtering, slope limiting, or numerical dissipation, modern \emph{entropy-stable} formulations enable rigorous proofs of stability (under certain physical admissibility criteria) by guaranteeing that a strictly convex function of the numerical solution remains bounded for all time. For example, in the context of fluid dynamics, an entropy-stable discretization can be constructed so as to satisfy the Second Law of Thermodynamics in a discrete sense, a property which bounds the growth of the numerical solution, provided that the pressure and density remain positive. Entropy-stable schemes were introduced for nonlinear hyperbolic systems of conservation laws by Tadmor \cite{tadmor_entropy_stable_fv_87}, who devised first-order and second-order methods using specially designed two-point fluxes which conserve a particular mathematical entropy function. LeFloch \etal \cite{lefloch_entropy_conservative_arbitrary_order_02} extended Tadmor's approach to high-order accuracy on periodic domains in one dimension. The modern era of high-order entropy-stable discretizations, however, began when Fisher \cite{fisher_phd_thesis_12} combined entropy-conservative two-point flux functions with summation-by-parts (SBP) operators, which are discrete differential operators mimetic of integration by parts on bounded domains (see, for example, the review papers by Sv\"ard and Nordstr\"om \cite{svard_nordstrom_sbpreview_14} and Del Rey Fern\'andez \etal \cite{delrey_sbp_sat_review_14}). Fisher's approach (see also Fisher and Carpenter \cite{fisher_carpenter_ssweno_sbp_13} and Fisher \etal \cite{fisher_carpenter_telescopingflux_conservation_13}) enabled the construction of entropy-stable high-order finite-difference methods for the compressible Euler and Navier--Stokes equations on curvilinear block-structured grids using affordable entropy-conservative two-point flux functions proposed by Ismail and Roe \cite{ismail_roe_ec_flux_09}. Such a combination of SBP operators with two-point flux functions came to be known in the entropy stability community as \emph{flux differencing} (not to be confused with the similarly named \emph{flux difference splitting} technique introduced by Roe \cite{roe81} decades earlier for approximate Riemann solvers). 
\par 
The extension of entropy stability to DSEM formulations was made possible through the work of Gassner \cite{gassner_dgsem_sbp_13}, who recognized that the matrix operators employed within discontinuous Galerkin (DG) methods based on collocated Legendre--Gauss--Lobatto quadrature were, in fact, SBP operators, an equivalence which he exploited in the construction of provably stable discretizations of Burgers' equation using split forms first proposed in the finite-difference community. Entropy-stable DSEMs for systems of conservation laws on tensor-product quadrilateral and hexahedral elements were then introduced by Carpenter \etal \cite{carpenter_entropystable_collocation_14} and Gassner \etal \cite{gassner_winters_kopriva_splitform_nodaldg_sbp_16}, with the latter demonstrating that a broad class of split-form and entropy-stable DSEMs could be recovered by choosing different two-point flux functions. These schemes are distinguished from the approach taken by Barth \cite{barth_numerical_methods_gasdynamic_systems_99} as well as Hiltebrand and Mishra \cite{hiltebrand_mishra_entropystable_dg_exact_int_14} based on the work of Hughes \etal \cite{hughes_franca_mallet_entropy_stable_fem_86}, wherein space-time DG schemes are formulated in terms of the entropy variables. Unlike flux-differencing approaches, the latter methodology results in discretizations which are only entropy stable under the assumption that all integrals in the corresponding variational formulation are evaluated exactly (which is impractical, if not impossible, for many PDEs of interest to practitioners, including the compressible Euler and Navier--Stokes equations) and, furthermore, cannot be formulated explicitly in time.
\par
Entropy-stable DSEMs were generalized beyond tensor-product quadrature rules on quadrilaterals and hexahedra through the use of \emph{multidimensional SBP operators}, which were introduced by Hicken \etal \cite{hicken_mdsbp_16} as an extension of the generalized SBP framework proposed in a one-dimensional setting by Del Rey Fern\'andez \etal \cite{delrey_generalized_framework_14}. The subsequent development of entropy-stable high-order methods on triangles and tetrahedra (which first appeared in papers by Chen and Shu \cite{chen_shu_entropy_stable_dgsbp_17}, Crean \etal \cite{crean_entropystable_sbp_euler_curved_17}, and Chan \cite{chan_discretely_entropy_conservative_dg_sbp_18}) is outlined in a review paper by Chen and Shu \cite{chen_shu_dgsbp_review_19}. While advantageous in terms of geometric flexibility, such triangular and tetrahedral SBP operators do not possess a tensor-product structure, and, as such, result in schemes which become considerably more costly than comparable tensor-product DSEMs on quadrilaterals and hexahedra as the polynomial degree is increased. Such a difference is due in part to the fact that tensor-product spectral-element operators are amenable to \emph{sum-factorization} techniques, which were proposed in the context of spectral methods by Orszag \cite{orszag_spectral_complex_geometries_80}. Sum factorization, which involves the application of tensor-product operators in a dimension-by-dimension manner, results in the number of floating-point operations required to evaluate spatial operators scaling asymptotically as $\mathcal{O}(p^{d+1})$ with the polynomial degree $p$, where $d$ denotes the number of spatial dimensions. This compares favourably to the $\mathcal{O}(p^{2d})$ complexity obtained using multidimensional operators without a tensor-product structure. In the context of an entropy-stable scheme, multidimensional SBP operators are further disadvantaged due to the fact that they require the evaluation of entropy-conservative two-point flux functions (which, in the case of the Euler and Navier--Stokes equations, are relatively expensive operations involving the logarithmic mean) between all pairs of quadrature nodes, rather than simply along lines of nodes as in a tensor-product formulation. Due in part to these limitations, as well as the difficulty in constructing suitable quadrature rules, entropy-stable discretizations based on multidimensional (i.e.\ non-tensor-product) SBP operators rarely employ polynomial degrees greater than four or five. 
\par
In a recent paper \cite{montoya_sem_23}, the present authors introduced tensor-product SBP operators on triangles and tetrahedra based on collapsed coordinate systems, which were used within skew-symmetric nodal and modal DSEM formulations to obtain energy-stable discretizations of the linear advection equation on curved elements. These contributions served to extend the SBP framework to tensor-product operators in collapsed coordinates, which had previously been recognized in the SEM community (see, for example, Sherwin and Karniadakis \cite{sherwin_karniadakis_triangular_sem_95,sherwin_karniadakis_tetrahedra_hp_fem_96}, Lomtev and Karniadakis \cite{lomtev_karniadakis_dg_99}, Kirby \etal \cite{kirby_spectral_hp_dg_hybrid_grids_00}, and Moxey \etal \cite{moxey_matrix_free_triangles_20}) to be crucial for obtaining efficient discretizations on triangles and tetrahedra (as well as prisms and pyramids) at high polynomial degrees. Specifically, it was shown in \cite{montoya_sem_23} that by exploiting the tensor-product structure of the proposed SBP operators, and, in the case of the modal approach, making use of the ``warped'' tensor-product structure of the Proriol--Koornwinder--Dubiner (PKD) orthogonal polynomial basis \cite{proriol_polynomials_57,koornwinder_orthogonal_polynomials_75,dubiner_spectral_triangle_91} alongside a weight-adjusted approximation  proposed by Chan \etal \cite{chan_weight_adjusted_dg_curvilinear_17} to invert the curvilinear mass matrix, the time derivative could be obtained in $\mathcal{O}(p^{d+1})$ floating-point operations (i.e.\ the same asymptotic complexity as for the DSEMs on quadrilaterals and hexahedra described above) through sum factorization. Moreover, the modal formulations were shown to be similar in accuracy and in spectral radius (which, for linear problems, dictates the maximum stable time step for explicit temporal integration) to those using multidimensional operators on triangles and tetrahedra, while requiring far fewer floating-point operations at higher polynomial degrees. 
\par 
Building upon the above contributions, the primary objective of this paper is to apply the tensor-product SBP operators on triangles and tetrahedra introduced in \cite{montoya_sem_23} to the construction of efficient entropy-stable DSEMs of arbitrary order for nonlinear hyperbolic systems of conservation laws. The proposed methods make use of a modal PKD basis and a weight-adjusted approximation of the mass matrix inverse within an entropy-stable flux-differencing formulation, in which we exploit operator sparsity as well as sum factorization in order to obtain $\mathcal{O}(p^{d+1})$ complexity for all local elemental operations. In the tetrahedral case, we also approximate the metric terms arising from the mapping from reference to physical coordinates using a curl formulation from Chan and Wilcox \cite{chan_wilcox_entropystable_curvilinear_19} to satisfy the discrete metric identities, which must hold in order to ensure free-stream preservation and entropy stability on curvilinear meshes. These contributions enable the efficient extension of entropy-stable schemes on triangles and tetrahedra to polynomial degrees beyond those typically considered for such element types.
\par
We now outline the structure of the remainder of this paper. In \cref{sec:preliminaries}, we introduce some relevant notation and review some fundamental concepts relating to hyperbolic systems of conservation laws and SBP operators as well as classical orthogonal polynomials and Gaussian quadrature rules. In \cref{sec:tensor_sbp}, we review the relevant aspects of  the approach introduced in \cite{montoya_sem_23} for the construction of tensor-product SBP operators on triangles and tetrahedra. \Cref{sec:methods} describes a general framework for the construction of entropy-stable DSEMs using tensor-product or multidimensional SBP operators on curved triangular or tetrahedral elements, while \cref{sec:implementation} describes algorithmic considerations for the efficient implementation of the proposed schemes. The methods are then analyzed with respect to their conservation, free-stream preservation, and entropy stability properties in \cref{sec:analysis}. Numerical studies of accuracy as well as robustness for the compressible Euler equations are presented in \cref{sec:numerical}, and conclusions are provided in \cref{sec:conclusions}.

\section{Preliminaries}\label{sec:preliminaries}
\subsection{Notation} 
The notation in this paper follows that introduced by the authors in \cite{montoya_tensor_product_22}, \cite{montoya_unifying_21}, and \cite{montoya_sem_23}. Single underlines are used to denote vectors (treated as column matrices), whereas double underlines denote matrices. Symbols in bold are used specifically to denote Cartesian (i.e.\ spatial) vectors, for which we employ the usual dot product $\vec{x} \cdot \vec{y} \df x_1y_1 + \cdots + x_dy_d$ and Euclidean norm $\lVert \vec{x} \rVert^2 \df \vec{x} \cdot \vec{x}$. The symbol $\nabla$ is used to denote componentwise differentiation with respect to either type of vector, for example, as $\nabla_{\vec{x}} \df [\partial/\partial_{x_1}, \ldots, \partial/\partial_{x_d}]^\T$ or $\nabla_{\vc{U}} \df [\partial/\partial_{U_1}, \ldots, \partial/\partial_{U_d}]^\T$. The symbols $\mathbb{R}$, $\mathbb{R}^+$, $\mathbb{R}_0^+$, $\mathbb{N}$, $\mathbb{N}_0$, and $\mathbb{S}^{d-1}$ denote the real numbers, the positive real numbers, the non-negative real numbers, the natural numbers (excluding zero), the natural numbers including zero, and the unit $(d-1)$-sphere $\mathbb{S}^{d-1} \df \{ \vec{x} \in \mathbb{R}^d : \lVert \vec{x} \rVert = 1\}$, respectively. The symbols $\vc{0}^{(N)}$ and $\vc{1}^{(N)}$ are reserved for vectors of length $N \in \mathbb{N}$ containing all zeros and all ones, respectively, where the superscript is omitted when clear from the context, and the symbol $\mat{0}^{(M \times N)}$ likewise denotes an $M$ by $N$ matrix of zeros. The notation $\{1:N\}$ is used as shorthand for the index set $\{1,2,\ldots,N\}$. Given an arbitrary bounded domain $\mathcal{D} \subset \mathbb{R}^d$, we use $\partial\mathcal{D}$ to denote its boundary and $\bar{\mathcal{D}} \df \mathcal{D} \cup \partial \mathcal{D}$ to denote its closure. The space of polynomials of maximum total degree $p \in \mathbb{N}_0$ on $\mathcal{D}$ is then defined as $\mathbb{P}_p(\mathcal{D}) \df \operatorname{span}\{\mathcal{D} \ni \vec{x} \mapsto x_1^{\alpha_1} \cdots x_d^{\alpha_d} : \vec{\alpha} \in \mathcal{N}(p)\}$ in terms of the multi-index set $\mathcal{N}(p) \df \{\vec{\alpha} \in \mathbb{N}_0^d : \alpha_1 + \cdots + \alpha_d \leq p\}$ of cardinality $\npoly \df \binom{p+d}{d}$. Other relevant notational conventions and definitions are introduced as they appear. 

\subsection{Systems of conservation laws and the entropy inequality}
We are interested in systems of conservation laws governing the evolution of $\ncons \in \mathbb{N}$ variables $\vc{\fn{u}}(\vec{x},t) \in \Upsilon \subset \mathbb{R}^{\ncons}$  on the domain $\Omega \subset \mathbb{R}^d$ over the time interval $(0,T) \subset \mathbb{R}_0^+$, taking the form
\begin{subequations}\label{eq:cons_law}
\begin{alignat}{2}
\frac{\partial\vc{\fn{u}}(\vec{x},t)}{\partial t} + \sum_{m=1}^d \frac{\partial\vc{\fn{f}}_m(\vc{\fn{u}}(\vec{x},t))}{\partial x_m}&= \vc{0}, \qquad && \forall \, (\vec{x},t) \in \Omega \times (0,T),\label{eq:pde}\\
\fn{\vc{u}}(\vec{x},0) &= \fn{\vc{u}}^0(\vec{x}), \qquad && \forall \, \vec{x} \in \Omega,\label{eq:ic}
\end{alignat}
\end{subequations}
subject to appropriate boundary conditions, where $\vc{\fn{f}}_m(\vc{\fn{u}}(\vec{x},t)) \in \mathbb{R}^{\ncons}$ is the $m^{\mathrm{th}}$ Cartesian flux component and $\vc{\fn{u}}^0(\vec{x}) \in \Upsilon$ represents the initial data, where $\Upsilon$ denotes the set of admissible solution states. For convenience, we define the flux in any direction $\vec{n} \in \mathbb{S}^{d-1}$ for an arbitrary state $\vc{U} \in \Upsilon$ as 
\begin{equation}
\vc{\fn{f}}(\vc{U},\vec{n}) \df \sum_{m=1}^d n_m \vc{\fn{F}}_m(\vc{U}).
\end{equation}
A system of conservation laws in the form of \eqref{eq:cons_law} is then called \emph{hyperbolic} if the flux Jacobian $\nabla_{\vc{U}}\vc{\fn{f}}(\vc{U},\vec{n}) \in \mathbb{R}^{\ncons \times \ncons}$ is diagonalizable with all real eigenvalues for all $\vc{U} \in \Upsilon$ and $\vec{n} \in \mathbb{S}^{d-1}$. 

\begin{remark}
In this work, we tacitly assume that numerical solutions remain within the admissible set $\Upsilon$, corresponding, for example, to the requirement for thermodynamic quantities such as pressure or density to be positive. This is difficult to ensure \emph{a priori} for nonlinear problems, and often requires the use of bespoke limiting procedures, which are described in the context of entropy-stable schemes, for example, in recent papers by Rueda-Ram\'irez \etal \cite{rueda_ramirez_subcell_limiting_22}, Yamaleev and Upperman \cite{yamaleev_positivity_entropy_stable_23}, and Lin \etal \cite{lin_positivity_preservation_euler_23}.
\end{remark}
We are specifically interested in hyperbolic systems of conservation laws endowed with an \emph{entropy function} and corresponding \emph{entropy flux}, which are defined as follows. 
\begin{definition}\label{def:entropy}
The functions $\mathcal{S} : \Upsilon \to \mathbb{R}$ and $\vec{\mathcal{F}} : \Upsilon \to \mathbb{R}^d$ are, respectively, an \emph{entropy function} and \emph{entropy flux} if $\mathcal{S}$ is strictly convex (i.e.\ its Hessian is positive definite), and the relation
\begin{equation}\label{eq:entropy_contract}
\big(\nabla_{\vc{U}} \mathcal{F}_m(\vc{U})\big)^\T = \big(\nabla_{\vc{U}}\mathcal{S}(\vc{U})\big)^\T \big(\nabla_{\vc{U}} \vc{\fn{f}}_m(\vc{U})\big), \quad \forall \, \vc{U} \in \Upsilon,
\end{equation}
holds for all $m \in \{1:d\}$, where $\nabla_{\vc{U}} \mathcal{S}(\vc{U}), \nabla_{\vc{U}} \mathcal{F}_m(\vc{U}) \in \mathbb{R}^{\ncons}$ are the gradients of the entropy and entropy flux components with respect to the conservative variables, respectively, and $\nabla_{\vc{U}} \vc{\fn{f}}_m(\vc{U}) \in \mathbb{R}^{\ncons \times \ncons}$ denotes the Jacobian of the $m^{\mathrm{th}}$ flux component with respect to the conservative variables. 
\end{definition}
The entries of the vector $\vc{\mathcal{W}}(\vc{U}) \df \nabla_{\vc{U}}\mathcal{S}(\vc{U})$ are referred to as the \emph{entropy variables}, where the mapping $\vc{\mathcal{W}}$ has an inverse given by $\vc{\mathcal{U}}$ due to the strict convexity of $\mathcal{S}$ over the admissible set $\Upsilon$. As described by Friedrichs and Lax \cite{friedrichs_lax_systems_conservation_71}, the existence of an entropy--entropy flux pair in the sense of \cref{def:entropy} implies that any classical (i.e.\ continuously differentiable) solution to \eqref{eq:cons_law} satisfies an additional conservation equation of the form
\begin{equation}\label{eq:entropy_diff}
\frac{\partial \fn{\mathcal{S}}(\vc{\fn{u}}(\vec{x},t))}{\partial t} + \nabla_{\vec{x}} \cdot \vec{\mathcal{F}}(\vc{\fn{U}}(\vec{x},t)) = 0, \quad \forall \, (\vec{x},t) \in \Omega \times (0,T).
\end{equation}
Integrating \eqref{eq:entropy_diff} over the spatial domain and using the divergence theorem then results in 
\begin{equation}\label{eq:entropy_equality}
\begin{aligned}
\frac{\dd}{\dd t} \int_{\Omega} \mathcal{S}(\vc{\fn{U}}(\vec{x},t)) \, \dd \vec{x} &= -\int_{\partial\Omega} \vec{\mathcal{F}}(\vc{\fn{U}}(\vec{x},t)) \cdot \vec{n}(\vec{x}) \, \dd s\\
&= \int_{\partial\Omega} \Big(\vec{\varPsi}\big(\vc{\mathcal{W}}(\vc{\fn{U}}(\vec{x},t))\big) \cdot \vec{n}(\vec{x}) - \big(\vc{\mathcal{W}}(\vc{\fn{U}}(\vec{x},t))\big)^\T\vc{\fn{f}}(\vc{\fn{U}}(\vec{x},t),\vec{n}(\vec{x}))\Big) \, \dd s,
\end{aligned}
\end{equation}
where $\vec{n}(\vec{x}) \in \mathbb{S}^{d-1}$ denotes the outward unit normal to $\partial\Omega$, and we have defined the \emph{flux potential} $\vec{\varPsi}(\vc{W}) \in \mathbb{R}^d$ such that $\varPsi_m(\vc{W}) \df \vc{W}^\T \vc{\fn{f}}_m(\vc{\mathcal{U}}(\vc{W})) - \mathcal{F}_m(\vc{\mathcal{U}}(\vc{W}))$. We are, however, often interested in \emph{weak solutions}, which satisfy \eqref{eq:cons_law} in the sense of distributions and can therefore be discontinuous, describing phenomena such as shock waves. Replacing \eqref{eq:entropy_equality} with an \emph{entropy inequality} of the form
\begin{equation}\label{eq:entropy_balance}
\frac{\dd}{\dd t} \int_{\Omega} \mathcal{S}(\vc{\fn{U}}(\vec{x},t)) \, \dd \vec{x} \leq  \int_{\partial\Omega} \Big(\vec{\varPsi}\big(\vc{\mathcal{W}}(\vc{\fn{U}}(\vec{x},t))\big) \cdot \vec{n}(\vec{x}) - \big(\vc{\mathcal{W}}(\vc{\fn{U}}(\vec{x},t))\big)^\T\vc{\fn{f}}(\vc{\fn{U}}(\vec{x},t),\vec{n}(\vec{x}))\Big) \, \dd s
\end{equation}
then provides an admissibility criterion for physically relevant weak solutions (see, for example, Kru{\v{z}}kov \cite{kruzhkov_first_order_quasilinear_70} or Lax \cite{lax_shock_waves_and_entropy_71}). Provided that $\vc{\fn{U}}(\vec{x},t)$ remains within $\Upsilon$ and that the boundary conditions are imposed correctly, it can be shown (see, for example, Dafermos \cite{dafermos16}) that \eqref{eq:entropy_balance} implies a bound on the solution itself due to the strict convexity of the entropy function. We are therefore interested in constructing discretizations which respect a semi-discrete form of such an entropy bound.
\begin{remark}
The requirement for a strictly convex entropy function in \cref{def:entropy} is consistent with the mathematical literature, but is opposite the convention used in physics. As such, the inequality in \eqref{eq:entropy_balance} is a generalized statement of the Second Law of Thermodynamics, up to a change in sign.
\end{remark}

\subsection{Summation-by-parts operators}\label{sec:sbp}
The discretizations described in this paper involve the construction of SBP operators on a canonical reference element, which are then used to discretize the PDE on an unstructured mesh through the use of a bijective mapping from the reference element to each physical element. We now require the following definition of a nodal SBP operator from \cite{hicken_mdsbp_16}, which extends the generalized definition proposed in \cite{delrey_generalized_framework_14} to the multidimensional setting.
\begin{definition}\label{def:sbp}
Let $\hat{\Omega} \subset \mathbb{R}^d$ denote a closed, bounded, and connected reference domain on which we define a set of $\nvolnodes \in \mathbb{N}$ distinct nodes $\{\vec{\xi}^{(i)}\}_{i\in \{1:\nvolnodes\}} \subset \hat{\Omega}$, and let the vectors
\begin{equation}\label{eq:nodal_vectors}
\vc{u} \df \big[\fn{u}(\vec{\xi}^{(1)}),\ldots, \fn{u}(\vec{\xi}^{(\nvolnodes)})\big]^\T \quad \text{and} \quad \vc{v} \df \big[\fn{v}(\vec{\xi}^{(1)}),\ldots, \fn{v}(\vec{\xi}^{(\nvolnodes)})\big]^\T
\end{equation}
contain the nodal values of arbitrary functions $\fn{u},\fn{v} : \hat{\Omega} \to \mathbb{R}$. A matrix $\mat{D}^{(m)} \in \mathbb{R}^{\nvolnodes \times \nvolnodes}$ approximating the partial derivative $\partial/\partial \xi_m$ is then a \emph{nodal SBP operator} of degree $q \in \mathbb{N}_0$ if it satisfies
\begin{equation}\label{eq:accuracy}
\mat{D}^{(m)}\vc{v} = \big[(\partial\fn{v}/\partial\xi_m)(\vec{\xi}^{(1)}), \ldots, (\partial\fn{v}/\partial\xi_m)(\vec{\xi}^{(\nvolnodes)}) \big]^\T, \quad \forall \, \fn{v} \in \mathbb{P}_p(\hat{\Omega}),
\end{equation}
and may be decomposed as $\mat{D}^{(m)} = \mat{W}^{-1}\mat{Q}^{(m)}$ such that $\mat{W}\in \mathbb{R}^{\nvolnodes\times \nvolnodes}$ is symmetric positive-definite (SPD) and $\mat{Q}^{(m)}\in \mathbb{R}^{\nvolnodes\times \nvolnodes}$ satisfies the \emph{SBP property} given by $
\mat{Q}^{(m)} + (\mat{Q}^{(m)})^\T = \mat{E}^{(m)}$, where
\begin{equation}\label{eq:facet_accuracy}
\vc{u}^\T\mat{E}^{(m)}\vc{v} = \int_{\partial\hat{\Omega}} \fn{u}(\vec{\xi})\fn{v}(\vec{\xi}) \hat{n}_m(\vec{\xi}) \, \dd\hat{s}, \quad \forall \, \fn{u},\fn{v} \in \mathbb{P}_r(\hat{\Omega}),
\end{equation}
holds for some $r \geq q$, and $\hat{\vec{n}} : \partial\hat{\Omega} \to \mathbb{S}^{d-1}$ denotes the outward unit normal vector to $\hat{\Omega}$.
\end{definition}
The SBP property is equivalent to the decomposition $\mat{Q}^{(m)} = \mat{S}^{(m)} + \frac{1}{2}\mat{E}^{(m)}$, where $\mat{S}^{(m)}$ is skew-symmetric and $\mat{E}^{(m)}$ is a symmetric matrix satisfying the conditions in \eqref{eq:facet_accuracy}. Moreover, SBP operators mimic integration by parts in a discrete sense, as given by
\begin{equation}
\begin{alignedat}{3}
\underset{\rotatebox{90}{$\, \approx $}}{\int_{\hat{\Omega}} \fn{u}(\vec{\xi}) \frac{\partial\fn{v}(\vec{\xi})}{\partial\xi_m} \, \dd \vec{\xi} }  &+ \underset{\rotatebox{90}{$\, \approx $}}{\int_{\hat{\Omega}}\frac{\partial\fn{u}(\vec{\xi})}{\partial\xi_m} \fn{v}(\vec{\xi})  \, \dd \vec{\xi}} & = & \  \underset{\rotatebox{90}{$\, \approx $}}{ \int_{\partial\hat{\Omega}} \fn{u}(\vec{\xi})\fn{v}(\vec{\xi}) \hat{n}_m(\vec{\xi}) \, \dd\hat{s}}. \\ 
\vc{u}^\T\mat{Q}^{(m)}\vc{v} \qquad &+ \quad \vc{u}^\T\big(\mat{Q}^{(m)}\big)^\T\vc{v} & \ = & \qquad\quad \vc{u}^\T\mat{E}^{(m)}\vc{v}
\end{alignedat}
\end{equation}
We refer to any SBP operator for which the associated SPD matrix $\mat{W}$ is diagonal as a \emph{diagonal-norm} SBP operator. In such a case, it was shown in \cite[Theorem 3.2]{hicken_mdsbp_16} that the diagonal entries of $\mat{W}$ constitute the weights $\{\omega^{(i)}\}_{i \in \{1:\nvolnodes\}} \subset \mathbb{R}^+$ for a quadrature rule satisfying
\begin{equation}\label{eq:sbp_quadrature}
\sum_{i=1}^{\nvolnodes} \fn{v}(\vec{\xi}^{(i)})\, \omega^{(i)} = \int_{\hat{\Omega}} \fn{v}(\vec{\xi})\, \dd \vec{\xi}, \quad \forall \, \fn{v} \in \mathbb{P}_{\tau}(\hat{\Omega}),
\end{equation}
which is of degree $\tau \geq 2q - 1$ for any diagonal-norm SBP operator of degree $q$. 
\begin{remark}
The polynomial exactness condition in \eqref{eq:accuracy} implies that any SBP operator of degree $q$ is, by definition, also an SBP operator of any positive integer less than $q$. Where such ambiguity arises, the degree $q$ of an SBP operator is taken to refer uniquely to the \emph{maximum} integer value for which \eqref{eq:accuracy} holds, and we adopt an analogous convention for quadrature rules such as that in \eqref{eq:sbp_quadrature}. 
\end{remark}
\begin{remark}
Although a multidimensional discretization requires an SBP operator of the form $\mat{D}^{(m)} = \mat{W}^{-1}\mat{Q}^{(m)}$ to be constructed for each partial derivative $\partial / \partial x_m$, it will be assumed throughout this work, as is typical in the SBP literature, that $\mat{W}$ is the same for all coordinate indices $m \in \{1:d\}$.
\end{remark}

\subsection{Decomposition of the boundary operators}\label{sec:sbp_boundary}
Let us now assume that the reference element $\hat{\Omega} \subset \mathbb{R}^d$ is a polytope, corresponding to a polygon in two dimensions or a polyhedron in three dimensions, and partition its boundary into $\nfac \in \mathbb{N}$ closed subsets $\{ \hat{\Gamma}^{(\zeta)}\}_{\zeta \in \{1:\nfac\}}$ with disjoint interiors, which we denote as \emph{facets}. Each facet $\hat{\Gamma}^{(\zeta)} \subset \partial\hat{\Omega}$ is assumed to be flat, and we denote its (constant) outward unit normal vector by $\vec{n}^{(\zeta)} \in \mathbb{S}^{d-1}$. On each facet, we then introduce $\nfacnodes \in \mathbb{N}$ quadrature nodes and weights given, respectively, by
\begin{equation}\label{eq:facet_quad}
\{\vec{\xi}^{(\zeta,i)}\}_{i \in \{1:\nfacnodes\}} \subset \hat{\Gamma}^{(\zeta)}, \quad  \{\omega^{(\zeta,i)}\}_{i \in \{1:\nfacnodes\}} \subset \mathbb{R}_0^+.
\end{equation}
As in Del Rey Fern\'andez \etal \cite[Section 3]{delrey_mdsbp_sat_18}, we restrict our attention to the class of SBP operators for which the boundary matrices in \eqref{eq:facet_accuracy} can be constructed as
\begin{equation}\label{eq:e_decomp}
\mat{E}^{(m)} \df \sum_{\zeta=1}^{\nfac} \hat{n}_m^{(\zeta)} \big(\mat{R}^{(\zeta)}\big)^\T \mat{B}^{(\zeta)}\mat{R}^{(\zeta)},
\end{equation}
where $\mat{B}^{(\zeta)} \in \mathbb{R}^{\nfacnodes\times\nfacnodes}$ is a diagonal matrix with entries $B_{ij}^{(\zeta)} \df \omega^{(\zeta,i)}\delta_{ij}$, and $\mat{R}^{(\zeta)} \in \mathbb{R}^{\nfacnodes \times \nvolnodes}$ is an interpolation/extrapolation operator of degree $r^{(\zeta)} \geq p$, satisfying the accuracy conditions
\begin{equation}\label{eq:extrap_acc}
\mat{R}^{(\zeta)}\vc{v} = \big[\fn{v}(\vec{\xi}^{(\zeta,1)}), \ldots, \fn{v}(\vec{\xi}^{(\zeta,\nfacnodes)})\big]^\T, \quad \forall \, \fn{v} \in \mathbb{P}_{r^{(\zeta)}}(\hat{\Omega}).
\end{equation}
As a special case of the decomposition in \eqref{eq:e_decomp}, we note that $\mat{E}^{(m)}$ can be made diagonal by constructing SBP operators for which the nodal sets in \eqref{eq:facet_quad} all form subsets of the volume quadrature nodes, wherein $\mat{R}^{(\zeta)}$ simply selects boundary values from nodal vectors such as those in \eqref{eq:nodal_vectors}. Such SBP operators, denoted here as \emph{diagonal-E} operators, are analogous to those on one-dimensional nodal sets including both endpoints, and were first constructed on triangles by Chen and Shu \cite{chen_shu_entropy_stable_dgsbp_17}.

\subsection{Orthogonal polynomials and Gaussian quadrature rules}\label{sec:poly_quad}
The fundamental building blocks used for constructing the triangular and tetrahedral SBP operators developed by the authors in \cite{montoya_sem_23} are Jacobi and Legendre polynomials as well as their associated Gaussian quadrature rules and interpolants, the basic properties of which we will review here. The normalized \emph{Jacobi polynomials} are denoted by $\fn{P}_i^{(a,b)} \in \mathbb{P}_i([-1,1])$, satisfying
\begin{equation}
\int_{-1}^1 P_i^{(a,b)}(\eta) P_j^{(a,b)}(\eta)(1-\eta)^{a} (1+\eta)^{b} \, \dd \eta = \delta_{ij}, \quad \forall \, a,b > -1.
\end{equation}
Such polynomials can be constructed through recurrence relations, as shown, for example, by Hesthaven and Warburton \cite[Appendix A]{hesthaven08}. For a given non-negative integer $q$, the \emph{Gaussian quadrature rules} corresponding to a Jacobi weight with exponents $a$ and $b$ have nodes $\{\eta_{q,i}^{(a,b)}\}_{i\in\{0:q\}} \subset [-1,1]$ given by the $q+1$ solutions to a polynomial equation. Such equations are given by
\begin{subequations}
\begin{alignat}{2}
\text{Gauss:}& \quad &\fn{P}_{q+1}^{(a,b)}(\eta) &= 0, \\
\text{Gauss--Radau:} &\quad  &(1+\eta)\fn{P}_q^{(a,b+1)}(\eta) &= 0, \label{eq:gauss_radau}\\
\text{Gauss--Lobatto:}&\quad &(1-\eta^2)\fn{P}_{q-1}^{(a+1,b+1)}(\eta) &= 0,
\end{alignat}
\end{subequations}
where the Gauss, Gauss--Radau, and Gauss--Lobatto families of quadrature rules include zero, one, and two endpoints of the interval, respectively.\footnote{Here, we define Gauss--Radau quadrature rules including a node at the left endpoint. By flipping the sign of $\eta$ in \eqref{eq:gauss_radau}, we obtain a rule including the right endpoint instead. We also note that Gauss--Lobatto rules require $q \geq 1$.} The \emph{Lagrange polynomials} $\{\ell_{q,i}^{(a,b)}\}_{i\in\{0:q\}}$ associated with such a nodal set constitute a basis for $\mathbb{P}_q([-1,1])$ satisfying $\ell_{q,i}^{(a,b)}(\eta_{q,j}^{(a,b)}) = \delta_{ij}$ and are given by 
\begin{equation}\label{eq:lagrange}
\ell_{q,i}^{(a,b)}(\eta) \df \prod_{j \in \{0:q\} \setminus \{i\}} \frac{\eta - \eta_{q,j}^{(a,b)}}{\eta_{q,i}^{(a,b)}-\eta_{q,j}^{(a,b)}}.
\end{equation}
The corresponding Gaussian quadrature weights $\{\omega_{q,i}^{(a,b)}\}_{i\in\{0:q\}}$ can then be expressed as
\begin{equation}
\omega_{q,i}^{(a,b)} \df \int_{-1}^1 \ell_{q,i}^{(a,b)}(\eta)  (1-\eta)^{a} (1+\eta)^{b} \, \dd \eta,
\end{equation}
where alternative expressions for such weights can be found, for example, in references such as Karniadakis and Sherwin \cite[Appendix B]{karniadakis_sherwin_spectral_hp_element}. The resulting quadrature rule satisfies
\begin{equation}\label{eq:quad_1d}
\sum_{i=0}^{q} \fn{v}(\eta_{q,i}^{(a,b)})\, \omega_{q,i}^{(a,b)} = \int_{-1}^1 \fn{v}(\eta) (1-\eta)^{a} (1+\eta)^{b}\, \dd \eta, \quad \forall \, \fn{v} \in \mathbb{P}_{2q+\delta}([-1,1]),
\end{equation} 
where $\delta = 1$ for Gauss nodes, $\delta = 0$ for Gauss--Radau nodes, and $\delta = -1$ for Lobatto nodes. In the special case of $a = b = 0$, we recover the familiar \emph{Legendre polynomials} and the \emph{Legendre--Gauss} (LG), \emph{Legendre--Gauss--Radau} (LGR), \emph{Legendre--Gauss--Lobatto} (LGL) quadrature rules.

\section{Tensor-product summation-by-parts operators in collapsed coordinates}\label{sec:tensor_sbp}
In an earlier paper \cite{montoya_sem_23}, the present authors proposed a procedure for constructing SBP operators of arbitrary degree on triangles and tetrahedra, which, unlike those developed previously, are sparse and possess a tensor-product structure, allowing for the application of sum-factorization techniques. In this section, we briefly review the essential steps involved in constructing such operators, where we refer to the above paper for details regarding the construction of SBP operators and more broadly to Karniadakis and Sherwin \cite{karniadakis_sherwin_spectral_hp_element} regarding SEM formulations in collapsed coordinates.

\subsection{Tensor-product summation-by-parts operators on the reference triangle}
The reference element on which we construct our triangular SBP operators is taken to be
\begin{equation}\label{eq:reference_triangle}
\hat{\Omega} \df \big\{\vec{\xi} \in [-1,1]^2 : \xi_1 + \xi_2 \leq 0  \big\},
\end{equation}
where, as a convention, we number the facets (i.e.\ edges of the triangle) counter-clockwise as
\begin{equation}\label{eq:tri_facets}
\begin{gathered}
\hat{\Gamma}^{(1)} \df \big\{ \vec{\xi} \in \hat{\Omega} : \xi_2 = -1 \big\}, \quad \hat{\Gamma}^{(2)} \df \big\{ \vec{\xi} \in \hat{\Omega} : \xi_1 + \xi_2 = 0 \big\}, \quad
\hat{\Gamma}^{(3)} \df \big\{ \vec{\xi} \in \hat{\Omega} : \xi_1 = -1 \big\}.
\end{gathered}
\end{equation}
The collapsed coordinate transformation from the square to the reference triangle is then given by
\begin{equation}\label{eq:tri_mapping}
\vec{\chi}(\vec{\eta}) \df  \mqty[\tfrac{1}{2}(1+\eta_1)(1-\eta_2) - 1,\\ \eta_2],
\end{equation}
which is illustrated in \cref{fig:mapping_2d}. 
\begin{figure}[!t]
\centering
\begin{subfigure}[c]{0.35\textwidth}
\centering
\includegraphics[height=42mm]{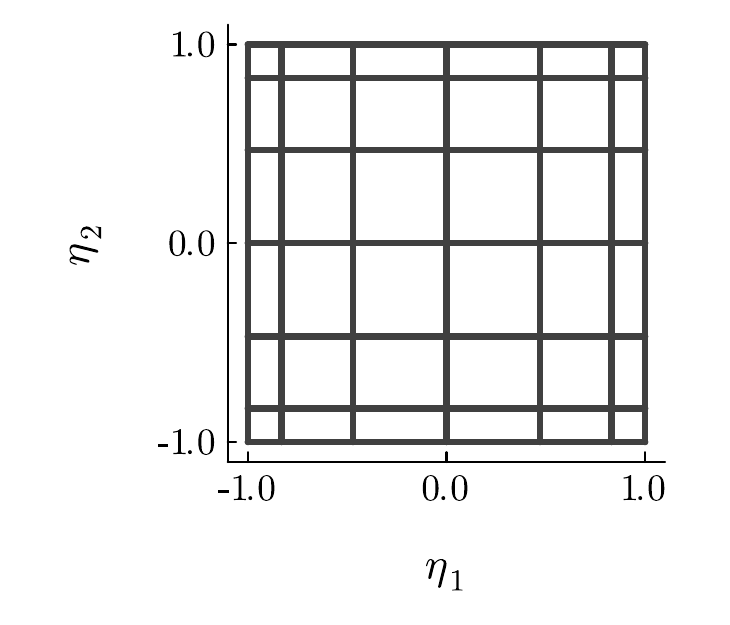}
\end{subfigure}~
\begin{subfigure}[c]{0.15\textwidth}
\centering
\Large $\longrightarrow$\\
$\vec{\chi}$
\end{subfigure}~
\begin{subfigure}[c]{0.35\textwidth}
\centering
\includegraphics[height=42mm]{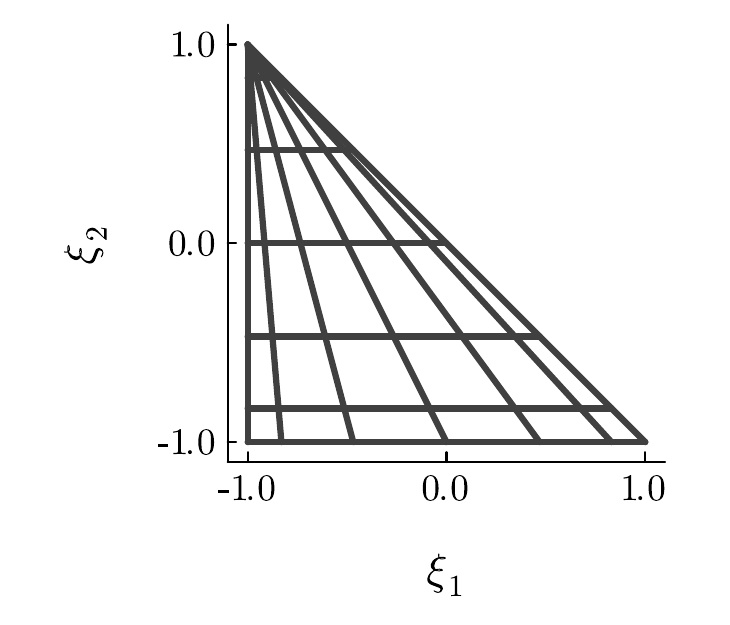}
\end{subfigure}
\caption{Illustration of the collapsed coordinate transformation $\vec{\xi} = \vec{\chi}(\vec{\eta})$ from the square to the reference triangle}\label{fig:mapping_2d}
\end{figure}
Letting $q_m \in \mathbb{N}_0$ denote the degree of the approximation in the $\eta_m$ coordinate, tensor-product quadrature rules are obtained with nodes $\{\vec{\xi}^{(i)}\}_{i\in \{1:\nvolnodes\}} \subset \hat{\Omega}$ and corresponding weights $\{\omega^{(i)}\}_{i\in \{1:\nvolnodes\}}\subset \mathbb{R}^+$ given for a multi-index $\vec{\alpha} \in \{0:q_1\} \times \{0:q_2\}$ as
\begin{equation}\label{eq:tri_quadrature}
\vec{\xi}^{(\sigma(\vec{\alpha}))} \df \vec{\chi}\big(\eta_{q_1,\alpha_1}^{(0,0)},\eta_{q_2,\alpha_2}^{(0,0)}\big), \quad \omega^{(\sigma(\vec{\alpha}))} \df \frac{1-\eta_{q_2,\alpha_2}^{(0,0)}}{2}\omega_{q_1,\alpha_1}^{(0,0)}\omega_{q_2,\alpha_2}^{(0,0)},
\end{equation}
where $\sigma : \{0:q_1\} \times \{0:q_2\} \to \{1:\nvolnodes\}$ is a bijective mapping which defines an ordering of the $\nvolnodes \df (q_1+1)(q_2+1)$ volume quadrature nodes, and we use the notation introduced in \cref{sec:poly_quad} for Gaussian quadrature rules. Similarly, we let $q_f \in \mathbb{N}_0$ and define the $\nfacnodes \df q_f + 1$ facet quadrature nodes and weights in \eqref{eq:facet_quad} as
\begin{equation}\label{eq:tri_facet_quadrature_nodes}
\begin{alignedat}{3}
\vec{\xi}^{(1,i)} &\df \vec{\chi}(\eta_{q_f,i-1}^{(0,0)},-1), \quad &\vec{\xi}^{(2,i)} &\df \vec{\chi}(1,\eta_{q_f,i-1}^{(0,0)}), \quad &\vec{\xi}^{(3,i)} &\df \vec{\chi}(-1,\eta_{q_f,i-1}^{(0,0)}),\\
\omega^{(1,i)} &\df \omega_{q_f,i-1}^{(0,0)}, \quad &\omega^{(2,i)} &\df \sqrt{2} \omega_{q_f,i-1}^{(0,0)}, \quad &\omega^{(3,i)} &\df  \omega_{q_f,i-1}^{(0,0)}.
\end{alignedat}
\end{equation}
Defining the volume and facet quadrature weight matrices $\mat{W} \in \mathbb{R}^{\nvolnodes \times \nvolnodes}$ and $\mat{B}^{(\zeta)} \in \mathbb{R}^{\nfacnodes\times\nfacnodes}$ with entries given by $W_{ij} \df \omega^{(i)}\delta_{ij}$ and $B_{ij}^{(\zeta)} \df \omega^{(\zeta,i)}\delta_{ij}$, respectively, derivative operators $\mat{D}^{(m)} \in \mathbb{R}^{\nvolnodes \times \nvolnodes}$ can be constructed by applying the chain rule to a tensor-product Lagrange interpolant in collapsed coordinates, resulting in
\begin{subequations}\label{eq:d_2d}
\allowdisplaybreaks
\begin{align}
D_{\sigma(\vec{\alpha}) \sigma(\vec{\beta})}^{(1)} &\df \frac{2}{1-\eta_{q_2,\alpha_2}^{(0,0)}} \frac{\dd \ell_{q_1,\beta_1}^{(0,0)}}{\dd \eta_1} (\eta_{q_1,\alpha_1}^{(0,0)}) \delta_{\alpha_2\beta_2}, \label{eq:d1_2d}\\
D_{\sigma(\vec{\alpha}) \sigma(\vec{\beta})}^{(2)} &\df \frac{1+\eta_{q_1,\alpha_1}^{(0,0)}}{1-\eta_{q_2,\alpha_2}^{(0,0)}} \frac{\dd \ell_{q_1,\beta_1}^{(0,0)}}{\dd \eta_1} (\eta_{q_1,\alpha_1}^{(0,0)}) \delta_{\alpha_2\beta_2} + \delta_{\alpha_1\beta_1}\frac{\dd \ell_{q_2,\beta_2}^{(0,0)}}{\dd \eta_2} (\eta_{q_2,\alpha_2}^{(0,0)}).\label{eq:d2_2d}
\end{align}
\end{subequations}
Similarly, the matrices $\mat{R}^{(\zeta)} \in \mathbb{R}^{\nfacnodes\times\nvolnodes}$ have entries given by
\begin{subequations}\label{eq:r_2d}
\begin{align}
R_{i, \sigma(\vec{\beta})}^{(1)} &\df \ell_{q_1,\beta_1}^{(0,0)}(\eta_{q_f,i-1}^{(0,0)})\ell_{q_2,\beta_2}^{(0,0)}(-1),\\ 
R_{i, \sigma(\vec{\beta})}^{(2)} &\df \ell_{q_1,\beta_1}^{(0,0)}(1)\ell_{q_2,\beta_2}^{(0,0)}(\eta_{q_f,i-1}^{(0,0)}),\\
R_{i, \sigma(\vec{\beta})}^{(3)} &\df \ell_{q_1,\beta_1}^{(0,0)}(-1)\ell_{q_2,\beta_2}^{(0,0)}(\eta_{q_f,i-1}^{(0,0)}).
\end{align}
\end{subequations} 
In this work, we construct SBP operators of degree $q \in \mathbb{N}_0$ using LG quadrature rules in the $\eta_1$ and $\eta_2$ directions as well as on each facet, with $q_1 = q_2 = q_f = q$. As a consequence of \cite[Theorem 3.1]{montoya_sem_23}, such operators are guaranteed to satisfy the conditions of \cref{def:sbp} for any polynomial degree $q$, with boundary operators in the form of \eqref{eq:e_decomp}.

\subsection{Tensor-product summation-by-parts operators on the reference tetrahedron}\label{sec:sbp_tet}
\begin{figure}[!t]
\centering
\begin{subfigure}[c]{0.35\textwidth}
\centering
\includegraphics[height=42mm]{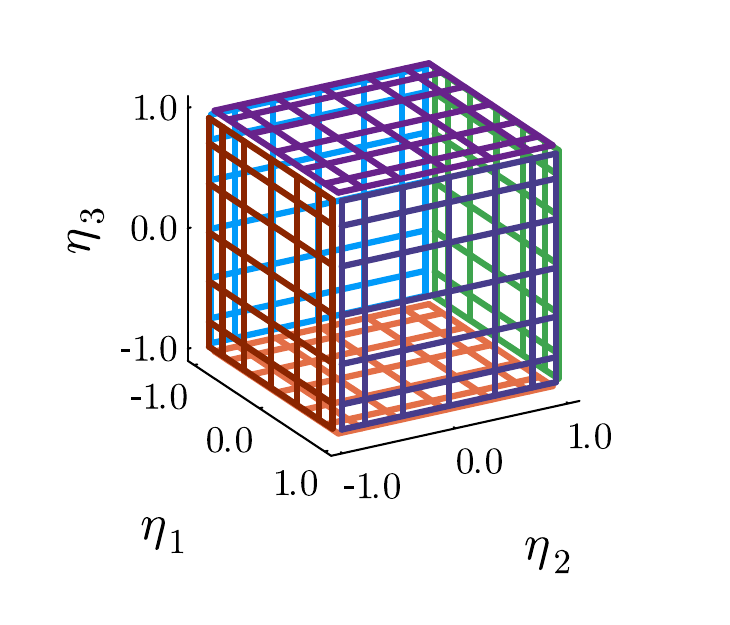}
\end{subfigure}~
\begin{subfigure}[c]{0.15\textwidth}
\centering
\Large $\longrightarrow$\\
$\vec{\chi}$
\end{subfigure}~
\begin{subfigure}[c]{0.35\textwidth}
\centering
\includegraphics[height=42mm]{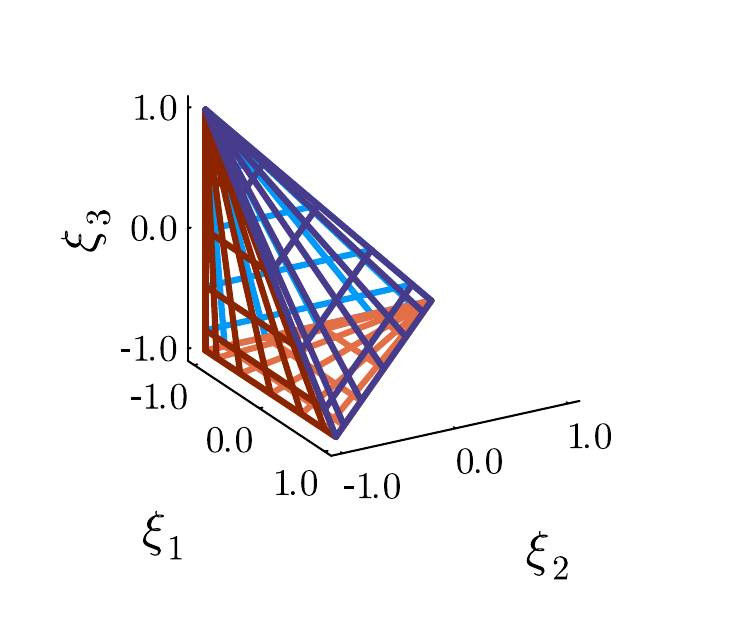}
\end{subfigure}
\caption{Illustration of the collapsed coordinate transformation $\vec{\xi} = \vec{\chi}(\vec{\eta})$ from the cube to the reference tetrahedron}\label{fig:mapping_3d}
\end{figure}
For tetrahedral elements, the reference domain is given by
\begin{equation}\label{eq:reference_tetrahedron}
\hat{\Omega} \df \big\{\vec{\xi} \in [-1,1]^3 : \xi_1 + \xi_2 + \xi_3 \leq -1  \big\},
\end{equation}
with the facets (i.e.\ faces of the tetrahedron) numbered as
\begin{equation}\label{eq:tet_facets}
\begin{alignedat}{2}
\hat{\Gamma}^{(1)} &\df \big\{ \vec{\xi} \in \hat{\Omega} : \xi_2 = -1 \big\}, \quad \hat{\Gamma}^{(2)} &&\df \big\{ \vec{\xi} \in \hat{\Omega} : \xi_1 + \xi_2 + \xi_3 = -1 \big\}, \\
\hat{\Gamma}^{(3)} &\df \big\{ \vec{\xi} \in \hat{\Omega} : \xi_1 = -1 \big\}, \quad \hat{\Gamma}^{(4)} &&\df \big\{ \vec{\xi} \in \hat{\Omega} : \xi_3 = -1 \big\}.
\end{alignedat}
\end{equation}
The collapsed coordinate transformation $\vec{\chi} : [-1,1]^3 \to \hat{\Omega}$ from the cube to the tetrahedron, as depicted in \cref{fig:mapping_3d}, is constructed from three successive applications of \eqref{eq:tri_mapping}, resulting in
\begin{equation}\label{eq:tet_mapping}
\vec{\chi}(\vec{\eta}) \df \mqty[\tfrac{1}{4}(1+\eta_1)(1-\eta_2)(1-\eta_2) - 1 ,\\ \tfrac{1}{2}(1+\eta_2)(1-\eta_3) - 1 ,\\ \eta_3].
\end{equation}
Similarly to the triangular case, we let $q_m \in \mathbb{N}_0$ denote the degree of the approximation in the $\eta_m$ coordinate and introduce $\nvolnodes \df (q_1+1)(q_2+1)(q_3+1)$ tensor-product volume quadrature nodes, which are ordered using the bijective mapping $\sigma : \{0:q_1\} \times \{0:q_2\} \times \{0:q_3\} \to \{1:\nvolnodes\}$. Taking $(a,b) = (0,0)$ in the $\eta_1$ and $\eta_2$ directions and $(a,b) = (1,0)$ in the $\eta_3$ direction, we obtain 
\begin{equation}\label{eq:tet_quadrature}
\vec{\xi}^{(\sigma(\vec{\alpha}))} \df \vec{\chi}\big(\eta_{q_1,\alpha_1}^{(0,0)},\eta_{q_2,\alpha_2}^{(0,0)}, \eta_{q_3,\alpha_3}^{(1,0)}\big), \quad \omega^{(\sigma(\vec{\alpha}))} \df \frac{(1-\eta_{q_2,\alpha_2}^{(0,0)})(1-\eta_{q_3,\alpha_3}^{(1,0)})}{8}\omega_{q_1,\alpha_1}^{(0,0)}\omega_{q_2,\alpha_2}^{(0,0)}\omega_{q_3,\alpha_3}^{(1,0)}.
\end{equation}
Defining a collapsed coordinate system $(\eta_{f1},\eta_{f2})$ on each face of the tetrahedron, we let $q_{f1},q_{f2} \in \mathbb{N}_0$ and introduce the bijective mapping $\sigma_{f} : \{0:q_{f1}\} \times \{0:q_{f2}\} \to \{1:\nfacnodes\}$, where $\nfacnodes \df (q_{f1} + 1)(q_{f2}+1)$ denotes the number of nodes on the facet $\hat{\Gamma}^{(\zeta)} \subset \partial\hat{\Omega}$. If a Jacobi quadrature rule with $(a,b) = (1,0)$ is used in the $\eta_{f2}$ coordinate, the facet quadrature nodes are given by
\begin{equation}\label{eq:tet_facet_quadrature_nodes}
\begin{alignedat}{4}
\vec{\xi}^{(1,\sigma_f(\vec{\alpha}))} &\df \vec{\chi}(\eta_{q_{f1},\alpha_1}^{(0,0)},-1, \eta_{q_{f2},\alpha_2}^{(1,0)}), \quad &&\vec{\xi}^{(2,\sigma_f(\vec{\alpha}))} &&\df \vec{\chi}(1,\eta_{q_{f1},\alpha_1}^{(0,0)}, \eta_{q_{f2},\alpha_2}^{(1,0)}), \\
\vec{\xi}^{(3,\sigma_f(\vec{\alpha}))} &\df \vec{\chi}(-1,\eta_{q_{f1},\alpha_1}^{(0,0)}, \eta_{q_{f2},\alpha_2}^{(1,0)}), \quad &&\vec{\xi}^{(4,\sigma_f(\vec{\alpha}))} &&\df \vec{\chi}(\eta_{q_{f1},\alpha_1}^{(0,0)}, \eta_{q_{f2},\alpha_2}^{(1,0)},-1),
\end{alignedat}
\end{equation}
and the corresponding weights are
\begin{equation}\label{eq:tet_facet_quadrature_weights}
\begin{alignedat}{4}
\omega^{(1,\sigma_f(\vec{\alpha}))} &\df \frac{1}{2}\omega_{q_{f1},\alpha_1}^{(0,0)}\omega_{q_{f2},\alpha_2}^{(1,0)}, \quad &&\omega^{(2,\sigma_f(\vec{\alpha}))} &&\df \frac{\sqrt{3}}{2}\omega_{q_{f1},\alpha_1}^{(0,0)}\omega_{q_{f2},\alpha_2}^{(1,0)} , \\
\omega^{(3,\sigma_f(\vec{\alpha}))} &\df \frac{1}{2}\omega_{q_{f1},\alpha_1}^{(0,0)}\omega_{q_{f2},\alpha_2}^{(1,0)}, \quad &&\omega^{(4,\sigma_f(\vec{\alpha}))} &&\df \frac{1}{2}\omega_{q_{f1},\alpha_1}^{(0,0)}\omega_{q_{f2},\alpha_2}^{(1,0)}.
\end{alignedat}
\end{equation}
The entries of the differentiation and interpolation/extrapolation matrices are then given by
\begin{subequations}\label{eq:d_3d}
\begin{align}
D_{\sigma(\vec{\alpha}) \sigma(\vec{\beta})}^{(1)} \df& \frac{4}{(1-\eta_{q_2,\alpha_2}^{(0,0)})(1-\eta_{q_3,\alpha_3}^{(1,0)})}\frac{\dd \ell_{q_1,\beta_1}^{(0,0)}}{\dd \eta_1} (\eta_{q_1,\alpha_1}^{(0,0)}) \delta_{\alpha_2\beta_2}\delta_{\alpha_3\beta_3},\label{eq:d1_3d} \\
D_{\sigma(\vec{\alpha}) \sigma(\vec{\beta})}^{(2)} \df& \frac{2(1+\eta_{q_1,\alpha_1}^{(0,0)})}{(1-\eta_{q_2,\alpha_2}^{(0,0)})(1-\eta_{q_3,\alpha_3}^{(1,0)})}  \frac{\dd \ell_{q_1,\beta_1}^{(0,0)}}{\dd \eta_1} (\eta_{q_1,\alpha_1}^{(0,0)}) \delta_{\alpha_2\beta_2} \delta_{\alpha_3\alpha_3} + \frac{2}{1-\eta_{q_3,\alpha_3}^{(1,0)}}\delta_{\alpha_1\beta_1}\frac{\dd \ell_{q_2,\beta_2}^{(0,0)}}{\dd \eta_2} (\eta_{q_2,\alpha_2}^{(0,0)})\delta_{\alpha_3\beta_3} \label{eq:d2_3d} , \\
D_{\sigma(\vec{\alpha}) \sigma(\vec{\beta})}^{(3)} \df&  \frac{2(1+\eta_{q_1,\alpha_1}^{(0,0)})}{(1-\eta_{q_2,\alpha_2}^{(0,0)})(1-\eta_{q_3,\alpha_3}^{(1,0)})}  \frac{\dd \ell_{q_1,\beta_1}^{(0,0)}}{\dd \eta_1} (\eta_{q_1,\alpha_1}^{(0,0)}) \delta_{\alpha_2\beta_2} \delta_{\alpha_3\alpha_3}\notag\\*  &+ \frac{1+\eta_{q_2,\alpha_2}^{(0,0)}}{1-\eta_{q_3,\alpha_3}^{(1,0)}}\delta_{\alpha_1\beta_1}\frac{\dd \ell_{q_2,\beta_2}^{(0,0)}}{\dd \eta_2} (\eta_{q_2,\alpha_2}^{(0,0)})\delta_{\alpha_3\beta_3} + \delta_{\alpha_1\beta_1} \delta_{\alpha_2\beta_2} \frac{\dd \ell_{q_3,\beta_3}^{(1,0)}}{\dd \eta_3} (\eta_{q_3,\alpha_3}^{(1,0)}), \label{eq:d3_3d} 
\end{align}
\end{subequations}
and
\begin{subequations}\label{eq:r_3d}
\begin{alignat}{2}
R_{\sigma_f(\vec{\alpha})\sigma(\vec{\beta})}^{(1)} &\df \ell_{q_1,\beta_1}^{(0,0)}(\eta_{q_{f1},\alpha_1}^{(0,0)}) \ell_{q_2,\beta_2}^{(0,0)}(-1)\ell_{q_3,\beta_3}^{(1,0)}(\eta_{q_{f2},\alpha_2}^{(1,0)}),  \\
R_{\sigma_f(\vec{\alpha})\sigma(\vec{\beta})}^{(2)}  &\df \ell_{q_1,\beta_1}^{(0,0)}(1)\ell_{q_2,\beta_2}^{(0,0)}(\eta_{q_{f1},\alpha_1}^{(0,0)})\ell_{q_3,\beta_3}^{(1,0)}(\eta_{q_{f2},\alpha_2}^{(1,0)}), \\
R_{\sigma_f(\vec{\alpha})\sigma(\vec{\beta})}^{(3)}  &\df \ell_{q_1,\beta_1}^{(0,0)}(-1)\ell_{q_2,\beta_2}^{(0,0)}(\eta_{q_{f1},\alpha_1}^{(0,0)})\ell_{q_3,\beta_3}^{(1,0)}(\eta_{q_{f2},\alpha_2}^{(1,0)}), \\
R_{\sigma_f(\vec{\alpha})\sigma(\vec{\beta})}^{(4)}  &\df \ell_{q_1,\beta_1}^{(0,0)}(\eta_{q_{f1},\alpha_1}^{(0,0)})\ell_{q_2,\beta_2}^{(0,0)}(\eta_{q_{f2},\alpha_2}^{(1,0)}) \ell_{q_3,\beta_3}^{(1,0)}(-1),
\end{alignat}
\end{subequations}  
respectively, where the diagonal matrices $\mat{W}$ and $\mat{B}^{(\zeta)}$ are defined similarly to the triangular case using the quadrature weights in \eqref{eq:tet_quadrature} and \eqref{eq:tet_facet_quadrature_weights}, respectively. In this paper, we use Gauss quadrature rules with respect to the Jacobi weights indicated in the superscripts, and we take $q_1$, $q_2$, $q_3$, $q_{f1}$, and $q_{f2}$ to be all equal to $q \in \mathbb{N}_0$. It then follows from \cite[Theorem 4.1]{montoya_sem_23} that such choices result in diagonal-norm SBP operators of degree $q$ for which the boundary matrices are given as in \eqref{eq:e_decomp}. Examples of volume quadrature nodes used to construct tensor-product SBP operators on the reference triangle and tetrahedron are shown in \cref{fig:tensor_nodes}.

\begin{figure}[t!]
\centering
\begin{subfigure}{0.50\textwidth}
\centering
\includegraphics[height=42mm]{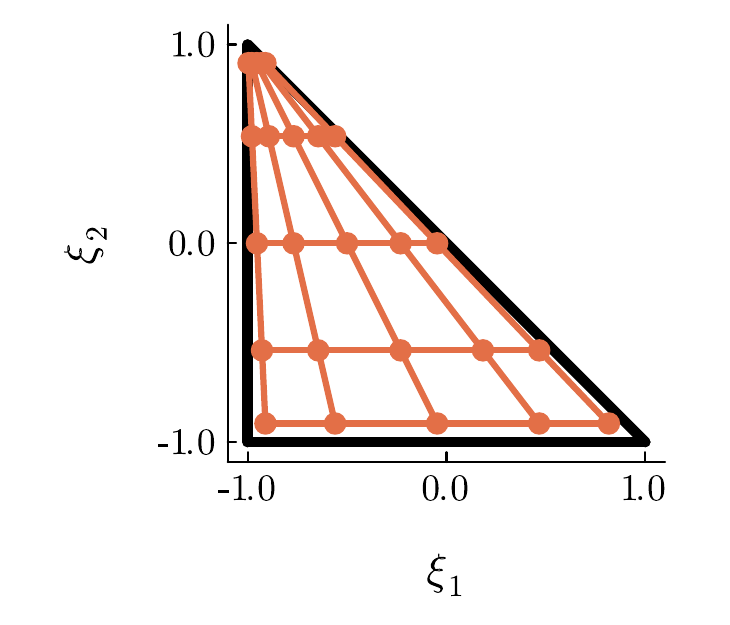}
\caption{Tensor-product quadrature nodes on the triangle}
\end{subfigure}~
\begin{subfigure}{0.48\textwidth}
\centering
\includegraphics[height=42mm]{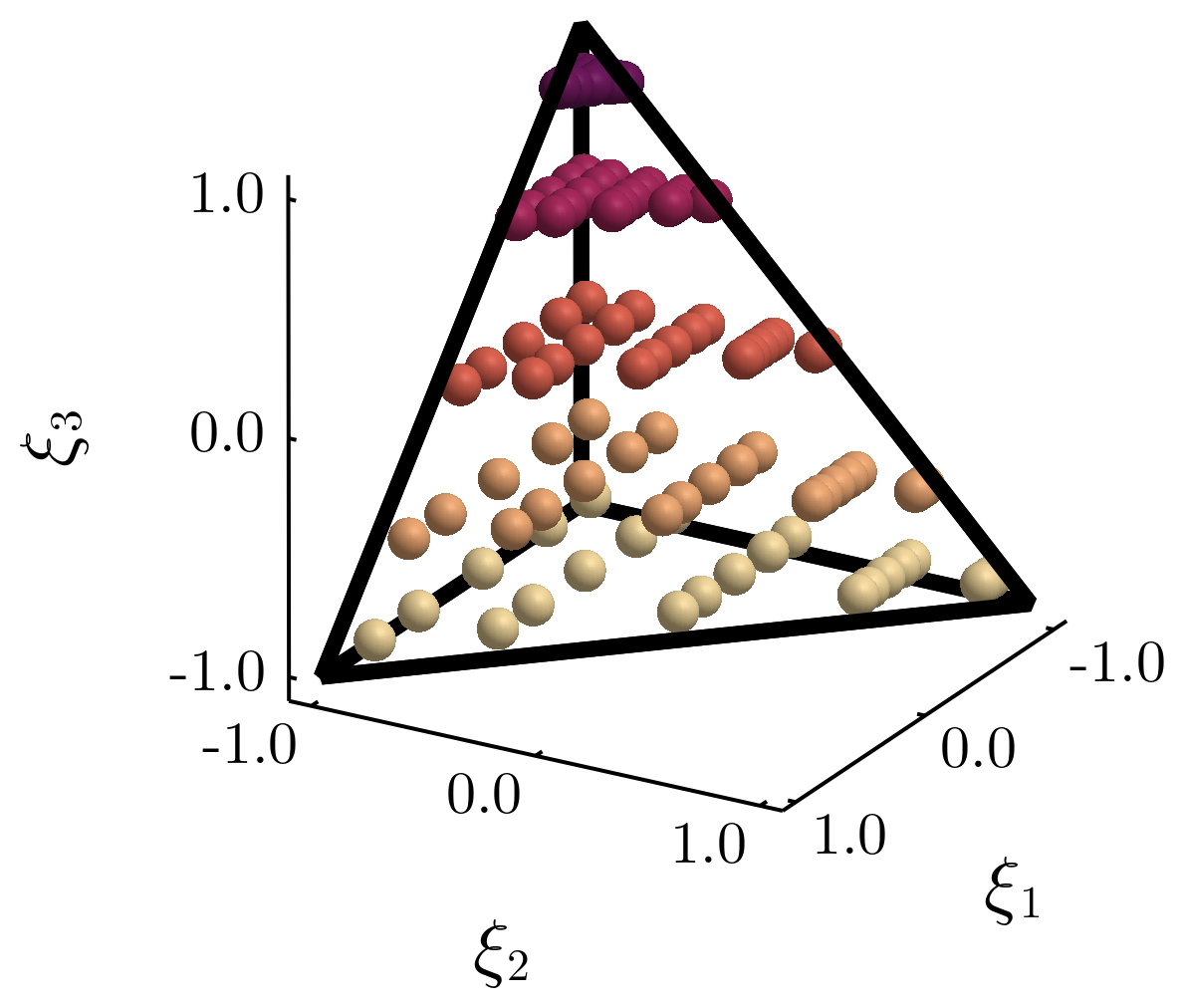}
\caption{Tensor-product quadrature nodes on the tetrahedron}
\end{subfigure}
\caption{Tensor-product volume quadrature nodes for SBP operators of degree $q = 4$ on the triangle and tetrahedron}\label{fig:tensor_nodes} 
\end{figure}
\begin{remark}
The choices of $\omega^{(\sigma(\vec{\alpha}))} \df \omega_{q_1,\alpha_1}^{(0,0)}\omega_{q_2,\alpha_2}^{(1,0)}/2$ and $\omega^{(\sigma(\vec{\alpha}))} \df \omega_{q_1,\alpha_1}^{(0,0)}\omega_{q_2,\alpha_2}^{(1,0)}\omega_{q_3,\alpha_3}^{(2,0)}/8$ for triangles and tetrahedra, respectively, are made in \cite{sherwin_karniadakis_triangular_sem_95} and \cite{sherwin_karniadakis_tetrahedra_hp_fem_96} in order to use the Jacobi weight to subsume the factor in \eqref{eq:tri_quadrature} or \eqref{eq:tet_quadrature} resulting from the Jacobian determinant of the collapsed coordinate transformation. However, such Jacobi weights lead to nodal derivative operators in \eqref{eq:d_2d} and \eqref{eq:d_3d} which do not, in general, satisfy the SBP property on the reference triangle and tetrahedron, respectively. The operators introduced in \cite{montoya_sem_23} and described in this section, which are constructed specifically to satisfy the SBP property on the reference simplex, therefore differ from existing SEM operators in collapsed coordinates.
\end{remark}
\section{Entropy-stable discontinuous spectral-element framework}\label{sec:methods}
In this section, we detail how the SBP operators described in the previous section can be used to construct entropy-stable DSEMs of any order on curvilinear unstructured grids. Due to the generality of the SBP approach, however, the formulations which we present can be used with \emph{any} set of diagonal-norm SBP operators on the reference triangle or tetrahedron for which the boundary matrices can be decomposed as in \eqref{eq:e_decomp}. While we introduce the proposed schemes within the context of a generalized framework in order to enable the unified implementation and analysis of DSEMs using SBP operators and to facilitate the comparison of the proposed methods to existing schemes, efficient algorithms require consideration of the particular properties of the operators which constitute a discretization, a topic which we will address in \cref{sec:implementation}.

\subsection{Mesh and coordinate transformation}\label{sec:meshmap}
The first step in constructing a DSEM is to subdivide the spatial domain $\Omega$ into a \emph{mesh} or \emph{grid}, which consists of a collection $\{\Omega^{(\kappa)}\}_{\kappa \in \{1:\nelem\}}$ of $\nelem \in \mathbb{N}$ closed, bounded, and connected elements $\Omega^{(\kappa)} \subset \Omega$ with non-empty, non-overlapping interiors, satisfying
\begin{equation}\label{eq:non_overlapping}
\bigcup_{\kappa=1}^{\nelem} \Omega^{(\kappa)} = \bar{\Omega}.
\end{equation}
We denote the characteristic element size for such a mesh as $h \in \mathbb{R}^+$. In this work, we assume that the mesh is time invariant and that each element is the image of the reference triangle or tetrahedron under a polynomial mapping $\fvec{x}^{(\kappa)} \in [\mathbb{P}_{p_g}(\hat{\Omega})]^d$ of total degree $p_g \in \mathbb{N}$. Such a mapping is given in terms of a multivariate Lagrange basis $\{\ell_{p_g}^{(i)}\}_{i\in \{1:\ngeom\}}$ associated with a set of nodes $\{\vec{\xi}_{p_g}^{(i)}\}_{i\in\{1:\ngeom\}} \subset \hat{\Omega}$ as
\begin{equation}\label{eq:poly_mapping}
\fvec{x}^{(\kappa)}(\vec{\xi}) \df \sum_{i=1}^{\ngeom} \vec{x}_{p_g}^{(\kappa,i)} \ell_{p_g}^{(i)}(\vec{\xi}),
\end{equation}
where $\{\vec{x}_{p_g}^{(\kappa,i)}\}_{i\in\{1:\ngeom\}}$ are the prescribed physical positions of the mapping nodes. To ensure a watertight mesh, we assume that the mapping nodes contain a subset of nodes on each facet $\hat{\Gamma}^{(\zeta)} \subset \partial\hat{\Omega}$ which are unisolvent for the corresponding trace space $\mathbb{P}_{p_g}(\hat{\Gamma}^{(\zeta)}) \df \{\fn{v}\lvert_{\hat{\Gamma}^{(\zeta)}} : \fn{v} \in \mathbb{P}_{p_g}(\hat{\Omega})\}$, thus resulting in continuity at element interfaces. Denoting the Jacobian of the transformation by $\nabla_{\vec{\xi}}\fvec{x}^{(\kappa)}(\vec{\xi}) \in \mathbb{R}^{d\times d}$ and defining $\fn{J}^{(\kappa)}(\vec{\xi}) \df \operatorname{det}(\nabla_{\vec{\xi}}\fvec{x}^{(\kappa)}(\vec{\xi}))$, we assume that the mapping is bijective and orientation preserving, satisfying
\begin{equation}\label{eq:jacobian_positive}
\fn{J}^{(\kappa)}(\vec{\xi}) > 0, \quad \forall \, \vec{\xi} \in \hat{\Omega}.
\end{equation}
The adjugate of the Jacobian matrix is given by $\vc{\vec{G}}^{(\kappa)}(\vec{\xi}) \df \operatorname{det}(\nabla_{\vec{\xi}}\fvec{x}^{(\kappa)}(\vec{\xi}))(\nabla_{\vec{\xi}}\fvec{x}^{(\kappa)}(\vec{\xi}))^{-1}$ with entries (often referred to as the \emph{metric terms} in the literature) satisfying the \emph{metric identities},
\begin{equation}\label{eq:metric_identities}
\sum_{l=1}^d \frac{\partial\fn{G}_{lm}^{(\kappa)} (\vec{\xi})}{\partial\xi_l} = 0, \quad \forall \, \vec{\xi} \in \hat{\Omega}, \quad \forall\, m \in \{1:d\}.
\end{equation}
Using the metric identities, we can express \eqref{eq:pde} in conservation form on the reference element as
\begin{equation}\label{eq:pde_ref}
\fn{J}^{(\kappa)}(\vec{\xi})\frac{\partial\vc{\fn{u}}(\fvec{x}^{(\kappa)}(\vec{\xi}),t)}{\partial t} + \sum_{l=1}^d \frac{\partial}{\partial \xi_l}\Bigg(\sum_{m=1}^d G_{lm}^{(\kappa)}(\vec{\xi})\vc{\fn{f}}_m(\vc{\fn{u}}(\fvec{x}^{(\kappa)}(\vec{\xi}),t))  \Bigg) = \vc{0}.
\end{equation}
Further details regarding the formulation of conservation laws in curvilinear coordinates are provided, for example, in Pulliam and Zingg \cite[Section 4.2]{pulliam14} and Kopriva \cite[Section 6.2]{kopriva09}.

\subsection{Approximation of the metric terms}\label{sec:metrics}
Considering a mapping using polynomials of degree $p_g$ as in \eqref{eq:poly_mapping}, the metric terms are polynomials of degree $p_g - 1$ in two dimensions and degree $2p_g-2$ in three dimensions. Since operations such as differentiation using SBP operators are exact for polynomials of at most degree $q$, we cannot expect that a discrete analogue of the metric identities in \eqref{eq:metric_identities} will hold unless $p_g \leq q + 1$ in two dimensions or $p_g \leq \lfloor q/2 \rfloor + 1$ in three dimensions. To circumvent this requirement for a subparametric mapping in three-dimensional case, we use an adaptation by Chan and Wilcox \cite[Section 5]{chan_wilcox_entropystable_curvilinear_19} of Kopriva's approximation of the metric terms in \emph{conservative curl form} \cite[Eq.\ (36)]{kopriva_metric_identities_discontinuous_sem_curved_06}, which is itself based on techniques introduced by Thomas and Lombard \cite{thomas_lombard_gcl_79}. To obtain such an approximation, we introduce a Lagrange basis $\{\ell_{q+1}^{(i)}\}_{i\in \{1:N_{q+1}^*\}}$ for $\mathbb{P}_{q+1}(\hat{\Omega})$ associated with a set of nodes $\{\vec{\xi}_{q+1}^{(i)}\}_{i\in\{1:N_{q+1}^*\}} \subset \hat{\Omega}$ and construct polynomial interpolants of the form
\begin{subequations}\label{eq:metric_interpolants}
\begin{align}
\vec{r}_{q+1}^{(\kappa,1)}(\vec{\xi}) &\df\sum_{i=1}^{N_{q+1}^*} \fn{x}_3^{(\kappa)}(\vec{\xi}_{q+1}^{(i)}) \nabla_{\vec{\xi}}\fn{x}_2^{(\kappa)}(\vec{\xi}_{q+1}^{(i)}) \ell_{q+1}^{(i)}(\vec{\xi}),\\
\vec{r}_{q+1}^{(\kappa,2)}(\vec{\xi}) &\df \sum_{i=1}^{N_{q+1}^*} \fn{x}_3^{(\kappa)}(\vec{\xi}_{q+1}^{(i)}) \nabla_{\vec{\xi}}\fn{x}_1^{(\kappa)}(\vec{\xi}_{q+1}^{(i)}) \ell_{q+1}^{(i)}(\vec{\xi}),\\
\vec{r}_{q+1}^{(\kappa,3)}(\vec{\xi}) &\df \sum_{i=1}^{N_{q+1}^*} \fn{x}_1^{(\kappa)}(\vec{\xi}_{q+1}^{(i)}) \nabla_{\vec{\xi}}\fn{x}_2^{(\kappa)}(\vec{\xi}_{q+1}^{(i)}) \ell_{q+1}^{(i)}(\vec{\xi}).
\end{align}
\end{subequations}
The above functions $\vec{r}_{q+1}^{(\kappa,m)} \in [\mathbb{P}_{q+1}(\hat{\Omega})]^3$ are used to define the matrix of approximate metric terms
\begin{equation}\label{eq:conservative_curl}
\vc{\vec{G}}^{(\kappa)}(\vec{\xi}) \df \big[-\nabla_{\vec{\xi}} \times \vec{r}_{q+1}^{(\kappa,1)}(\vec{\xi}),\ \nabla_{\vec{\xi}} \times \vec{r}_{q+1}^{(\kappa,2)}(\vec{\xi}),\ \nabla_{\vec{\xi}} \times \vec{r}_{q+1}^{(\kappa,3)}(\vec{\xi}) \big]^\T,
\end{equation}
which has entries of degree $q$ and satisfies \eqref{eq:metric_identities} by construction. Using the exact metric terms in two dimensions and the approximation \eqref{eq:conservative_curl} in three dimensions, we can compute the outward unit normal vector to the facet $\Gamma^{(\kappa,\zeta)} \subset \partial\Omega^{(\kappa)}$ which is the image of $\hat{\Gamma}^{(\zeta)}$ under the mapping $\fvec{x}^{(\kappa)}$ as 
\begin{equation}\label{eq:nanson}
\vec{n}^{(\kappa,\zeta)}(\fvec{x}^{(\kappa)}(\vec{\xi})) \df \frac{\vc{\vec{G}}^{(\kappa)}(\vec{\xi})^\T\hat{\vec{n}}^{(\zeta)}}{\fn{J}^{(\kappa,\zeta)}(\vec{\xi})}, \quad \text{where} \quad \fn{J}^{(\kappa,\zeta)}(\vec{\xi}) \df \lVert\vc{\vec{G}}^{(\kappa)}(\vec{\xi})^\T\hat{\vec{n}}^{(\zeta)}\rVert. 
\end{equation}
\begin{remark}
In general, the normals computed as in \eqref{eq:nanson} are not exact when the metric terms are computed approximately using \eqref{eq:conservative_curl}. However, if the analytically defined mesh is watertight and the nodes used for the interpolants in \eqref{eq:metric_interpolants} define a continuous approximation space, the approximate normals remain equal and opposite at element interfaces (see, for example, \cite[Theorem 5]{chan_wilcox_entropystable_curvilinear_19}).
\end{remark}

\subsection{Modal polynomial expansions on triangles and tetrahedra}\label{sec:modal_expansion}
The focus of this paper is on \emph{modal} formulations, in which the degrees of freedom for the semi-discrete approximation to the $e^{\mathrm{th}}$ solution variable on the $k^{\mathrm{th}}$ element are taken to be the expansion coefficients with respect to a polynomial basis $\{\phi^{(i)}\}_{i \in \{1:\npoly\}}$ of degree $p \leq q$ contained within the vector $\vc{\tilde{u}}^{(h,\kappa,e)}(t) \in \mathbb{R}^{\npoly}$. Such coefficients define an approximate solution $\vc{\fn{u}}^{(h,\kappa)}(\cdot,t) \in [\mathbb{P}_p(\hat{\Omega})]^{\ncons}$ as
\begin{equation}\label{eq:poly_approx}
\fn{u}_e(\fvec{x}^{(\kappa)}(\vec{\xi}),t) \approx \fn{u}_e^{(h,\kappa)}(\vec{\xi},t) \df \sum_{i=1}^{\npoly} \tilde{u}_i^{(h,\kappa,e)}(t) \phi^{(i)}(\vec{\xi}),
\end{equation}
resulting in a global approximation given piecewise by $ \vc{\fn{u}}(\vec{x},t) \approx \vc{\fn{u}}^h(\vec{x},t) \df \vc{\fn{u}}^{(h,\kappa)}((\fvec{x}^{(\kappa)})^{-1}(\vec{x}),t)$ for $\vec{x} \in \Omega^{(\kappa)}$. Evaluating \eqref{eq:poly_approx} at each volume quadrature node, the vector $\vc{u}^{(h,\kappa, e)}(t) \in \mathbb{R}^{\nvolnodes}$ of nodal solution values can be expressed in terms of the \emph{generalized Vandermonde matrix} $\mat{V} \in \mathbb{R}^{\nvolnodes\times\npoly}$ as
\begin{equation}\label{eq:modal_to_nodal}
\vc{u}^{(h,\kappa, e)}(t) = \mat{V}\vc{\tilde{u}}^{(h,\kappa,e)}(t), \quad \text{where} \quad V_{ij} \df \phi^{(j)}(\vec{\xi}^{(i)}).
\end{equation}
In order to construct polynomial bases on the triangle and tetrahedron for which operations such as \eqref{eq:modal_to_nodal} are amenable to sum factorization when tensor-product quadrature rules in collapsed coordinates are used, we follow \cite[Section 3.2]{karniadakis_sherwin_spectral_hp_element} and define the one-dimensional \emph{principal functions}
\begin{equation}\label{eq:principal_functions}
\begin{gathered}
\psi_1^{(\alpha_1)}(\eta_1) \df \sqrt{2}\fn{P}_{\alpha_1}^{(0,0)}(\eta_1), \quad
\psi_2^{(\alpha_1,\alpha_2)}(\eta_2) \df (1-\eta_2)^{\alpha_1}\fn{P}_{\alpha_2}^{(2\alpha_1+1,0)}(\eta_2), \\
\psi_3^{(\alpha_1,\alpha_2,\alpha_3)}(\eta_3) \df 2(1-\eta_3)^{\alpha_1 + \alpha_2}\fn{P}_{\alpha_3}^{(2\alpha_1+2\alpha_2+2,0)}(\eta_3).
\end{gathered}
\end{equation}
The normalized Proriol--Koornwinder--Dubiner (PKD) orthogonal polynomials \cite{proriol_polynomials_57,koornwinder_orthogonal_polynomials_75,dubiner_spectral_triangle_91} are then given with respect to the collapsed coordinate system on the reference triangle as
\begin{equation}\label{eq:modal_basis_2d}
\phi^{(\pi(\vec{\alpha}))}(\vec{\chi}(\vec{\eta})) \df \psi_1^{(\alpha_1)}(\eta_1) \psi_2^{(\alpha_1,\alpha_2)}(\eta_2), 
\end{equation}
and on the reference tetrahedron as
\begin{equation}\label{eq:modal_basis_3d}
\phi^{(\pi(\vec{\alpha}))}(\vec{\chi}(\vec{\eta})) \df \psi_1^{(\alpha_1)}(\eta_1) \psi_2^{(\alpha_1,\alpha_2)}(\eta_2)\psi_3^{(\alpha_1,\alpha_2,\alpha_3)}(\eta_3), 
\end{equation}where we order the multi-indices $\vec{\alpha} \in \mathcal{N}(p)$ using the bijection $\pi : \mathcal{N}(p) \to \{1:\npoly\}$. As described, for example, in \cite[Sections 4.1.6.1 and 4.1.6.2]{karniadakis_sherwin_spectral_hp_element}, the ``warped'' tensor-product structure of the PKD basis allows for the generalized Vandermonde matrix in \eqref{eq:modal_to_nodal} or its transpose to be applied in $\mathcal{O}(p^{d+1})$ operations through sum factorization when tensor-product quadrature rules are used with $\mathcal{O}(p)$ nodes in each direction. Moreover, since the principal functions in \eqref{eq:principal_functions} have been scaled to obtain an orthonormal basis, the reference mass matrix $\mat{M} \df \mat{V}^\T\mat{W}\mat{V}$ is the identity matrix if the quadrature rule in \eqref{eq:sbp_quadrature} is of degree $\tau \geq 2p$.
\begin{remark}
A nodal formulation is recovered by taking $\mat{V}$ to be the identity matrix and directly evolving the nodal solution vector $\vc{u}^{(h,\kappa,e)}(t)$. Such collocation-based approaches using tensor-product LGL \cite{carpenter_entropystable_collocation_14,gassner_winters_kopriva_splitform_nodaldg_sbp_16} or LG \cite{chan_delrey_carpenter_gauss_collocation_19} quadrature rules are popular for quadrilateral and hexahedral elements due to the numerical solution being available at the volume quadrature nodes (and, in the case of LGL quadrature, the facet quadrature nodes) without needing to perform the matrix operation in \eqref{eq:modal_to_nodal}. When employing collapsed coordinate systems on triangular and tetrahedral elements, however, we are primarily interested in modal formulations, since the use of a total-degree polynomial expansion as in \eqref{eq:poly_approx} avoids the severe explicit time step restriction which would otherwise result from the use of a collocated formulation in collapsed coordinates (i.e.\ due to the clustering of resolution near the singular vertices). This was a primary motivation for Dubiner's introduction of such an approach in \cite{dubiner_spectral_triangle_91}, and is further motivated in the present context by numerical experiments in \cite[Section 7.4]{montoya_sem_23}, in which the authors found the spectral radii for energy-stable modal discretizations of the linear advection equation using the tensor-product operators described in \cref{sec:tensor_sbp} to be nearly identical to those obtained using multidimensional SBP operators based on symmetric quadrature rules.
\end{remark}
\subsection{Discontinuous spectral-element methods using summation-by-parts operators}
Before introducing the proposed entropy-stable discretizations, we will first present the fundamental concepts and notation relating to the formulation of DSEMs using SBP operators within the context of standard weak-form DG methods as well as skew-symmetric formulations on curvilinear meshes. We will begin by deriving a weak-form DG approximation of \eqref{eq:cons_law} either by integrating by parts against a smooth test function on the physical element and performing a change of variables within each of the resulting integrals, or by integrating the transformed conservation law in \eqref{eq:pde_ref} by parts against a smooth test function on the reference element. In either case, we approximate each solution variable as in \eqref{eq:poly_approx} and consider test functions belonging to the same space, resulting in 
\begin{equation}\label{eq:weak_dg}
\begin{multlined}
\int_{\hat{\Omega}} \fn{v}(\vec{\xi}) \fn{J}^{(\kappa)}(\vec{\xi}) \frac{\partial\vc{\fn{u}}^{(h,\kappa)}(\vec{\xi},t)}{\partial t} \, \dd \vec{\xi} = \sum_{l=1}^d\int_{\hat{\Omega}} \frac{\partial \fn{v}(\vec{\xi})}{\partial \xi_l} \sum_{m=1}^d G_{lm}^{(\kappa)}(\vec{\xi})\vc{\fn{f}}_m(\vc{\fn{u}}^{(h,\kappa)}(\vec{\xi},t)) \, \dd \vec{\xi} \\ - \sum_{\zeta=1}^{\nfac} \int_{\hat{\Gamma}^{(\zeta)}} \fn{v}(\vec{\xi})\fn{J}^{(\kappa,\zeta)}(\vec{\xi})\vc{\fn{f}}^*\big(\vc{\fn{u}}^{(h,\kappa)}(\vec{\xi},t),\,\vc{\fn{u}}^{(h,\kappa,+)}(\vec{\xi},t),\, \vec{n}^{(\kappa,\zeta)}(\fvec{x}^{(\kappa)}(\vec{\xi}))\big)\, \dd \hat{s}, \quad \forall \, \fn{v} \in \mathbb{P}_p(\hat{\Omega}),
\end{multlined}
\end{equation}
where a \emph{numerical flux function} $\vc{\fn{f}}^* : \Upsilon \times \Upsilon \times \mathbb{S}^{d-1} \to \mathbb{R}^{\ncons}$ has been used to resolve the discontinuity in the global numerical solution at each facet, on which we denote the exterior solution as $\vc{\fn{u}}^{(h,\kappa,+)}(\vec{\xi},t) \in \Upsilon$. In order to obtain an algebraic formulation of \eqref{eq:weak_dg}, we begin by forming the diagonal matrices $\mat{J}^{(\kappa)}, \mat{G}^{(\kappa,l,m)} \in \mathbb{R}^{\nvolnodes\times\nvolnodes}$ and $\mat{J}^{(\kappa,\zeta)}, \mat{N}^{(\kappa,\zeta,m)} \in \mathbb{R}^{\nfacnodes \times \nfacnodes}$ with entries given by
\begin{equation}\label{eq:geometric_factors}
\begin{alignedat}{2}
J_{ij}^{(\kappa)} &\df \fn{J}^{(\kappa)}(\vec{\xi}^{(i)})\delta_{ij}, \quad &G_{ij}^{(\kappa,l,m)} &\df G_{lm}^{(\kappa)}(\vec{\xi}^{(i)}) \delta_{ij},\\
J_{ij}^{(\kappa,\zeta)} &\df \fn{J}^{(\kappa,\zeta)}(\vec{\xi}^{(\zeta,i)})\delta_{ij}, \quad & N_{ij}^{(\kappa,\zeta,m)} &\df n_m^{(\kappa,\zeta)}(\fvec{x}^{(\kappa)}(\vec{\xi}^{(\zeta,i)}))\delta_{ij}.
\end{alignedat}
\end{equation}
Defining the physical mass matrix as $\mat{M}^{(\kappa)} \df \mat{V}^\T\mat{W}\mat{J}^{(\kappa)}\mat{V}$, we then obtain a semi-discrete formulation governing the evolution of the vector of modal coefficients for the expansion in \eqref{eq:poly_approx} as
\begin{equation}\label{eq:modal}
\mat{M}^{(\kappa)}\frac{\dd \vc{\tilde{u}}^{(h,\kappa,e)}(t)}{\dd t} = \mat{V}^\T\vc{r}^{(h,\kappa,e)}(t), 
\end{equation}
where the form of $\vc{r}^{(h,\kappa,e)}(t) \in \mathbb{R}^{\nvolnodes}$ determines the particular DSEM recovered within the present framework. Evaluating the nodal solution as $\vc{u}^{(h,\kappa,e)}(t) \df \mat{V}\vc{\tilde{u}}^{(h,\kappa,e)}(t)$, we gather the solution variables at the volume and facet quadrature nodes as
\begin{equation}\label{eq:nodal_conservative variables}
\vc{u}_i^{(h,\kappa)}(t) \df \mqty[u_i^{(h,\kappa,1)}(t) \\ \vdots\\ u_i^{(h,\kappa,\ncons)}(t)], \quad \vc{u}_i^{(h,\kappa,\zeta)}(t) \df \mqty[[\mat{R}^{(\zeta)}\vc{u}^{(h,\kappa,1)}(t)]_i \\ \vdots \\ [\mat{R}^{(\zeta)}\vc{u}^{(h,\kappa,\ncons)}(t)]_i],
\end{equation}
and form the vectors $\vc{f}^{(\kappa,m,e)}(t) \in \mathbb{R}^{\nvolnodes}$ and $\vc{f}^{(*,\kappa,\zeta,e)} \in \mathbb{R}^{\nfacnodes}$ with entries given by
\begin{subequations}
\begin{align}
f_i^{(\kappa,m,e)}(t) &\df \fn{F}_{me}(\vc{u}_i^{(h,\kappa,e)}(t)), \\
f_i^{(*,\kappa,\zeta,e)}(t) &\df \fn{f}_e^*(\vc{u}_i^{(h,\kappa,\zeta)}(t),\, \vc{u}_i^{(+,\kappa,\zeta,e)}(t),\, \vec{n}^{(\kappa,\zeta)}(\fvec{x}^{(\kappa)}(\vec{\xi}^{(\zeta,i)}))),
\end{align}
\end{subequations}
where $\vc{u}_i^{(+,\kappa,\zeta,e)}(t) \in \Upsilon$ denotes the exterior solution state. 
With such definitions in place, the right-hand side of \eqref{eq:weak_dg} can be discretized directly using nodal SBP operators as
\begin{equation}\label{eq:rhs_cons}
\begin{multlined}
\vc{r}^{(h,\kappa,e)}(t) \df  \sum_{l=1}^d \big(\mat{Q}^{(l)}\big)^\T \sum_{m=1}^d \mat{G}^{(\kappa,l,m)}\vc{f}^{(\kappa,m,e)}(t) -\sum_{\zeta=1}^{\nfac}\big(\mat{R}^{(\zeta)}\big)^\T \mat{B}^{(\zeta)}\mat{J}^{(\kappa,\zeta)}\vc{f}^{(*,\kappa,\zeta,e)}(t),
\end{multlined}
\end{equation}
where it follows from a similar analysis to \cite[Section 4]{montoya_unifying_21} that the resulting scheme is conservative for an appropriate choice of numerical flux as well as energy stable for linear, constant-coefficient problems on meshes for which the mapping from reference to physical space is affine.\footnote{While the SBP operators in \cite{montoya_unifying_21} act on modal coefficients rather than nodal values, the analysis is essentially the same, relying on the SBP property and the exact differentiation and interpolation/extrapolation of constants.}  \par
To ensure energy stability for linear problems on curvilinear meshes, a property which cannot be guaranteed for discretizations in the form of \eqref{eq:rhs_cons}, the schemes proposed by the authors in \cite{montoya_tensor_product_22} and \cite{montoya_sem_23} employ a \emph{skew-symmetric} weak formulation adapted from the work of Gassner and Kopriva \cite{kopriva_gassner_dgsem_variable_coeff_14} and Del Rey Fern\'andez \etal \cite{delrey_mdsbp_sat_18,delrey_extension_of_dense_gsbp_tensor_curvilinear_19}, which is given by
\begin{equation}\label{eq:rhs_skew}
\begin{aligned}
\vc{r}^{(h,\kappa,e)}(t) &\df \frac{1}{2} \sum_{l=1}^d \bigg(\big(\mat{Q}^{(l)}\big)^\T\sum_{m=1}^d \mat{G}^{(\kappa,l,m)}\vc{f}^{(\kappa,m,e)}(t) -  \sum_{m=1}^d \mat{G}^{(\kappa,l,m)}\mat{Q}^{(l)}\vc{f}^{(\kappa,m,e)}(t)\bigg)  \\  &-\sum_{\zeta=1}^{\nfac}\big(\mat{R}^{(\zeta)}\big)^\T \mat{B}^{(\zeta)}\mat{J}^{(\kappa,\zeta)}\bigg(\vc{f}^{(*,\kappa,\zeta,e)}(t) -\frac{1}{2}\sum_{m=1}^d \mat{N}^{(\kappa,\zeta,m)}\mat{R}^{(\zeta)}\vc{f}^{(\kappa,m,e)}(t) \bigg).
\end{aligned}
\end{equation}
In \cref{sec:entropy_stable}, we will construct modifed formulations, which, unlike \eqref{eq:rhs_cons} or \eqref{eq:rhs_skew}, are provably \emph{entropy stable} for nonlinear systems of conservation laws endowed with suitable entropy functions and entropy fluxes, satisfying semi-discrete bounds analogous to \eqref{eq:entropy_balance}. However, we must first address the issue of inverting the mass matrix in \eqref{eq:modal} in order to obtain the time derivative.

\subsection{Weight-adjusted approximation of the inverse mass matrix}\label{sec:weight_adjusted}
Even when using orthonormal bases on the reference element such as those presented in \cref{sec:modal_expansion}, the mass matrix $\mat{M}^{(\kappa)}$ appearing on the left-hand side of \eqref{eq:modal} is dense when the mapping from the reference element to the physical element is not affine, and its inverse lacks a tensor-product structure amenable to sum factorization. Obtaining the time derivative for such a scheme in the context of explicit temporal integration thus requires either the storage and application of a non-tensorial factorization or inverse, or, otherwise, the solution of a dense $\npoly$ by $\npoly$ linear system, for each curved element in the mesh. To obtain a fully explicit formulation for the time derivative in \eqref{eq:modal}, we use a \emph{weight-adjusted} approximation given by
\begin{equation}\label{eq:weight_adjusted_inverse}
\big(\mat{M}^{(\kappa)}\big)^{-1} \approx \mat{M}^{-1}\mat{V}^\T\mat{W}\big(\mat{J}^{(\kappa)}\big)^{-1}\mat{V}  \mat{M}^{-1} \eqdf \big(\mat{\tilde{M}}^{(\kappa)}\big)^{-1}.
\end{equation}
The above approximation was initially proposed by Chan \etal \cite{chan_weight_adjusted_dg_curvilinear_17} for the purpose of reducing storage requirements for curved elements. However, it was also demonstrated in \cite{montoya_sem_23} that such an approach preserves the tensor-product operator structure which would otherwise be lost by taking the inverse of the mass matrix. The time derivative can then be computed explicitly as
\begin{equation}\label{eq:dudt_modal}
\frac{\dd \vc{\tilde{u}}^{(h,\kappa,e)}(t)}{\dd t} = \big(\mat{\tilde{M}}^{(\kappa)}\big)^{-1}\mat{V}^\T\vc{r}^{(h,\kappa,e)}(t),
\end{equation}
where we can exploit sum factorization in the application of the operators $\mat{V}$ and $\mat{V}^\T$ in \eqref{eq:weight_adjusted_inverse}, as discussed in \cref{sec:modal_expansion}. While the formulation in \eqref{eq:dudt_modal} is not, in general, discretely conservative with respect to the quadrature rule defined by the diagonal entries of $\mat{W}$ as in \eqref{eq:sbp_quadrature}, we restore conservation using a technique proposed by Chan and Wilcox \cite[Lemma 2]{chan_wilcox_entropystable_curvilinear_19}. In the context of a mapping in the form of \eqref{eq:poly_mapping}, such a modification involves approximating the determinant of the mapping Jacobian, which is of degree $2p_g - 2$ in two dimensions and $3p_g - 3$ in three dimensions, by an interpolant of degree $p_g$ given in terms of the nodal basis used in \eqref{eq:poly_mapping} as
\begin{equation}\label{eq:jacobian_approx}
\fn{J}^{(\kappa)}(\vec{\xi}) \df \sum_{i=1}^{\ngeom} \operatorname{det}(\nabla_{\vec{\xi}}\fvec{x}^{(\kappa)}(\vec{\xi}_{p_g}^{(i)})) \ell_{p_g}^{(i)}(\vec{\xi}),
\end{equation}
and using such an approximation to define $\mat{J}^{(\kappa)}$ in \eqref{eq:weight_adjusted_inverse}, noting that such a modification does not affect the (approximate) metric terms $\vc{\vec{G}}^{(\kappa)}(\vec{\xi})$ used to compute the right-hand side of \eqref{eq:modal}.

\subsection{Entropy-stable flux-differencing formulation}\label{sec:entropy_stable}
In this section, we will develop a formulation of the entropy-stable modal approach introduced by Chan \cite{chan_discretely_entropy_conservative_dg_sbp_18} which lends itself to an efficient implementation when used with the tensor-product operators on triangles and tetrahedra described in \cref{sec:tensor_sbp}. We begin by considering a fundamental issue in the development of entropy-stable modal DSEMs, which is the fact that the entropy variables may not lie within approximation space in which the solution is sought. As such, they cannot be taken as test functions in the discrete variational formulation, a critical step in establishing entropy stability for such schemes. To resolve this, entropy-stable modal formulations employ an \emph{entropy projection}, referring to the approximation of the entropy variables as
\begin{equation}
\mathcal{W}_e(\vc{\fn{u}}^{(h,\kappa)}(\vec{\xi},t)) \approx \sum_{i=1}^{\npoly} \tilde{w}_i^{(h,\kappa,e)}(t) \phi^{(i)}(\vec{\xi}),
\end{equation}
where, as in \cite[Eq.\ (31)]{chan_wilcox_entropystable_curvilinear_19}, we obtain the modal coefficients through a weight-adjusted projection as
\begin{equation}\label{eq:entropy_expansion}
\vc{\tilde{w}}^{(h,\kappa,e)}(t) \df \big(\mat{\tilde{M}}^{(\kappa)}\big)^{-1}\mat{V}^\T\mat{W}\mat{J}^{(\kappa)}\big[\mathcal{W}_e(\vc{u}_1^{(h,\kappa)}(t)),\ldots,\mathcal{W}_e(\vc{u}_{\nvolnodes}^{(h,\kappa)}(t))\big]^\T.
\end{equation}
Using the generalized Vandermonde matrix to evaluate the projected entropy variables at the volume quadrature nodes as $\vc{w}^{(h,\kappa,e)}(t) \df \mat{V}\vc{\tilde{w}}^{(h,\kappa,e)}(t)$, we then define
\begin{equation}\label{eq:entropy_variables}
\vc{w}_i^{(h,\kappa)}(t) \df \mqty[w_i^{(h,\kappa,1)}(t) \\ \vdots\\ w_i^{(h,\kappa,\ncons)}(t)], \quad \vc{w}_i^{(h,\kappa,\zeta)}(t) \df \mqty[[\mat{R}^{(\zeta)}\vc{w}^{(h,\kappa,1)}(t)]_i \\ \vdots \\ [\mat{R}^{(\zeta)}\vc{w}^{(h,\kappa,\ncons)}(t)]_i].
\end{equation}
Next, we present the following definition of an \emph{entropy-conservative two-point flux}, which is an essential component of the entropy-stable methods described in this work.
\begin{definition}\label{def:ec_flux}
A continuously differentiable bivariate function $\vc{\fn{F}}_m^\# : \Upsilon \times \Upsilon \to \mathbb{R}^{\ncons}$ is an \emph{entropy-conservative two-point flux} if it is symmetric with respect to its two arguments, consistent with the PDE in \eqref{eq:pde} such that $\vc{\fn{F}}_m^\#(\vc{U},\vc{U}) = \vc{\fn{f}}_m(\vc{U})$ for all $\vc{U} \in \Upsilon$, and satisfies
\begin{equation}\label{eq:tadmor}
\big(\vc{\mathcal{W}}(\vc{U}^+) - \vc{\mathcal{W}}(\vc{U}^-)\big)^\T \vc{\fn{F}}_m^\#(\vc{U}^-,\vc{U}^+) = \varPsi_m(\vc{\mathcal{W}}(\vc{U}^+)) - \varPsi_m(\vc{\mathcal{W}}(\vc{U}^-)), \quad \forall \, \vc{U}^-,  \vc{U}^+ \in \Upsilon.
\end{equation}
\end{definition}
First proposed in \cite{tadmor_entropy_stable_fv_87}, the property in \eqref{eq:tadmor} is referred to in the literature as \emph{Tadmor's condition} or the \emph{shuffle condition}, and enables the chain rule to be circumvented when deriving semi-discrete forms of bounds such as \eqref{eq:entropy_balance} in the context of an entropy-stable discretization. At element interfaces, we use \emph{entropy-stable} or \emph{entropy-conservative} directional numerical fluxes, for which the following definition is introduced (see, for example, \cite[Definitions 3.1 and 3.2]{chen_shu_entropy_stable_dgsbp_17}).
\begin{definition}\label{def:num_flux}
A directional numerical flux function $\vc{\fn{f}}^* : \Upsilon \times \Upsilon \times \mathbb{S}^{d-1} \to \mathbb{R}^{\nelem}$ is \emph{entropy stable} if, for any direction $\vec{n} \in \mathbb{S}^{d-1}$, it satisfies the \emph{conservation} and \emph{consistency} conditions given by
\begin{subequations}
\begin{alignat}{2}
\vc{\fn{f}}^*(\vc{U}^-,\vc{U}^+,\vec{n}) &= - \vc{\fn{f}}^*(\vc{U}^+,\vc{U}^-,-\vec{n}), \quad &&\forall \, \vc{U}^-,\vc{U}^+ \in \Upsilon,\label{eq:conservative_flux}\\
\vc{\fn{f}}^*(\vc{U},\vc{U},\vec{n}) &= \vc{\fn{f}}(\vc{U},\vec{n}), \quad &&\forall \, \vc{U} \in \Upsilon, \label{eq:consistent_flux}
\end{alignat}
\end{subequations}
respectively, as well as the \emph{entropy condition}, which is defined analogously to \eqref{eq:tadmor} by
\begin{equation}\label{eq:tadmor_directional}
\big(\vc{\mathcal{W}}(\vc{U}^+) - \vc{\mathcal{W}}(\vc{U}^-)\big)^\T \vc{\fn{F}}^*(\vc{U}^-,\vc{U}^+, \vec{n}) \leq \big(\vec{\varPsi}(\vc{\mathcal{W}}(\vc{U}^+)) - \vec{\varPsi}(\vc{\mathcal{W}}(\vc{U}^-))\big) \cdot \vec{n}, \quad \forall \, \vc{U}^-,\vc{U}^+ \in \Upsilon.
\end{equation}
Such a numerical flux is \emph{entropy conservative} if \eqref{eq:tadmor_directional} holds as an equality for all arguments.
\end{definition}
For the numerical experiments which we will present in \cref{sec:numerical}, we employ an interface flux consisting of an entropy-conservative two-point flux in the normal direction augmented by a local Lax--Friedrichs dissipative term (see, for example, Ranocha \cite[Section 6.1]{ranocha_comparison_entropy_conservative_fluxes_18}). Such a flux takes the form
\begin{equation}\label{eq:num_flux}
\vc{\fn{f}}^*(\vc{\fn{u}}^-, \vc{\fn{u}}^+, \vec{n}) \df \vc{\fn{f}}^\#(\vc{\fn{u}}^-, \vc{\fn{u}}^+, \vec{n}) - \frac{1}{2}\lambda(\vc{\fn{u}}^-, \vc{\fn{u}}^+, \vec{n})(\vc{\fn{u}}^+ - \vc{\fn{u}}^-),
\end{equation}
where the first term on the right-hand side is an entropy-conservative directional flux given by 
\begin{equation}
\vc{\fn{f}}^\#(\vc{\fn{u}}^-, \vc{\fn{u}}^+, \vec{n}) \df \sum_{m=1}^d n_m \vc{\fn{f}}_m^\#(\vc{\fn{u}}^-, \vc{\fn{u}}^+),
\end{equation}
and $\lambda(\vc{\fn{u}}^-, \vc{\fn{u}}^+, \vec{n}) \in \mathbb{R}^+$ is an estimate of the maximum wave speed in the normal direction. Alternatively, one could use entropy-stable matrix dissipation terms such as those described by Ismail and Roe \cite{ismail_roe_ec_flux_09} and Winters \etal \cite{winters_matrix_diss_17} or scalar dissipation terms based on the jumps in the entropy variables (see, for example, Chandrashekar \cite{chandrashekar_kep_entropy_fv_13}).
\par
In order to construct an expression for $\vc{r}^{(h,\kappa,e)}(t)$ in \eqref{eq:dudt_modal} which results in an entropy-stable scheme, we require the evaluation of the averaged volume and surface metric terms as
\begin{subequations}
\begin{align}
\avg{G_{lm}^{(\kappa)}}_{ij} &\df \frac{1}{2}\big[\vc{\vec{G}}^{(\kappa)}(\vec{\xi}^{(i)}) + \vc{\vec{G}}^{(\kappa)}(\vec{\xi}^{(j)})\big]_{lm},\label{eq:avg_vol_metrics}\\
\avg{\fn{J}^{(\kappa,\zeta)}n_m^{(\kappa,\zeta)}}_{ij} &\df \frac{1}{2}\big[\vc{\vec{G}}^{(\kappa)}(\vec{\xi}^{(i)})^\T\hat{\vec{n}}^{(\zeta)} + \vc{\vec{G}}^{(\kappa)}(\vec{\xi}^{(\zeta,j)})^\T\hat{\vec{n}}^{(\zeta)}\big]_m,\label{eq:avg_fac_metrics}
\end{align}
\end{subequations}
which define the matrices $\avg{G_{lm}^{(\kappa)}} \in \mathbb{R}^{\nvolnodes \times \nvolnodes}$ and $\avg{\fn{J}^{(\kappa,\zeta)}n_m^{(\kappa,\zeta)}} \in \mathbb{R}^{\nvolnodes \times \nfacnodes}$, respectively. The entropy-conservative two-point flux functions are then computed using the \emph{entropy-projected conservative variables}, a term introduced in \cite{chan_discretely_entropy_conservative_dg_sbp_18} in reference to the conservative variables evaluated in terms of the projected entropy variables in \eqref{eq:entropy_expansion}, as given by
\begin{subequations}
\begin{align}
F_{ij}^{(\kappa,m,e)}(t) &\df \fn{f}_{me}^\#(\vc{\mathcal{U}}(\vc{w}_i^{(h,\kappa)}(t)),\vc{\mathcal{U}}(\vc{w}_j^{(h,\kappa)}(t))),\label{eq:two_point_flux_volume}\\
F_{ij}^{(\kappa,\zeta,m,e)}(t) &\df \fn{F}_{me}^\#(\vc{\mathcal{U}}(\vc{w}_i^{(h,\kappa)}(t)),\vc{\mathcal{U}}(\vc{w}_j^{(h,\kappa,\zeta)}(t))),\label{eq:two_point_flux_facet}
\end{align}
\end{subequations}
which define the matrices $\mat{F}^{(\kappa,m,e)}(t) \in \mathbb{R}^{\nvolnodes \times \nvolnodes}$ and $\mat{F}^{(\kappa,\zeta,m,e)}(t) \in \mathbb{R}^{\nvolnodes \times \nfacnodes}$. Defining the exterior values of the entropy variables as $\vc{w}_i^{(h,\kappa,\zeta,+)}(t) \in \mathbb{R}^{\ncons}$, we likewise evaluate $\vc{f}^{(*,\kappa,\zeta,e)}(t) \in \mathbb{R}^{\nfacnodes}$ as
\begin{equation}\label{eq:interface_flux}
f_i^{(*,\kappa,\zeta,e)}(t) \df \fn{f}_e^*\big(\vc{\mathcal{U}}(\vc{w}_i^{(h,\kappa,\zeta)}(t)),\, \vc{\mathcal{U}}(\vc{w}_i^{(h,\kappa,\zeta,+)}(t)),\, \vec{n}^{(\kappa,\zeta)}(\fvec{x}^{(\kappa)}(\vec{\xi}^{(\zeta,i)}))\big).
\end{equation}
Having introduced the essential components of the scheme, flux-differencing weight-adjusted modal DSEM is now obtained by computing the time derivative as in \eqref{eq:dudt_modal}, with the nodal right-hand side computed as
\begin{equation}\label{eq:entropy_stable}
\begin{aligned}
\vc{r}^{(h,\kappa,e)}(t) \df  &- \sum_{l=1}^d \bigg(2\mat{S}^{(l)} \odot \sum_{m=1}^d\avg{G_{lm}^{(\kappa)}} \odot \mat{F}^{(\kappa,m,e)}(t)\bigg) \vc{1}^{(\nvolnodes)} \\ &- \sum_{\zeta=1}^{\nfac} \bigg(\mat{C}^{(\kappa,\zeta,e)}(t)\vc{1}^{(\nfacnodes)}  + \big(\mat{R}^{(\zeta)}\big)^\T\Big(\mat{B}^{(\zeta)}\mat{J}^{(\kappa,\zeta)}\vc{f}^{(*,\kappa,\zeta,e)}(t) - \big(\mat{C}^{(\kappa,\zeta,e)}(t)\big)^\T\vc{1}^{(\nvolnodes)}\Big)\bigg),
\end{aligned}
\end{equation}
where $\odot$ denotes the Hadamard product given by $[\mat{A} \odot \mat{B}]_{ij} \df A_{ij}B_{ij}$, and we define
\begin{equation}\label{eq:facet_correction}
\mat{C}^{(\kappa,\zeta,e)}(t) \df \big(\mat{R}^{(\zeta)}\big)^\T\mat{B}^{(\zeta)} \odot \sum_{m=1}^d\avg{\fn{J}^{(\kappa,\zeta)}n_m^{(\kappa,\zeta)}} \odot \mat{F}^{(\kappa,\zeta,m,e)}(t).
\end{equation}
The above formulation can be used with any set of diagonal-norm SBP operators for which the boundary matrices can be decomposed as in \eqref{eq:e_decomp}, and it will be shown in \cref{sec:analysis} that the formulation is mathematically equivalent to that in \cite[Eq.\ (35)]{chan_wilcox_entropystable_curvilinear_19} when the same SBP operators are used in both schemes, a fact which we will exploit in our analysis of conservation, free-stream preservation, and entropy stability.
\begin{remark}\label{rmk:diag_e}
When diagonal-E operators are used, the terms in \eqref{eq:entropy_stable} involving the correction operator $\mat{C}^{(\kappa,\zeta,e)}(t)$ in \eqref{eq:facet_correction} cancel, and hence the scheme simplifies to
\begin{equation}\label{eq:diag_e_entropy_stable}
\vc{r}^{(h,\kappa,e)}(t) =  - \sum_{l=1}^d \bigg(2\mat{S}^{(l)} \odot \sum_{m=1}^d\avg{G_{lm}^{(\kappa)}} \odot \mat{F}^{(\kappa,m,e)}(t)\bigg) \vc{1}^{(\nvolnodes)} -\sum_{\zeta=1}^{\nfac}\big(\mat{R}^{(\zeta)}\big)^\T\mat{B}^{(\zeta)}\mat{J}^{(\kappa,\zeta)}\vc{f}^{(*,\kappa,\zeta,e)}(t),
\end{equation}
a formulation which Ranocha \etal employ in \cite[Section 2.1]{ranocha_efficient_implementation_23} to obtain highly efficient algorithms using tensor-product LGL quadrature rules on curved quadrilateral and hexahedral elements. However, since the tensor-product SBP operators on the reference triangle and tetrahedron introduced in \cref{sec:tensor_sbp} are not diagonal-E operators, the proposed schemes require the evaluation of the full right-hand side in \eqref{eq:entropy_stable}, including the facet correction, the efficient implementation of which will be discussed in the following section.
\end{remark}

\section{Efficient implementation}\label{sec:implementation}
In this section, we discuss and analyze several important algorithmic considerations pertaining to the implementation of the proposed schemes, particularly regarding techniques for exploiting the sparsity and tensor-product structure of the SBP operators described in \cref{sec:tensor_sbp} within the context of an entropy-stable flux-differencing DSEM. The algorithms described here are implemented within the open-source Julia code \texttt{StableSpectralElements.jl} developed by the first author.\footnote{\texttt{StableSpectralElements.jl} is available at \url{https://github.com/tristanmontoya/StableSpectralElements.jl}.} 
\subsection{Exploiting operator sparsity in flux differencing}
As discussed, for example, by Ranocha \etal \cite[Figure 3]{ranocha_efficient_implementation_23}, the cost of an entropy-stable scheme is dominated by the flux-differencing terms, for which the primary expense is the evaluation of two-point entropy-conservative flux functions between pairs of quadrature nodes. By rewriting the volume contributions appearing in the first term on the right-hand side of \eqref{eq:entropy_stable} or \eqref{eq:diag_e_entropy_stable} as
\begin{equation}\label{eq:flux_differencing}
\bigg[\Big(2\mat{S}^{(l)} \odot \sum_{m=1}^d\avg{G_{lm}^{(\kappa)}} \odot \mat{F}^{(\kappa,m,e)}(t)\Big) \vc{1}^{(\nvolnodes)}\bigg]_i  = \sum_{j = 1}^{\nvolnodes} S_{ij}^{(l)} \fn{f}_e^\#\big(\vc{\mathcal{U}}(\vc{w}_i^{(h,\kappa)}(t)),\, \vc{\mathcal{U}}(\vc{w}_j^{(h,\kappa)}(t)),\, \avg{2\vec{g}^{(\kappa,l)}}_{ij}\big),
\end{equation}
we observe, as noted in \cite[Section 2.2]{ranocha_efficient_implementation_23}, that due to the symmetry of $\avg{G_{lm}^{(\kappa)}} \odot \mat{F}^{(\kappa,m,e)}(t)$ and the skew-symmetry of $\mat{S}^{(l)}$, it is only necessary to iterate over the indices $i$ and $j$ corresponding to the strictly upper-triangular parts of such matrices. Moreover, the sum need only be taken over the indices for which $S_{ij}^{(l)} \neq 0$, and the corresponding values of the vector
\begin{equation}\label{eq:avg_volume_metric}
\avg{2\vec{g}^{(\kappa,l)}}_{ij} \df \big[\fn{g}_{l1}^{(\kappa)}(\vec{\xi}^{(i)}) +  \fn{g}_{l1}^{(\kappa)}(\vec{\xi}^{(j)}), \ldots, \fn{g}_{ld}^{(\kappa)}(\vec{\xi}^{(i)}) +  \fn{g}_{ld}^{(\kappa)}(\vec{\xi}^{(j)})\big]^\T
\end{equation}
within the directional two-point flux can be computed on the fly in order to avoid storing the dense matrices $\avg{G_{lm}^{(\kappa)}}$ in memory. Computing the sum on the right-hand side of \eqref{eq:flux_differencing} for all $i \in \{1:\nvolnodes\}$ therefore requires the evaluation of the two-point flux function, which for entropy-stable schemes is typically a relatively expensive operation involving the logarithmic mean, once per nonzero entry in the strictly upper-triangular part of $\mat{S}^{(l)}$, with the total number of required floating-point operations being proportional to such a quantity. Similarly, expressing the averaged metric terms in \eqref{eq:avg_fac_metrics} as
\begin{equation}\label{eq:avg_facet_metric}
\avg{\fn{J}^{(\kappa,\zeta)}\vec{n}^{(\kappa,\zeta)}}_{ij} \df \big[\avg{\fn{J}^{(\kappa,\zeta)}n_1^{(\kappa,\zeta)}}_{ij}, \ldots, \avg{\fn{J}^{(\kappa,\zeta)}n_d^{(\kappa,\zeta)}}_{ij} \big]^\T,
\end{equation}
we can evaluate the contributions from $\mat{C}^{(\kappa,\zeta,e)}(t)\vc{1}^{(\nfacnodes)}$ and $(\mat{C}^{(\kappa,\zeta,e)}(t))^\T\vc{1}^{(\nvolnodes)}$ on the second line of \eqref{eq:entropy_stable} simultaneously by initializing both such vectors to zero and then iterating over values of $i$ and $j$ such that $[(\mat{R}^{(\zeta)})^\T\mat{B}^{(\zeta)}]_{ij} \neq 0$. Within each iteration, we compute the corresponding two-point flux and multiply each component by the corresponding nonzero matrix entry to obtain
\begin{equation}
C_{ij}^{(\kappa,\zeta,e)}(t) = \big[(\mat{R}^{(\zeta)})^\T\mat{B}^{(\zeta)}\big]_{ij} \fn{f}_e^\#\big(\vc{\mathcal{U}}(\vc{w}_i^{(h,\kappa)}(t)),\, \vc{\mathcal{U}}(\vc{w}_j^{(h,\kappa,\zeta)}(t)),\, \avg{\fn{J}^{(\kappa,\zeta)}\vec{n}^{(\kappa,\zeta)}}_{ij}\big).
\end{equation} 
Such values are then accumulated within $\mat{C}^{(\kappa,\zeta,e)}(t)\vc{1}^{(\nfacnodes)}$ and $(\mat{C}^{(\kappa,\zeta,e)}(t))^\T\vc{1}^{(\nvolnodes)}$ as
\begin{subequations}
\begin{align}
\big[\mat{C}^{(\kappa,\zeta,e)}(t)\vc{1}^{(\nfacnodes)}\big]_i &\gets  \big[\mat{C}^{(\kappa,\zeta,e)}(t)\vc{1}^{(\nfacnodes)}\big]_i + C_{ij}^{(\kappa,\zeta,e)}(t), \\
\big[\big(\mat{C}^{(\kappa,\zeta,e)}(t)\big)^\T\vc{1}^{(\nvolnodes)}\big]_j &\gets  \big[\big(\mat{C}^{(\kappa,\zeta,e)}(t)\big)^\T\vc{1}^{(\nvolnodes)}\big]_j + C_{ij}^{(\kappa,\zeta,e)}(t)
\end{align}
\end{subequations}
before proceeding to the next nonzero entry of the matrix $(\mat{R}^{(\zeta)})^\T\mat{B}^{(\zeta)}$. \par 
Letting $\operatorname{nnz}(\mat{A})$ denote the number of nonzero entries in a given matrix $\mat{A}$, it follows from the above discussion that the number of element-local directional two-point flux evaluations (i.e.\ neglecting the numerical interface flux, which is computed in a separate routine) required to evaluate the right-hand side in \eqref{eq:entropy_stable} can be computed as 
\begin{equation}\label{eq:num_fluxes}
\text{Number of two-point flux evaluations} = \sum_{l=1}^d \frac{1}{2}\operatorname{nnz}\Big(\mat{S}^{(l)}\Big) + \sum_{\zeta=1}^{\nfac} \operatorname{nnz}\Big(\big(\mat{R}^{(\zeta)}\big)^\T\mat{B}^{(\zeta)}\Big).
\end{equation}
To the authors' knowledge, the matrices $\mat{S}^{(l)}$ are dense for all existing high-order entropy-stable discretizations on triangles or tetrahedra. Recalling from \cite{hicken_mdsbp_16} and \cite{delrey_mdsbp_sat_18} that the minimum number of volume quadrature nodes for an SBP operator of degree $q$ is the dimension $N_q^*$ of the associated total-degree polynomial space, which scales as $\mathcal{O}(q^d)$, the number of two-point fluxes, and hence the computational work required to evaluate the flux-differencing volume terms, is therefore expected to scale as $\mathcal{O}(q^{2d})$. By contrast, such matrices are sparse for the operators described in \cref{sec:tensor_sbp}, with their one-dimensional coupling along lines of nodes resulting in the same $\mathcal{O}(q^{d+1})$ complexity as for tensor-product DSEMs on quadrilaterals or hexahedra. 
\begin{remark}
While the second term in \eqref{eq:num_fluxes} vanishes for diagonal-E multidimensional SBP operators due to the simplifications made in \eqref{eq:diag_e_entropy_stable}, such operators require quadrature rules using a much larger number of nodes for a given degree than would otherwise be needed, and are currently only available for modest polynomial degrees.\footnote{To the authors' knowledge, the highest-order diagonal-E SBP operators on tetrahedra are those recently proposed by Worku \etal \cite{worku_quadrature_sbp_23}, who provide quadrature rules of up to degree 10 suitable for SBP operators of up to degree 5.} For more general non-tensor-product quadrature rules without collocated facet nodes, the matrices $(\mat{R}^{(\zeta)})^\T\mat{B}^{(\zeta)}$ are also dense, coupling every volume quadrature node to every facet quadrature node. 
\end{remark}
\begin{remark}
It may be possible to achieve a further reduction in computational expense by arranging the flux-differencing loops so as to iterate along lines of nodes in the $\eta_1$, $\eta_2$, and $\eta_3$ directions instead of over nonzero entries in the operators $\mat{S}^{(l)}$ and $(\mat{R}^{(\zeta)})^\T\mat{B}^{(\zeta)}$ on the reference simplex, thereby performing flux differencing directly within the collapsed coordinate system. Such additional optimizations would neither alter the underlying numerical methods nor affect the algorithms' asymptotic complexity, and will be investigated in future work focusing on the implementation and performance aspects of entropy-stable schemes on triangular and tetrahedral elements using tensor-product as well as multidimensional SBP operators.
\end{remark}

\subsection{Exploiting sum factorization for tensor-product operators}
In addition to the reduction in computational complexity of the flux-differencing terms, the algorithmic benefits of tensor-product operators discussed in \cite[Section 6]{montoya_sem_23} regarding the skew-symmetric scheme in \eqref{eq:rhs_skew} extend directly to the matrix-vector products in the proposed entropy-stable formulations, which involve $\mat{V}$ and $\mat{R}^{(\zeta)}$ as well as their transposes. As such, we are able to exploit the structure of such matrices through sum factorization in the evaluation of the conservative variables at the volume and facet quadrature nodes as
\begin{subequations}\label{eq:matvec_cons_vars}
\begin{align}
\vc{u}^{(h,\kappa,e)}(t) & = \mat{V}\vc{\tilde{u}}^{(h,\kappa,e)}(t), \quad\\
\vc{u}^{(h,\kappa,\zeta, e)}(t) & = \mat{R}^{(\zeta)}\vc{u}^{(h,\kappa,e)}(t), \quad \forall \, \zeta \in \{1:\nfac\},
\end{align}
\end{subequations}
as well as in the weight-adjusted projection of the entropy variables as
\begin{subequations}\label{eq:matvec_entropy_projection}
\begin{align}
\vc{\tilde{w}}^{(h,\kappa,e)}(t) & = \mat{M}^{-1}\mat{V}^\T\mat{W}\big(\mat{J}^{(\kappa)}\big)^{-1} \mat{V}\mat{M}^{-1}\mat{V}^\T\mat{W}\mat{J}^{(\kappa)}\big[\mathcal{W}_e(\vc{u}_1^{(h,\kappa)}(t)),\ldots,\mathcal{W}_e(\vc{u}_{\nvolnodes}^{(h,\kappa)}(t))\big]^\T,\\
\vc{w}^{(h,\kappa, e)}(t) & = \mat{V}\vc{\tilde{w}}^{(h,\kappa,e)}(t),\\
\vc{w}^{(h,\kappa,\zeta, e)}(t) & = \mat{R}^{(\zeta)}\vc{w}^{(h,\kappa,e)}(t), \quad \forall \, \zeta \in \{1:\nfac\},
\end{align}
\end{subequations}
where we also exploit the fact that $\mat{M}^{-1}$ is the identity matrix due to the PKD basis remaining orthonormal under all quadrature rules considered in this work. Furthermore, we employ sum factorization when applying $(\mat{R}^{(\zeta)})^\T$ to obtain $\vc{r}^{(h,\kappa,e)}(t)$ in \eqref{eq:entropy_stable}, and, finally, in the evaluation of the time derivative in \eqref{eq:dudt_modal} using the weight-adjusted inverse as
\begin{equation}\label{eq:matvec_dudt}
\frac{\dd \vc{\tilde{u}}^{(h,\kappa,e)}(t)}{\dd t} = \mat{M}^{-1}\mat{V}^\T\mat{W}\big(\mat{J}^{(\kappa)}\big)^{-1} \mat{V}\mat{M}^{-1} \mat{V}^\T\vc{r}^{(h,\kappa,e)}(t).
\end{equation}
Such an approach results in the entire algorithm for computing the local time derivative requiring $\mathcal{O}(p^{d+1})$ floating-point operations under the standard assumption that $q$ scales as $\mathcal{O}(p)$ with the polynomial  degree $p$ of the modal expansion. To the authors' knowledge, this is not achieved by any prior entropy-stable method on triangles or tetrahedra, for which the required dense matrix operations are of $\mathcal{O}(p^{2d})$ complexity. Moreover, by avoiding the construction of physical operator matrices and averaging the metric terms on the fly, memory usage is minimized, with per-element memory requirements scaling as $\mathcal{O}(p^{d})$ due to the only necessary storage being geometric information at each volume and facet quadrature node, and, optionally, precomputed diagonal entries of the matrices $\mat{W}(\mat{J}^{(\kappa)})^{-1}$, $\mat{W}\mat{J}^{(\kappa)}$, and $\mat{B}^{(\zeta)}\mat{J}^{(\kappa,\zeta)}$. 
\subsection{Comparison to multidimensional summation-by-parts operators}
To make the above discussion more quantitative, we now analyze the number of two-point flux evaluations using \eqref{eq:num_fluxes} for specific SBP operators of varying polynomial degrees, where we compare our sparse tensor-product operators on the reference triangle  and tetrahedron to those constructed using symmetric quadrature rules as described by Chan \cite[Lemma 1]{chan_discretely_entropy_conservative_dg_sbp_18}, for which the matrices $\mat{S}^{(l)}$ and $(\mat{R}^{(\zeta)})^\T\mat{B}^{(\zeta)}$ are dense (i.e.\ they are not diagonal-E operators). 
We designate the latter class of operator as \emph{multidimensional} to distinguish them from the \emph{tensor-product} operators on the reference triangle and tetrahedron described in \cref{sec:tensor_sbp}, for which the volume quadrature rules are shown in \cref{fig:tensor_nodes}. On the triangle, we construct such multidimensional SBP operators using $2q$ Xiao-Gimbutas quadrature rules \cite{xiao_gimbutas_quadrature_10} for volume integration and degree $2q+1$ LG quadrature rules for facet integration. On the tetrahedron, the multidimensional SBP operators are constructed using degree $2q$ Ja\'skowiec-Sukumar quadrature rules \cite{jaskowiec_sukumar_symmetric_cubature_21} for volume integration and degree $2q$ Xiao-Gimbutas triangular quadrature rules for facet integration. Examples of such quadrature nodes are shown in \cref{fig:multi_nodes}, where we note for the multidimensional operators that the symmetric quadrature rules used for volume integration on the triangle are also used to integrate over each face of the tetrahedron. \par 
\begin{figure}[t!]
\centering
\begin{subfigure}{0.50\textwidth}
\centering
\includegraphics[height=42mm]{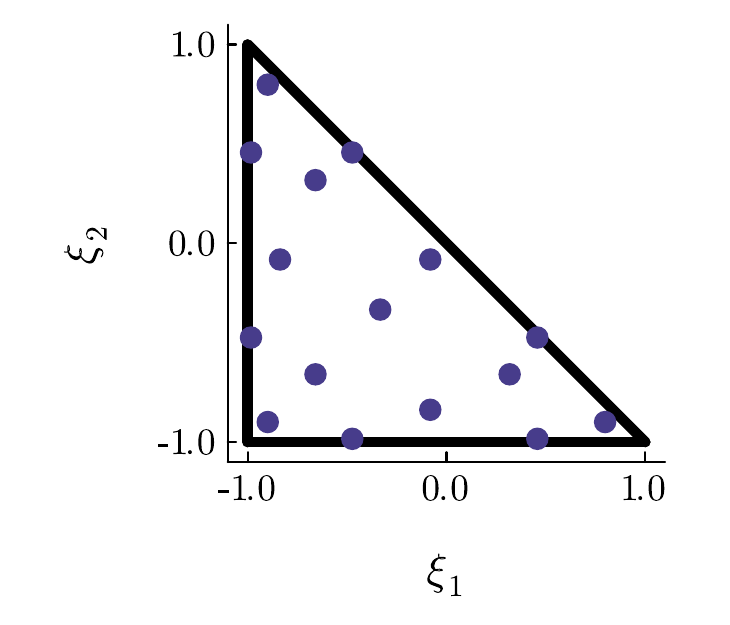}
\caption{Multidimensional quadrature nodes on the triangle}
\end{subfigure}~
\begin{subfigure}{0.48\textwidth}
\centering
\includegraphics[height=42mm]{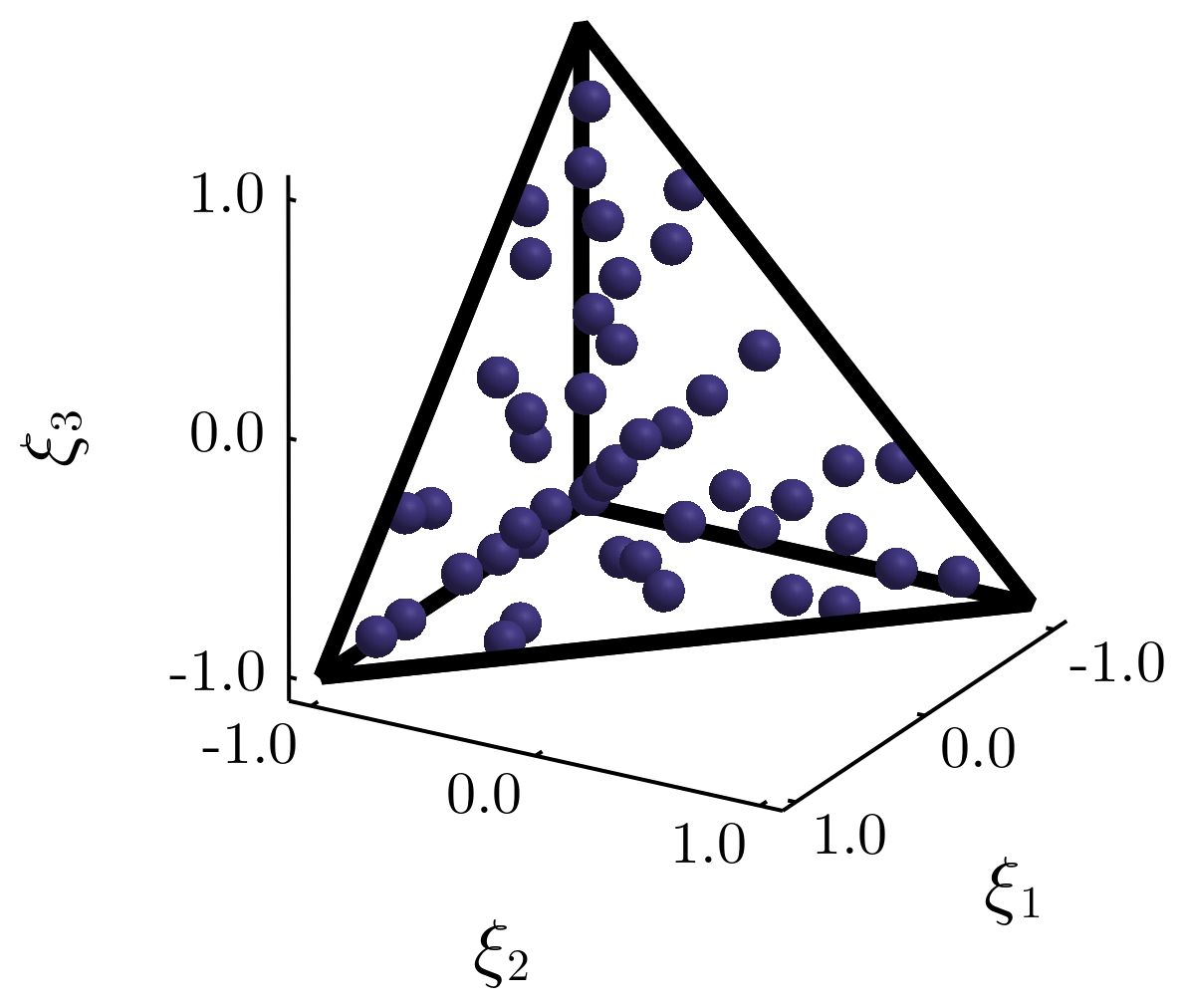}
\caption{Multidimensional quadrature nodes on the tetrahedron}
\end{subfigure}
\caption{Symmetric quadrature nodes for multidimensional SBP operators of degree $q = 4$ on the triangle and tetrahedron}\label{fig:multi_nodes}
\end{figure}
In \cref{fig:two_point_fluxes}, we plot the number of required two-point flux evaluations given by \eqref{eq:num_fluxes} for each class of operator over a range of polynomial degrees from $q=2$ to $q=15$, with the exception of the multidimensional operators on tetrahedra, for which suitable symmetric quadrature rules are only currently available, to the authors' knowledge, for SBP operators of degree $q \leq 10$. Since dense matrix storage is used in our implementation of the multidimensional operators, the only zero entries considered in \eqref{eq:num_fluxes} for such operators are those along the main diagonal of the skew-symmetric matrix $\mat{S}^{(l)}$, although we observe numerically that a small fraction of the off-diagonal entries are, in fact, on the order of machine precision. The scaling in \cref{fig:two_point_fluxes} is observed to be slightly better than the asymptotic estimates for both classes of operators, and the number of two-point flux evaluations required for the proposed tensor-product approach is smaller than that required when using multidimensional SBP operators for all polynomial degrees considered. As expected, this discrepancy increases with the polynomial degree; for example, the number of two-point flux evaluations is reduced by factors of 1.56 at $q = 2$, 2.78 at $q=5$, and 4.57 at $q=10$ when using tensor-product operators on triangles, and by factors of 1.88 at $q=2$, 3.44 at $q=5$, and 10.99 at $q=10$ when using tensor-product operators on tetrahedra. We also consider the total number of floating-point operations per variable incurred in evaluating the matrix-vector products in \eqref{eq:entropy_stable}, \eqref{eq:matvec_cons_vars}, \eqref{eq:matvec_entropy_projection}, and \eqref{eq:matvec_dudt} on a given element, where we take $p=q$ in all cases. The results of such an analysis, which are displayed in \cref{fig:matrix_ops}, are qualitatively similar to those in \cref{fig:two_point_fluxes} for higher polynomial degrees, although the benefit of the tensor-product operators for low polynomial degrees is less substantial in this regard, requiring roughly the same number of floating-point operations as the multidimensional operators, for example, at $p=2$. \par
There are several caveats which must be addressed regarding the preceding discussion and analysis. First, we note that due to their use of a larger number of volume and facet quadrature nodes than typical multidimensional operators of the same degree, the tensor-product operators require a somewhat greater number of conversions between conservative and entropy variables as well as a somewhat larger number of numerical interface flux evaluations. These operations do not, however, typically constitute the most significant contribution to the overall expense of an entropy-stable scheme and incur a cost which grows more slowly with the polynomial degree than that of the flux-differencing terms. Second, we recognize that comparisons on the basis of floating-point operation count, while more objective than implementation-specific and hardware-specific timing comparisons, are only truly representative of computational cost for compute-bound algorithms. For lower polynomial degrees, algorithm performance is often substantially limited by memory bandwidth, and hence the tensor-product operators may not necessarily outperform their multidimensional counterparts in practice for such regimes. However, due to the arithmetic intensity of an SEM increasing with the polynomial degree of the discretization (see, for example, the roofline analysis in \cite{moxey_matrix_free_triangles_20}), floating-point operation count indeed becomes a relevant measure at higher polynomial degrees, which is precisely where the proposed tensor-product approach shows significant cost savings. The point at which such benefits become substantial is dependent on the specifics of the implementation and hardware (e.g.\ considering the memory access pattern, cache size, as well as the use of single-instruction-multiple-data vectorization or multithreading) as well as the PDE and choice of two-point flux, and is therefore an important topic of future investigation within the context of the high-performance implementation and evaluation of the proposed algorithms.
\begin{figure}[t!]
\centering
\begin{subfigure}{0.48\textwidth}
\centering
\includegraphics[height=42mm]{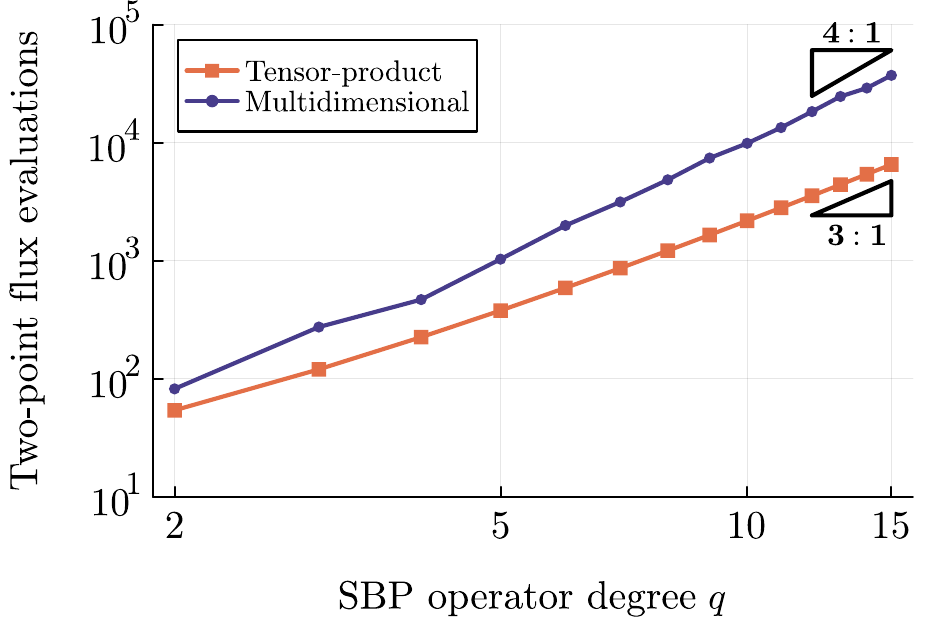}
\caption{Triangles}
\end{subfigure}~
\begin{subfigure}{0.48\textwidth}
\centering
\includegraphics[height=42mm]{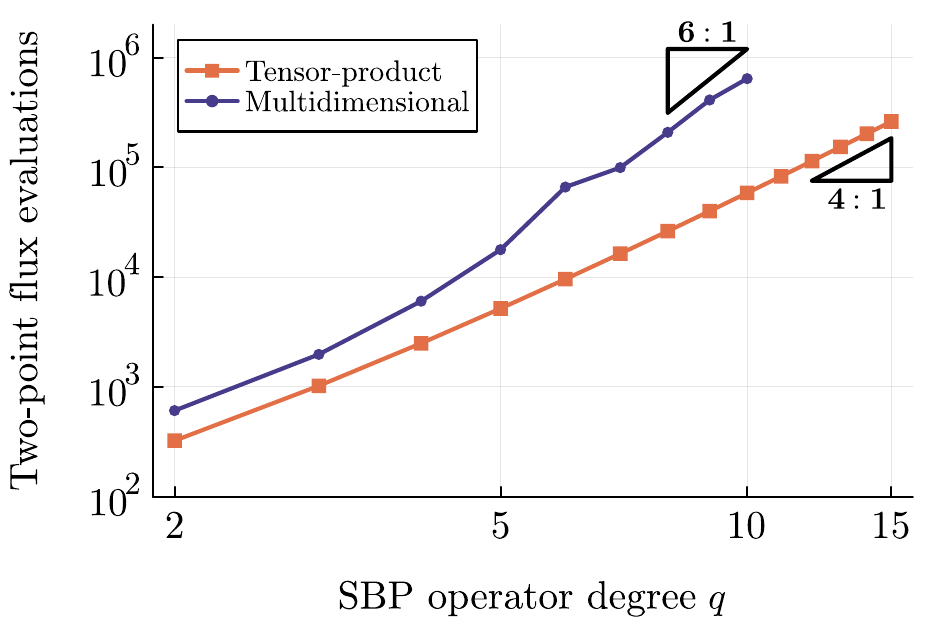}
\caption{Tetrahedra}
\end{subfigure}
\caption{Number of entropy-conservative two-point flux evaluations in local flux-differencing terms}\label{fig:two_point_fluxes}
\end{figure}
\begin{figure}[t!]
\begin{subfigure}{0.48\textwidth}
\centering
\includegraphics[height=42mm]{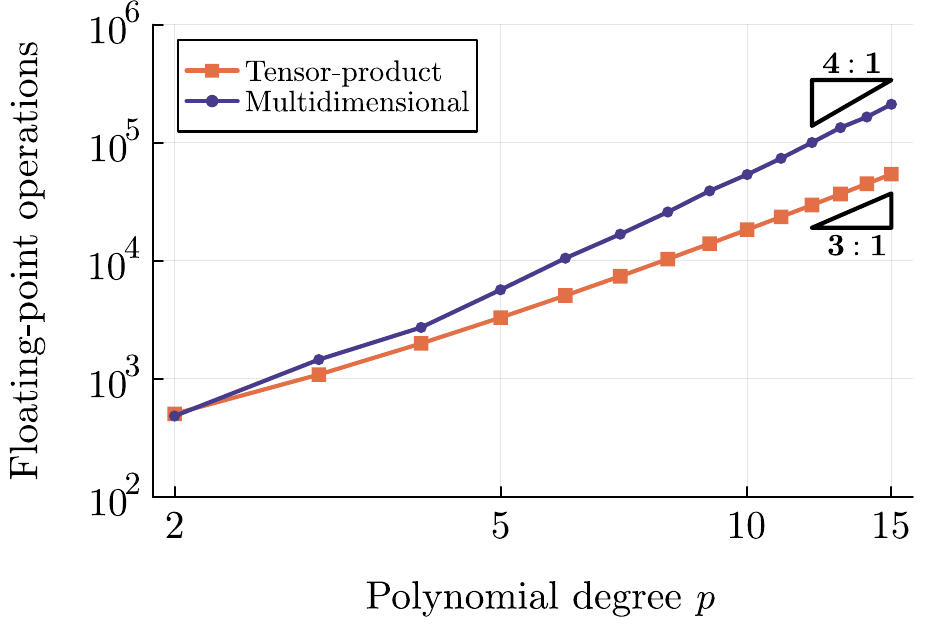}
\caption{Triangles}
\end{subfigure}~
\begin{subfigure}{0.48\textwidth}
\centering
\includegraphics[height=42mm]{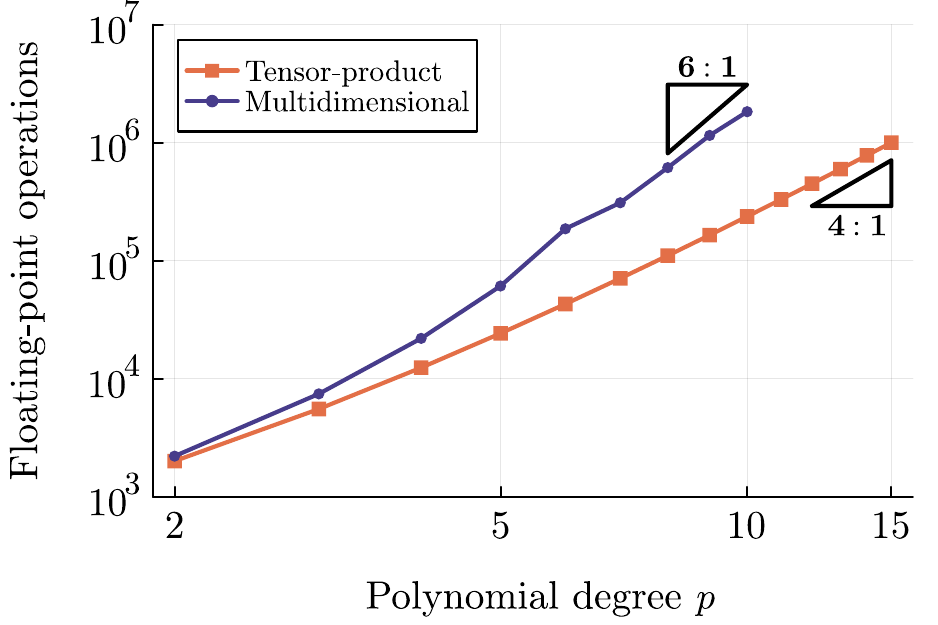}
\caption{Tetrahedra}
\end{subfigure}
\caption{Total number of floating-point operations per variable in local operator evaluation}\label{fig:matrix_ops}
\end{figure}

\section{Conservation, free-stream preservation, and entropy stability}\label{sec:analysis}
The previous section demonstrates that the use of sparse tensor-product operators in collapsed coordinates enables algorithmic improvements relative to comparable entropy-stable discretizations using multidimensional SBP operators, particularly for higher polynomial degrees. We will now demonstrate that the proposed schemes are conservative, free-stream preserving, and entropy stable due to their equivalence to formulations based on \emph{hybridized summation-by-parts operators}.\footnote{Such operators were originally introduced by Chan in \cite{chan_discretely_entropy_conservative_dg_sbp_18} as \emph{decoupled SBP operators}, with the term \emph{hybridized SBP operator} popularized following the review paper by Chen and Shu \cite{chen_shu_dgsbp_review_19}. We note, however, that this notion of ``hybridization'' should be distinguished from that associated with hybridized DG methods (see, for example, Cockburn \etal \cite{cockburn_hdg_09}), which exploit static condensation to reduce the number of coupled degrees of freedom.} The construction of such hybridized operators from SBP operators satisfying the conditions of \cref{def:sbp} is demonstrated with the following lemma.
\begin{lemma}\label{lem:hybridized_sbp_property}
Given any nodal SBP operator $\mat{D}^{(l)} \df \mat{W}^{-1}(\mat{S}^{(l)} + \tfrac{1}{2}\mat{E}^{(l)})$ on the reference element in the sense of \cref{def:sbp} for which $\mat{E}^{(l)}$ takes the form of \eqref{eq:e_decomp}, we can construct a block matrix
\begin{equation}\label{eq:q_hybridized}
\mat{\bar{Q}}^{(l)} \df \mqty[\mat{S}^{(l)}  &  \frac{1}{2}\hat{n}_l^{(1)}\big(\mat{R}^{(1)}\big)^\T\mat{B}^{(1)} & \cdots & \frac{1}{2}\hat{n}_l^{(\nfac)}\big(\mat{R}^{(\nfac)}\big)^\T\mat{B}^{(\nfac)} \\ -\frac{1}{2}\hat{n}_l^{(1)}\mat{B}^{(1)}\mat{R}^{(1)}  & \frac{1}{2}\hat{n}_l^{(1)}\mat{B}^{(1)}  &  & \\ \vdots  &  &  \ddots &  \\ -\frac{1}{2}\hat{n}_l^{(\nfac)}\mat{B}^{(\nfac)}\mat{R}^{(\nfac)}   &  &  & \frac{1}{2}\hat{n}_l^{(\nfac)}\mat{B}^{(\nfac)}  ],
\end{equation}
satisfying the following ``SBP-like'' property from \cite[Eq.\ (32)]{chan_discretely_entropy_conservative_dg_sbp_18}:
\begin{equation}\label{eq:sbp_hybridized}
\mat{\bar{Q}}^{(l)} + \big(\mat{\bar{Q}}^{(l)} \big)^\T = \mat{\bar{E}}^{(l)}, \quad \text{where} \quad \mat{\bar{E}}^{(l)}\df \mqty[\mat{0}^{(\nvolnodes \times \nvolnodes)} & & \\ & \hat{n}^{(1)}\mat{B}^{(1)} & \\ & & \ddots & \\ & & & \hat{n}^{(\nfac)}\mat{B}^{(\nfac)} ].
\end{equation}
\end{lemma}
\begin{proof}
Aside from the fact that our notation separates the operators on individual facets, the result is identical to \cite[Theorem 1]{chan_skewsymmetric_modaldg_sbp_19}. The proof relies on the fact that when adding \eqref{eq:q_hybridized} to its transpose, the top-left block vanishes by the skew-symmetry of $\mat{S}^{(l)}$ and the off-diagonal blocks vanish due to the blocks of the first row being the negative transposes of those in the first column, and hence the diagonal matrix on the right-hand side of \eqref{eq:sbp_hybridized} is all that remains.
\end{proof}
The hybridized operators in \eqref{eq:q_hybridized} are of dimension $\bar{N}_q$ by $\bar{N}_q$, where $\bar{N}_q \df \nvolnodes + N_{q_f}^{(1)} + \cdots +  N_{q_f}^{(\nfac)}$. Within a flux-differencing formulation, they operate on block matrices of two-point fluxes given by
\begin{equation}\label{eq:flux_hybridized}
\mat{\bar{F}}^{(\kappa,m,e)}(t) \df \mqty[\mat{F}^{(\kappa,m,e)}(t) & \mat{F}^{(\kappa,1,m,e)}(t) & \cdots & \mat{F}^{(h,\nfac,m,e)}(t)\\ \mat{F}^{(\kappa,1,m,e)}(t) & \mat{F}^{(\kappa,1,1,m,e)}(t)  &  \cdots & \mat{F}^{(\kappa,1,\nfac,m,e)}(t)  \\ \vdots & \vdots  & \ddots  & \vdots  \\ \mat{F}^{(\kappa,\nfac,m,e)}(t)  & \mat{F}^{(\kappa,\nfac,1,m,e)}(t)  &  \cdots & \mat{F}^{(\kappa,\nfac,\nfac,m,e)}(t)],
\end{equation}
where the entries of the blocks $\mat{F}^{(\kappa,\zeta,\eta,m,e)}(t) \in \mathbb{R}^{\nfacnodes \times N_{q_f}^{(\eta)}}$ are given by
\begin{equation}
F_{ij}^{(\kappa,\zeta,\eta,m,e)}(t) \df \fn{F}_{me}^\#(\vc{\mathcal{U}}(\vc{w}_i^{(h,\kappa,\zeta)}(t)),\vc{\mathcal{U}}(\vc{w}_j^{(h,\kappa,\eta)}(t))).
\end{equation}
As in \cite{chan_wilcox_entropystable_curvilinear_19}, hybridized SBP operators on the physical element are constructed in split form as
\begin{equation}\label{eq:phys_hybridized}
\mat{\bar{Q}}^{(\kappa,m)} \df  \frac{1}{2} \sum_{l=1}^d\Big(\mat{\bar{Q}}^{(l)}\mat{\bar{G}}^{(\kappa,l,m)} + \mat{\bar{G}}^{(\kappa,l,m)}\mat{\bar{Q}}^{(l)}\Big),
\end{equation}
where the diagonal matrices of concatenated volume and facet metric terms are given by 
\begin{equation}
\mat{\bar{G}}^{(\kappa,l,m)} \df \mqty[\mat{G}^{(\kappa,l,m)} & & & \\ & \mat{G}^{(\kappa,1,l,m)} & &\\ & & \ddots & \\ & & & \mat{G}^{(\kappa,\nfac,l,m)}],
\end{equation}
with the blocks $\mat{G}^{(\kappa,\zeta,l,m)} \in \mathbb{R}^{\nfacnodes\times \nfacnodes}$ defined by $G_{ij}^{(\kappa,\zeta,l,m)} \df G_{lm}^{(\kappa)}(\vec{\xi}^{(\zeta,i)})\delta_{ij}$. The following lemma demonstrates the relation of such a split form to the metric-averaging approach used in \cref{sec:entropy_stable}.
\begin{lemma}\label{lem:split_equiv_to_averaging}
The split-form operator in \eqref{eq:phys_hybridized} can be rewritten as 
\begin{equation}\label{eq:split_to_avg}
\mat{\bar{Q}}^{(\kappa,m)} = \sum_{l=1}^d\mat{\bar{Q}}^{(l)} \odot \avg{\bar{G}_{lm}^{(\kappa)}},
\end{equation}
where the entries of $\avg{\bar{G}_{lm}^{(\kappa)}} \in \mathbb{R}^{\bar{N}_q \times \bar{N}_q}$ are given by $\avg{\bar{G}_{lm}^{(\kappa)}}_{ij}  \df \frac{1}{2}(\bar{G}_{ii}^{(\kappa,l,m)} + \bar{G}_{jj}^{(\kappa,l,m)})$. Additionally, an SBP-like property analogous to \eqref{eq:sbp_hybridized} is satisfied on the physical element, as given by
\begin{equation}\label{eq:hybridized_sbp_phys}
\mat{\bar{Q}}^{(\kappa,m)} + \big(\mat{\bar{Q}}^{(\kappa,m)}\big)^\T = \mqty[\mat{0}^{(\nvolnodes \times \nvolnodes)} & & & \\ & \mat{B}^{(1)}\mat{J}^{(\kappa,1)}\mat{N}^{(\kappa,1,m)} & & \\ & & \ddots & \\ & & & \mat{B}^{(\nfac)}\mat{J}^{(\kappa,\nfac)}\mat{N}^{(\kappa,\nfac,m)}], 
\end{equation}
provided that $\mat{J}^{(\kappa,\zeta)}$ and $\mat{N}^{(\kappa,\zeta,m)}$ are computed using \eqref{eq:nanson} based on the volume metric terms.
\end{lemma}
\begin{proof}
Expressing \eqref{eq:phys_hybridized} in indicial form and factorizing, we obtain
\begin{equation}
\bar{Q}_{ij}^{(\kappa,m,n)} = \frac{1}{2}\sum_{l=1}^d \big(\bar{Q}_{ij}^{(l)}\bar{G}_{jj}^{(\kappa,l,m)} + \bar{G}_{ii}^{(\kappa,l,m)}\bar{Q}_{ij}^{(l)}\big) = \frac{1}{2}\sum_{l=1}^d \bar{Q}_{ij}^{(l)}\big(\bar{G}_{ii}^{(\kappa,l,m)} + \bar{G}_{jj}^{(\kappa,l,m)}),
\end{equation}
and hence the result in \eqref{eq:split_to_avg} follows directly from the definition of the Hadamard product. The SBP-like property in \eqref{eq:hybridized_sbp_phys} follows from \eqref{eq:sbp_hybridized} and \eqref{eq:phys_hybridized}, where we obtain the right-hand side using \eqref{eq:nanson} and the fact that the hybridized boundary operators $\mat{\bar{E}}^{(l)}$ are diagonal.
\end{proof}
Our analysis of the schemes constructed in \cref{sec:methods} then requires the following assumption. 
\begin{assumption}\label{asm:discretizations}
The discretization given by \eqref{eq:dudt_modal} with $\vc{r}^{(h,\kappa,e)}(t)$ defined as in \eqref{eq:entropy_stable} is constructed using a set of diagonal-norm SBP operators on the reference element sharing the same positive-definite matrix $\mat{W}$, with boundary operators in the form of \eqref{eq:e_decomp}. Whether computed exactly or approximately as in \eqref{eq:jacobian_approx}, the Jacobian determinant of the mapping from the reference element to each physical element satisfies \eqref{eq:jacobian_positive}. The two-point flux used to compute \eqref{eq:two_point_flux_volume} and \eqref{eq:two_point_flux_facet} is consistent and symmetric in the sense of \cref{def:ec_flux}, while the directional numerical flux in \eqref{eq:interface_flux} is consistent and conservative in the sense of \cref{def:num_flux}.
\end{assumption}
\begin{remark}
The entropy conditions in \eqref{eq:tadmor} and \eqref{eq:tadmor_directional} will not be invoked until the proof of entropy stability. As such, the analysis of conservation and free-stream preservation applies more generally to a wide range of flux-differencing DSEMs, which, as shown in \cite{gassner_winters_kopriva_splitform_nodaldg_sbp_16}, can be constructed so as to recover various split forms in the literature through the particular choice of two-point flux.
\end{remark}
We now have the following theorem relating the formulation proposed in \cref{sec:entropy_stable} to one which is readily analyzed using the properties of hybridized SBP operators.
\begin{theorem}
Under \cref{asm:discretizations}, the DSEM given by \eqref{eq:dudt_modal} with $\vc{r}^{(h,\kappa,e)}(t)$ defined as in \eqref{eq:entropy_stable} is equivalent to the following hybridized SBP formulation proposed in \cite[Eq. (35)]{chan_wilcox_entropystable_curvilinear_19}:
\begin{equation}\label{eq:skew_hybridized}
\begin{aligned}
\mat{\tilde{M}}^{(\kappa)}\frac{\dd \vc{\tilde{u}}^{(h,\kappa,e)}(t)}{\dd t} = &-\mqty[\mat{V}\\ \mat{V}^f ]^\T\sum_{m=1}^d \Big(2\mat{\bar{Q}}^{(\kappa,m)} \odot \mat{\bar{F}}^{(\kappa,m,e)}(t) \Big)\vc{1}^{(\bar{N}_q)}\\ &- \sum_{\zeta=1}^{\nfac}\big(\mat{R}^{(\zeta)}\big)^\T \mat{B}^{(\zeta)}\mat{J}^{(\kappa,\zeta)}\bigg(\vc{f}^{(*,\kappa,\zeta,e)} (t) -  \sum_{m=1}^d \mat{N}^{(\kappa,\zeta,m)}\vc{f}^{(\kappa,\zeta,m,e)}(t) \bigg),
\end{aligned}
\end{equation}
where we define
\begin{equation}
\mat{V}^f \df \mqty[\mat{R}^{(1)}\mat{V} \\ \vdots \\ \mat{R}^{(\nfac)}\mat{V}], \quad \vc{f}^{(\kappa,\zeta,m,e)}(t) \df \mqty[\fn{F}_{me}\big(\vc{\mathcal{U}}(\vc{w}_1^{(h,\kappa,\zeta)}(t))\big) \\ \vdots \\ \fn{F}_{me}\big(\vc{\mathcal{U}}(\vc{w}_{N_{q_f}^{(\zeta)}}^{(h,\kappa,\zeta)}(t))\big) ].
\end{equation}
\end{theorem}
\begin{proof}
Substituting \eqref{eq:entropy_stable} into \eqref{eq:dudt_modal} and grouping all Hadamard products into block matrices gives
\begin{equation}\label{eq:skew_hybridized_weak}
\begin{aligned}
\mat{\tilde{M}}^{(\kappa)}\frac{\dd \vc{\tilde{u}}^{(h,\kappa,e)}(t)}{\dd t} = &- \mqty[\mat{V}\\ \mat{V}^f ]^\T\sum_{l=1} ^d \bigg(2\mat{\bar{S}}^{(l)} \odot \sum_{m=1}^d \avg{\bar{G}_{lm}^{(\kappa)}} \odot \mat{\bar{F}}^{(\kappa,m,e)}(t) \bigg)\vc{1}^{(\bar{N}_q)} \\ &- \sum_{\zeta=1}^{\nfac}\big(\mat{R}^{(\zeta)}\big)^\T \mat{B}^{(\zeta)}\mat{J}^{(\kappa,\zeta)}\vc{f}^{(*,\kappa,\zeta,e)}(t),
\end{aligned}
\end{equation}
where the facet correction terms in \eqref{eq:facet_correction} have been incorporated into the hybridized operator
\begin{equation}\label{eq:s_hybridized}
\mat{\bar{S}}^{(l)} \df \mqty[\mat{S}^{(l)} &  \frac{1}{2}\hat{n}_l^{(1)}\big(\mat{R}^{(1)}\big)^\T\mat{B}^{(1)} & \cdots & \frac{1}{2}\hat{n}_l^{(\nfac)}\big(\mat{R}^{(\nfac)}\big)^\T\mat{B}^{(\nfac)} \\ -\frac{1}{2}\hat{n}_l^{(1)}\mat{B}^{(1)}\mat{R}^{(1)}  &  &  & \\ \vdots  & & &  \\ - \frac{1}{2}\hat{n}_l^{(\nfac)}\mat{B}^{(\nfac)}\mat{R}^{(\nfac)}   &  &  &  ].
\end{equation}
Recognizing the matrix in \eqref{eq:s_hybridized} as the skew-symmetric part of $\mat{\bar{Q}}^{(l)}$, we invoke the SBP-like property from \cref{lem:hybridized_sbp_property} to obtain $
2\mat{\bar{S}}^{(l)} = 2\mat{\bar{Q}}^{(l)} - \mat{\bar{E}}^{(l)}$. Substituting such a relation into \eqref{eq:skew_hybridized_weak} and moving the inner sum outside the Hadamard product, the flux-differencing terms can be rewritten as
\begin{align}\label{eq:rewrite_flux_differencing}
\mqty[\mat{V}\\ \mat{V}^f ]^\T\sum_{l=1} ^d \bigg(2\mat{\bar{S}}^{(l)} \odot  \avg{\bar{G}_{lm}^{(\kappa)}} \odot \mat{\bar{F}}^{(\kappa,m,e)}(t) \bigg)\vc{1}^{(\bar{N}_q)} &=
\mqty[\mat{V}\\ \mat{V}^f ]^\T\sum_{l=1} ^d \bigg(2\mat{\bar{Q}}^{(l)} \odot \avg{\bar{G}_{lm}^{(\kappa)}} \odot \mat{\bar{F}}^{(\kappa,m,e)}(t) \bigg)\vc{1}^{(\bar{N}_q)}\notag 
\\ &- \sum_{\zeta=1}^{\nfac}\big(\mat{R}^{(\zeta)}\big)^\T \mat{B}^{(\zeta)}\mat{J}^{(\kappa,\zeta)}\mat{N}^{(\kappa,\zeta,m)}\vc{f}^{(\kappa,\zeta,m,e)}(t),
\end{align}
where the second term on the right-hand side results from the consistency of the two-point flux and the fact that $\mat{\bar{E}}^{(l)}$ is diagonal. 
Finally, we invoke \cref{lem:split_equiv_to_averaging} to express the first term on the right-hand side of \eqref{eq:rewrite_flux_differencing} using the split-form operator in \eqref{eq:phys_hybridized}. Substituting the result into \eqref{eq:skew_hybridized_weak} then yields the scheme in \eqref{eq:skew_hybridized}.
\end{proof}
As a consequence of the above equivalence, the conservation, free-stream preservation, and entropy stability of the proposed discretizations follow directly from the analysis in \cite{chan_discretely_entropy_conservative_dg_sbp_18} and \cite{ chan_wilcox_entropystable_curvilinear_19}. The proofs of such results rely on the assumption of a conforming mesh in the following sense.
\begin{assumption}\label{asm:conformal}
For each pair of element indices $\kappa,\nu \in \{1:\nelem \}$ with $\kappa \neq \nu$ such that $\partial\Omega^{(\kappa)} \cap \partial\Omega^{(\nu)} \neq \emptyset$, there exist a unique pair of facet indices $\zeta,\eta \in \{1:\nfac\}$ such that $\Gamma^{(\kappa,\zeta)} = \Gamma^{(\nu,\eta)}$. Furthermore, for every $i \in \{1:\nfacnodes\}$ there exists a unique $j \in \{1:N_{q_f}^{(\eta)}\}$ for which we have
\begin{equation}
\vec{n}^{(\kappa,\zeta)}(\fvec{x}^{(\kappa)}(\vec{\xi}^{(\zeta,i)})) = -\vec{n}^{(\nu,\eta)}(\fvec{x}^{(\nu)}(\vec{\xi}^{(\eta,j)})), \quad
\omega^{(\zeta,i)}\fn{J}^{(\kappa,\zeta)}(\vec{\xi}^{(\zeta,i)}) = \omega^{(\eta,j)}\fn{J}^{(\nu,\eta)}(\vec{\xi}^{(\eta,j)}), \label{eq:weights_match}
\end{equation}
where the normals are computed as in \eqref{eq:nanson} using the (exact or approximate) volume metric terms.
\end{assumption}
\begin{remark}\label{rmk:connectivity}
Sherwin and Karniadakis \cite[Section 2.3]{sherwin_karniadakis_tetrahedra_hp_fem_96} and Warburton \etal \cite{warburton_unstructured_connectivity_95} describe preprocessing algorithms for orienting local coordinate systems on tetrahedra so as to satisfy \cref{asm:conformal} despite the asymmetry of the tensor-product facet quadrature rules. Using such techniques, the proposed operators can be used with standard mesh generation tools without the need for nonconforming interface procedures. Otherwise, mortar-based techniques similar to those introduced by Chan \etal \cite{chan_bencomo_delrey_mortar_based_entropy_stable_21} could be employed in order to enable the use of symmetric nodal sets at element interfaces, thereby eliminating the need for a preprocessing step, although such an approach would increase the cost of applying the interpolation/extrapolation operators.
\end{remark}
Referring to \cite{chan_discretely_entropy_conservative_dg_sbp_18} and \cite{chan_wilcox_entropystable_curvilinear_19} for further details, we summarize the critical steps of the analysis in the remainder of this section, beginning with the following theorem establishing discrete conservation.
\begin{theorem}\label{thm:conservation}
Let \cref{asm:discretizations,asm:conformal} hold and also assume that the Jacobian determinant of the mapping, whether computed exactly or approximately as in \eqref{eq:jacobian_approx}, satisfies $\fn{J}^{(\kappa)} \in \mathbb{P}_{p} (\hat{\Omega})$ and that the quadrature rule in \eqref{eq:sbp_quadrature} induced by $\mat{W}$ is of degree $\tau \geq 2p$. The discretization given by \eqref{eq:skew_hybridized}, or, equivalently, by \eqref{eq:dudt_modal} with $\vc{r}^{(h,\kappa,e)}(t)$ defined as in \eqref{eq:entropy_stable}, is then discretely conservative, satisfying
\begin{equation}\label{eq:conservation}
\frac{\dd}{\dd t} \sum_{\kappa=1}^{\nelem} \big(\vc{1}^{(\nvolnodes)}\big)^\T \mat{W}\mat{J}^{(\kappa)} \vc{u}^{(h,\kappa,e)}(t) = -\sum_{\Gamma^{(\kappa,\zeta)} \subset \partial\Omega }\big(\vc{1}^{(\nfacnodes)}\big)^\T\mat{B}^{(\zeta)}\mat{J}^{(\kappa)}\vc{f}^{(*,\kappa,\zeta,e)}(t).
\end{equation} 
\end{theorem}
\begin{proof}
Letting the vector $\vc{\mathit{1}} \in \mathbb{R}^{\npoly}$ contain the modal expansion coefficients for the constant function such that $\mat{V}\vc{\mathit{1}} = \vc{1}^{(\nvolnodes)}$, we use the skew-symmetry of the matrices $\mat{\bar{S}}^{(l)} \odot \avg{\bar{G}_{lm}^{(\kappa)}} \odot \mat{\bar{F}}^{(\kappa,m,e)}$ in \eqref{eq:skew_hybridized_weak} and the exactness of $\mat{R}^{(\zeta)}$ for constant functions to obtain
\begin{equation}
\frac{\dd}{\dd t}\Big(\vc{\mathit{1}}^\T\mat{\tilde{M}}^{(\kappa)}\vc{\tilde{u}}^{(h,\kappa,e)}(t)\Big) = -\sum_{\zeta=1}^{\nfac}\big(\vc{1}^{(\nfacnodes)}\big)^\T\mat{B}^{(\zeta)}\mat{J}^{(\kappa)}\vc{f}^{(*,\kappa,\zeta,e)}(t),
\end{equation}
which is a statement of local conservation with respect to the weight-adjusted mass matrix. From \cite[Lemma 2]{chan_wilcox_entropystable_curvilinear_19}, the exactness of the volume quadrature when $\fn{U}_e^{(h,\kappa)}\fn{J}^{(\kappa)} \in \mathbb{P}_{2p}(\hat{\Omega})$ gives
\begin{equation}
\vc{\mathit{1}}^\T\mat{\tilde{M}}^{(\kappa)}\vc{\tilde{u}}^{(h,\kappa,e)}(t) = \int_{\hat{\Omega}} \fn{u}_e^{(h,\kappa)}(\vec{\xi},t) \fn{J}^{(\kappa)}(\vec{\xi}) \, \dd \vec\xi = \big(\vc{1}^{(\nvolnodes)}\big)^\T \mat{W}\mat{J}^{(\kappa)} \vc{u}^{(h,\kappa,e)}(t).
\end{equation}
Using the conservation property of the numerical interface flux and invoking \cref{asm:conformal} then results in the statement of global conservation in \eqref{eq:conservation}.
\end{proof}
Next, we present the following lemma proven in \cite[Theorem 5]{chan_wilcox_entropystable_curvilinear_19}, which establishes conditions under which the hybridized SBP operators on the physical element satisfy a discrete form of \eqref{eq:metric_identities}.
\begin{lemma}
Assume that the metric terms, whether computed exactly or approximately, satisfy $G_{mn}^{(\kappa)} \in \mathbb{P}_q(\hat{\Omega})$ as well as the metric identities in \eqref{eq:metric_identities}, and that the normals are computed as in \eqref{eq:nanson}. Then, the hybridized SBP operators in \eqref{eq:phys_hybridized} satisfy the following discrete metric identities:
\begin{equation}\label{eq:discrete_metric_identities}
\mat{\bar{Q}}^{(\kappa,m)}\vc{1}^{(\bar{N}_q)} = \vc{0}^{(\bar{N}_q)}, \quad \forall\, m \in \{1:d\}.
\end{equation}
\end{lemma}
The discrete metric identities are closely related to the \emph{free-stream preservation} property, in which a uniform solution state is guaranteed to remain constant in time. This is established for the proposed schemes with the following theorem.
\begin{theorem}\label{thm:freestream}
Let \cref{asm:discretizations,asm:conformal} hold and also assume that the discrete metric identities in \eqref{eq:discrete_metric_identities} are satisfied for all $\kappa \in \{1:\nelem\}$. The scheme in \eqref{eq:skew_hybridized}, or, equivalently, in \eqref{eq:dudt_modal} with $\vc{r}^{(h,\kappa,e)}(t)$ defined as in \eqref{eq:entropy_stable}, is then free-stream preserving, such that the right-hand side of \eqref{eq:dudt_modal} vanishes for all $\kappa \in \{1:\nelem\}$ and $e \in \{1:\ncons\}$ for any uniform solution state satisfying the boundary conditions.
\end{theorem}
\begin{proof}
Considering the formulation in \eqref{eq:skew_hybridized}, the facet penalty on the second line vanishes when the solution is identical on both sides of the interface due to the consistency property of the numerical interface flux. Invoking the consistency of the two-point flux as well, we then see that the entire right-hand side of \eqref{eq:skew_hybridized} vanishes when \eqref{eq:discrete_metric_identities} holds. Since the weight-adjusted mass matrix is invertible by construction, the time derivative is zero, and the scheme is therefore free-stream preserving.
\end{proof}
Invoking the entropy conditions in \eqref{eq:tadmor} and \eqref{eq:tadmor_directional}, we now present the following theorem, which establishes that the proposed discretizations satisfy a discrete version of the bound in \eqref{eq:entropy_balance}.
\begin{theorem}\label{thm:entropy_stability}
Let \cref{asm:discretizations,asm:conformal} hold, and also assume that the discrete metric identities in \eqref{eq:discrete_metric_identities} are satisfied for all $\kappa \in \{1:\nelem\}$, that the two-point flux used to compute \eqref{eq:two_point_flux_volume} and \eqref{eq:two_point_flux_facet} is entropy conservative in the sense of \cref{def:ec_flux}, and that the numerical interface flux is entropy stable in the sense of \cref{def:num_flux}. The discretization given by \eqref{eq:skew_hybridized}, or, equivalently, by \eqref{eq:dudt_modal} with $\vc{r}^{(h,\kappa,e)}(t)$ defined as in \eqref{eq:entropy_stable}, is then discretely entropy stable, satisfying the entropy balance  
\begin{equation}\label{eq:discrete_entropy_bound}
\begin{aligned}
\frac{\dd}{\dd t} \sum_{\kappa=1}^{\nelem} \big(\vc{1}^{(\nvolnodes)}\big)^\T \mat{W}\mat{J}^{(\kappa)} \vc{s}^{(h,\kappa)}(t) &\leq \sum_{\Gamma^{(\kappa,\zeta)} \subset \partial\Omega }\Bigg( \sum_{m=1}^d \big(\vc{1}^{(\nfacnodes)}\big)^\T\mat{B}^{(\zeta)}\mat{J}^{(\kappa,\zeta)}\mat{N}^{(\kappa,\zeta,m)}\vc{\psi}^{(\kappa,\zeta,m)}(t) \\ &- \sum_{e=1}^{\ncons}\big(\vc{w}^{(h,\kappa,e)}(t)\big)^\T\big(\mat{R}^{(\zeta)}\big)^\T\mat{B}^{(\zeta)}\mat{J}^{(\kappa,\zeta)}\vc{f}^{(*,\kappa,\zeta,e)}(t)\Bigg),
\end{aligned}
\end{equation} 
where the entries of the vectors $\vc{s}^{(h,\kappa)}(t) \in \mathbb{R}^{\nvolnodes}$ and $\vc{\psi}^{(\kappa,\zeta,m)}(t) \in \mathbb{R}^{\nfacnodes} $ are given, respectively, by
\begin{equation}
s_i^{(h,\kappa)}(t) \df \mathcal{S}(\vc{u}_i^{(h,\kappa)}(t)), \quad \psi_i^{(\kappa,\zeta,m)}(t) \df \varPsi_m(\vc{w}_i^{(h,\kappa,\zeta)}(t)).
\end{equation}
Moreover, the entropy balance in \eqref{eq:discrete_entropy_bound} holds as an equality when the numerical interface flux is entropy conservative, and the right-hand side vanishes for periodic boundary conditions.
\end{theorem}
\begin{proof}
The proof of entropy conservation for a non-dissipative interface flux is identical to that of \cite[Theorem 2]{chan_wilcox_entropystable_curvilinear_19}, wherein we left-multiply \eqref{eq:skew_hybridized} by $(\vc{w}^{(h,\kappa,e)}(t))^\T$, sum over $e \in \{1:\ncons\}$, and use the metric identities in \eqref{eq:discrete_metric_identities} as well as the SBP property in \eqref{eq:hybridized_sbp_phys} to obtain the local entropy balance
\begin{equation}\label{eq:local_entropy_balance}
\begin{aligned}
\frac{\dd}{\dd t}\Big(\big(\vc{1}^{(\nvolnodes)}\big)^\T \mat{W}\mat{J}^{(\kappa)} \vc{s}^{(h,\kappa)}(t)\Big) &= \sum_{m=1}^d \big(\vc{1}^{(\nfacnodes)}\big)^\T\mat{B}^{(\zeta)}\mat{J}^{(\kappa,\zeta)}\mat{N}^{(\kappa,\zeta,m)}\vc{\psi}^{(\kappa,\zeta,m)}(t) \\ & - \sum_{e=1}^{\ncons}\sum_{\zeta=1}^{\nfac}\big(\vc{w}^{(h,\kappa,\zeta,e)}(t)\big)^\T\mat{B}^{(\zeta)}\mat{J}^{(\kappa,\zeta)}\vc{f}^{(*,\kappa,\zeta,e)}(t).
\end{aligned}
\end{equation}
For an entropy-conservative interface flux, summing \eqref{eq:local_entropy_balance} over all elements and splitting the interface contributions between adjacent elements results in a global statement of entropy conservation, corresponding to \eqref{eq:discrete_entropy_bound} being satisfied as an equality, where the boundary contributions vanish similarly to the interior interface contributions for periodic problems. The entropy inequality for an entropy-stable interface flux then follows in a straightforward manner, for example, from the analysis in \cite[Theorems 3.4 and 4.3]{chen_shu_entropy_stable_dgsbp_17}.
\end{proof}

\section{Numerical experiments}\label{sec:numerical}
We now present numerical experiments assessing the accuracy and robustness of entropy-conservative and entropy-stable DSEMs using tensor-product as well as multidimensional SBP operators on triangles and tetrahedra through the numerical solution of the Euler equations using \texttt{StableSpectralElements.jl}. The parameters used to run such simulations as well as the Jupyter notebooks used to generate the figures appearing in this section are provided in this paper's reproducibility repository, which is available at \url{https://github.com/tristanmontoya/ReproduceEntropyStableDSEM}.

\subsection{Euler equations}
The Euler equations constitute a system of $\ncons = d + 2$ coupled nonlinear PDEs governing the conservation of mass, momentum, and energy for a compressible, inviscid, and adiabatic fluid. Such a system takes the form of \eqref{eq:cons_law}, with the solution variables and flux components given by
\begin{equation}\label{eq:euler}
\vc{\fn{u}}(\vec{x},t) \df \mqty[\rho(\vec{x},t) \\ \rho(\vec{x},t) V_1(\vec{x},t) \\ \vdots \\ \rho(\vec{x},t) V_d(\vec{x},t) \\ \fn{E}(\vec{x},t)], \quad \vc{\fn{f}}_m(\vc{\fn{u}}(\vec{x},t)) \df \mqty[\rho(\vec{x},t)\fn{v}_m(\vec{x},t) \\ \rho(\vec{x},t)\fn{v}_1(\vec{x},t)\fn{v}_m(\vec{x},t)  + \fn{p}(\vec{x},t)\delta_{1m} \\ \vdots \\ \rho(\vec{x},t)\fn{v}_d(\vec{x},t)\fn{v}_m(\vec{x},t) + \fn{p}(\vec{x},t)\delta_{dm}  \\ \fn{v}_m(\vec{x},t)(\fn{E}(\vec{x},t) + \fn{p}(\vec{x},t))],
\end{equation}
where $\rho(\vec{x},t) \in \mathbb{R}$ denotes the fluid density, $\fvec{v}(\vec{x},t) \in \mathbb{R}^d$ denotes the flow velocity, $\fn{E}(\vec{x},t) \in \mathbb{R}$ denotes the total energy per unit volume, and $\fn{p}(\vec{x},t) \in \mathbb{R}$ denotes the pressure. We obtain the pressure using the equation of state for an ideal gas with constant specific heat,
\begin{equation}
\fn{p}(\vec{x},t) = (\gamma-1)\bigg(\fn{E}(\vec{x},t) - \frac{1}{2}\rho(\vec{x},t)\lVert \fvec{v}(\vec{x},t)\rVert^2\bigg),
\end{equation}
where $\gamma > 1$ is the specific heat ratio, which we take as $1.4$ for air in all numerical experiments. The Euler equations are hyperbolic for solutions belonging to the admissible set 
\begin{equation}
\Upsilon \df \big\{ \vc{\fn{u}}(\vec{x},t) \in \mathbb{R}^{d+2} : \fn{p}(\vec{x},t),\rho(\vec{x},t) > 0 \big\}.
\end{equation}
Although there exist many entropy--entropy flux pairs which satisfy the conditions of \cref{def:entropy} for the Euler equations, we restrict our attention to the choice of
\begin{equation}\label{eq:euler_entropy}
\mathcal{S}(\vc{U}(\vec{x},t)) \df - \frac{\rho(\vec{x},t)}{\gamma - 1} \ln(\frac{P(\vec{x},t)}{\rho(\vec{x},t)^\gamma}),\quad
\vec{\mathcal{F}}(\vc{U}(\vec{x},t)) \df - \frac{\rho(\vec{x},t)\vec{\fn{v}}(\vec{x},t) }{\gamma - 1} \ln(\frac{P(\vec{x},t)}{\rho(\vec{x},t)^\gamma}),
\end{equation}
which also symmetrizes the viscous terms of the compressible Navier--Stokes equations with heat conduction, and was shown by Hughes \etal \cite{hughes_franca_mallet_entropy_stable_fem_86} to be the unique member (up to an affine transformation) of Harten's family of entropy--entropy flux pairs \cite{harten_entropy_symmetric_form_83} to do so.
\subsection{Entropy-conservative and entropy-stable flux functions}
To obtain an entropy-conservative two-point flux in the sense of \cref{def:ec_flux} with respect to the entropy function and entropy flux in \eqref{eq:euler_entropy},
we first define the arithmetic mean and logarithmic mean, respectively, as
\begin{equation}
\avg{a} \df \frac{1}{2}\big(a^- + a^+\big), \quad \avg{a}_{\mathrm{ln}} \df \begin{cases}\dfrac{a^+ - a^-}{\ln(a^+) - \ln(a^-)},  &a^- \neq a^+\\
a^-, \quad &a^- = a^+
\end{cases},
\end{equation}
where we use Taylor-series approximations from Ranocha \etal \cite[Algorithms 2 and 3]{ranocha_efficient_implementation_23} to compute the logarithmic mean and its reciprocal in cases for which $a^-$ and $a^+$ are nearly equal. The particular entropy-conservative flux used in this work was proposed by Ranocha \cite{ranocha_phd_thesis_18,ranocha_ec_kep_20}, and is given by
\begin{equation}\label{eq:ranocha_flux}
\vc{\fn{f}}_m^\#(\vc{U}^-,\vc{U}^+) \df \mqty[\avg{\rho}_{\mathrm{ln}} \avg{V_m}  \\\avg{\rho}_{\mathrm{ln}} \avg{V_m} \avg{V_1} + \avg{P}\delta_{1m}\\ \vdots \\ \avg{\rho}_{\mathrm{ln}} \avg{V_m}\avg{V_d} + \avg{P}\delta_{dm} \\ \tfrac{1}{2}\avg{\rho}_{\mathrm{ln}} \avg{V_m} \big( \vec{V}^- \cdot \vec{V}^+ + \tfrac{1}{\gamma-1}\avg{\rho/P}_{\mathrm{ln}}^{-1}\big)+ \tfrac{1}{2}\big(P^-V_m^+ + P^+V_m^-\big)
].
\end{equation}
In addition to the entropy conservation property in \eqref{eq:tadmor}, Ranocha's flux is kinetic energy preserving and pressure equilibrium preserving (see, for example, Ranocha and Gassner \cite{ranocha_pressure_oscillations_22}). The interface flux takes the form of \eqref{eq:num_flux}, where we use Davis's wave speed estimate \cite{davis_godunov_88}, which is given by
\begin{equation}\label{eq:wave_speed}
\lambda(\vc{U}^-,\vc{U}^+, \vec{n}) \df \max\Big(\big\lvert\vec{V}^- \cdot \vec{n}\big\rvert, \big\lvert\vec{V}^+ \cdot \vec{n}\big\rvert\Big) + \max\Big(\sqrt{\gamma P^- / \rho^-}, \sqrt{\gamma P^+ / \rho^+}\Big).
\end{equation}
In the results which are to follow, we refer to the schemes employing local Lax--Friedrichs dissipation as \emph{entropy-stable} methods. We also implement a variant without dissipation, in which the second term on the right-hand side of \eqref{eq:num_flux} is absent; such schemes are denoted as \emph{entropy-conservative} methods.

\subsection{Curvilinear meshes}

The problems considered in this work are defined on the spatial domain $\Omega \df (0,L)^d$, where $L \in \mathbb{R}^+$ and $d \in \{2,3\}$. The mesh is generated by beginning with a Cartesian grid with $M$ edges in each direction and splitting each quadrilateral into two triangles or each hexahedron into six tetrahedra, resulting in $\nelem = 2 M^2$ in two dimensions and $\nelem = 6 M^3$ in three dimensions. Recalling \cref{rmk:connectivity}, we use the second algorithm described in \cite[Section 2.3]{sherwin_karniadakis_tetrahedra_hp_fem_96}, which was originally proposed by Warburton \etal \cite{warburton_unstructured_connectivity_95}, to orient the local coordinate systems on each element in order to ensure that \cref{asm:conformal} is satisfied for tetrahedral meshes. For both the triangular and tetrahedral case, we use isoparametric mappings, corresponding to \eqref{eq:poly_mapping} with the choice of $p_g = p = q$, where the mapping nodes are obtained using the interpolatory warp-and-blend procedure from Chan and Warburton \cite{chan_warburton_interpolation_nodes_15}. An affine transformation is used to obtain the positions of the mapping nodes on each element of the split Cartesian mesh. Following \cite[Section 5]{chan_delrey_carpenter_gauss_collocation_19}, the mesh is then warped by perturbing the mapping node positions as
\begin{equation}
\begin{aligned}
\tilde{x}_1 &\gets x_1 + \varepsilon L \cos(\tfrac{\pi}{L}\big(x_1 - \tfrac{1}{2}\big) )\cos(\tfrac{3\pi}{L}\big(x_2 - \tfrac{1}{2} \big) ),\\
\tilde{x}_2 &\gets x_2 + \varepsilon L \sin(\tfrac{4\pi}{L}\big(\tilde{x}_1 - \tfrac{1}{2} \big) )\cos(\tfrac{\pi}{L}\big(x_2 - \tfrac{1}{2}\big) ),
\end{aligned}
\end{equation}
in two dimensions, and as
\begin{equation}
\begin{aligned}
\tilde{x}_2 &\gets x_2 + \varepsilon L\cos(\tfrac{3\pi}{L}\big(x_1 - \tfrac{1}{2} \big) )\cos(\tfrac{\pi}{L}\big(x_2 - \tfrac{1}{2}\big) )\cos(\tfrac{\pi}{L}\big(x_3 - \tfrac{1}{2} \big) ),\\
\tilde{x}_1 &\gets x_1 + \varepsilon L \cos(\tfrac{\pi}{L}\big(x_1 - \tfrac{1}{2} \big) )\sin(\tfrac{4\pi}{L}\big(\tilde{x}_2 - \tfrac{1}{2}\big) ) \cos(\tfrac{\pi}{L}\big(x_3 - \tfrac{1}{2} \big) ),\\
\tilde{x}_3 &\gets x_3 + \varepsilon L \cos(\tfrac{\pi}{L}\big(\tilde{x}_1 - \tfrac{1}{2} \big)) \cos(\tfrac{2\pi}{L}\big(\tilde{x}_2 - \tfrac{1}{2}\big) ) \cos(\tfrac{\pi}{L}\big(x_3 - \tfrac{1}{2} \big)),
\end{aligned}
\end{equation}
in three dimensions, where, as in \cite[Section 7.1]{montoya_sem_23}, we take $\varepsilon = 1/16$ in both cases. The new node positions $\tilde{\vec{x}}$ are then used to define the mapping in \eqref{eq:poly_mapping}. Finally, the metric terms are computed using the approach described in \cref{sec:metrics}, where in the three-dimensional case we use the conservative curl formulation in \eqref{eq:conservative_curl}, with the nodes used for the interpolation of degree $q+1$ in \eqref{eq:metric_interpolants} again obtained using the interpolatory warp-and-blend procedure from \cite{chan_warburton_interpolation_nodes_15}. Examples of curvilinear meshes and the mapping nodes used to obtain such meshes are shown in \cref{fig:curved_meshes}.
\begin{figure}[t!]
\begin{subfigure}{0.325\textwidth}
\centering
\includegraphics[height=42mm]{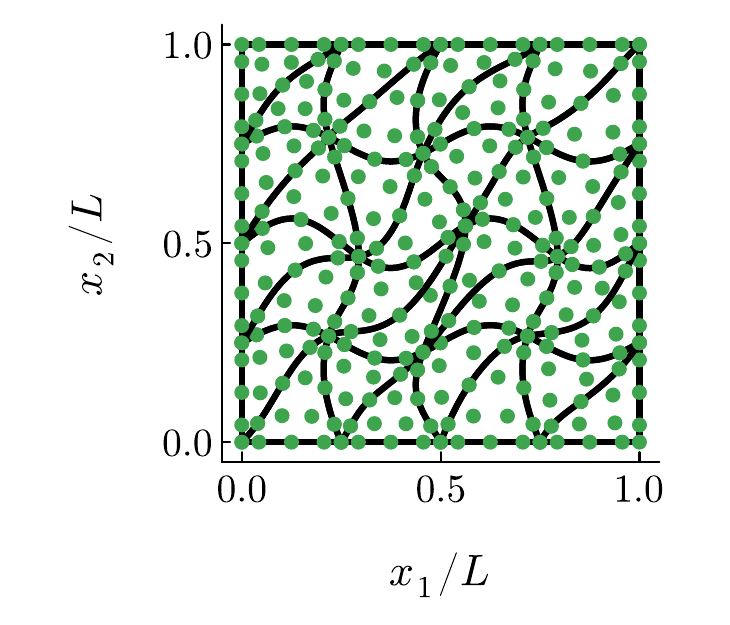}
\caption{Triangular mesh}
\end{subfigure}~
\begin{subfigure}{0.325\textwidth}
\centering
\includegraphics[height=42mm]{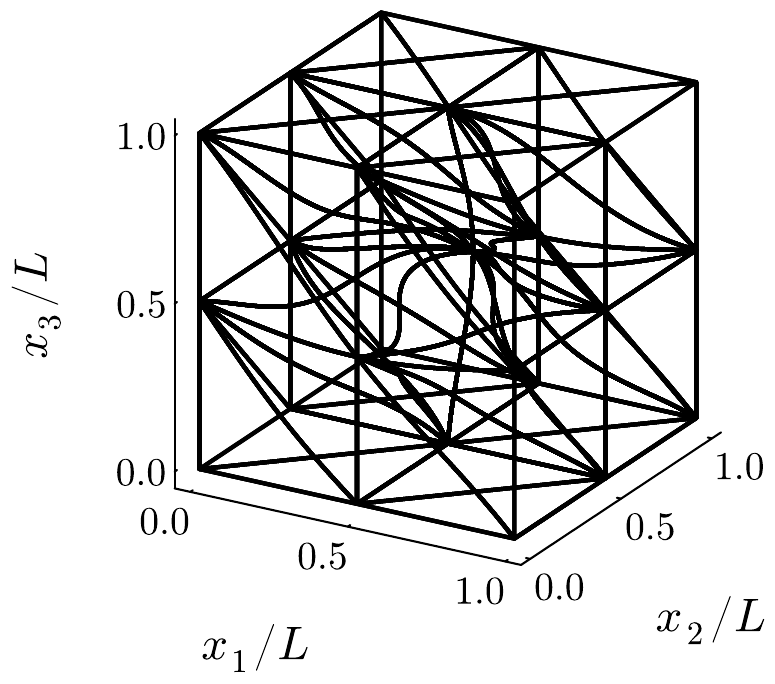}
\caption{Tetrahedral mesh}
\end{subfigure}
\begin{subfigure}{0.325\textwidth}
\centering
\includegraphics[height=42mm]{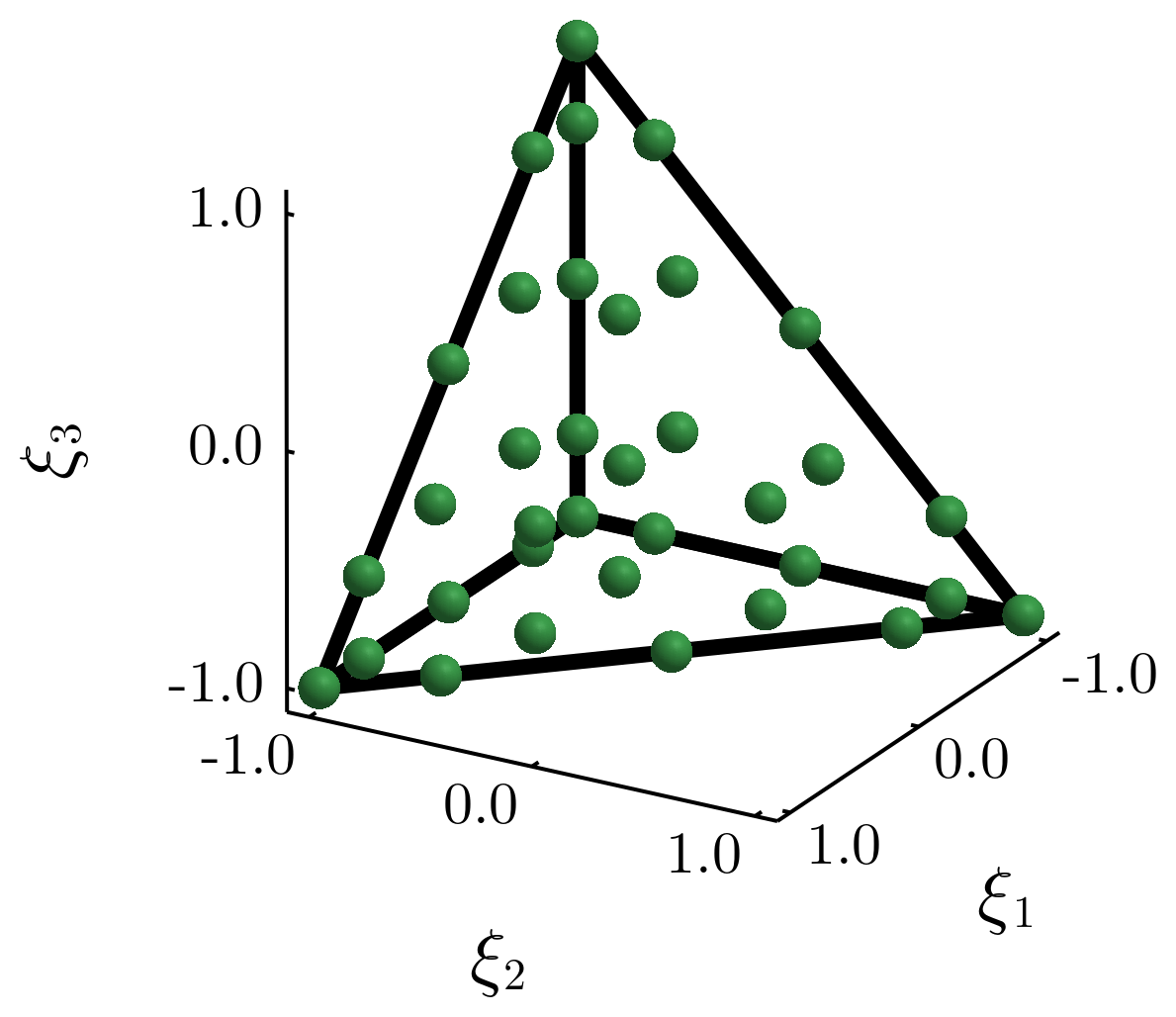}
\caption{Tetrahedral mapping nodes}
\end{subfigure}
\caption{Examples of warped meshes and mapping nodes for $p_g = 4$}\label{fig:curved_meshes}
\end{figure}

\subsection{Accuracy tests}
We assess the accuracy of the proposed entropy-conservative and entropy-stable discretizations of the Euler equations under refinement with respect to the element size $h$ as well as the polynomial degree $p$ in the context of smooth problems with known analytical solutions. The initial conditions are prescribed in terms of the primitive variables $\rho_0(\vec{x})$, $\fvec{v}_0(\vec{x})$, and $P_0(\vec{x})$, which are used to obtain the right-hand side of \eqref{eq:ic} as 
\begin{equation}
\vc{\fn{u}}^0(\vec{x}) \df \mqty[ \rho_0(\vec{x})\\  \rho_0(\vec{x}) \vec{V}_0(\vec{x})\\ \tfrac{1}{\gamma-1} P_0(\vec{x}) + \tfrac{1}{2}\rho_0(\vec{x}) \lVert \fvec{v}_0(\vec{x}) \rVert^2].
\end{equation}
We consider the smooth density wave problem from Jiang and Shu \cite[Section 7]{jiang_shu_efficient_implementation_weno_96}, for which the initial values of the primitive variables are given by
\begin{equation}
\rho_0(\vec{x}) \df 1 + \frac{1}{5}\sin\Bigg(\frac{2\pi}{L}\sum_{m=1}^d x_m \Bigg), \quad \vec{V}_0(\vec{x}) \df \Big[1, \ldots, 1\Big]^\T, \quad P_0(\vec{x}) \df 1,
\end{equation}
on the domain $\Omega \df (0,2)^d$, with periodic boundary conditions applied in all directions. The solution is advanced in time for one period of wave propagation (i.e.\ until a final time of $T = 2$) using a Julia implementation \cite{rackauckas_julia_diffeq_17} of the eighth-order Dormand--Prince algorithm described in \cite[Section II.5]{hairer_93}, with the time step taken sufficiently small for the temporal discretization error to be negligible in comparison to that due to the spatial discretization. \par
\begin{figure}[t!]
\centering
\begin{subfigure}{0.325\textwidth}
\includegraphics[height=42mm]{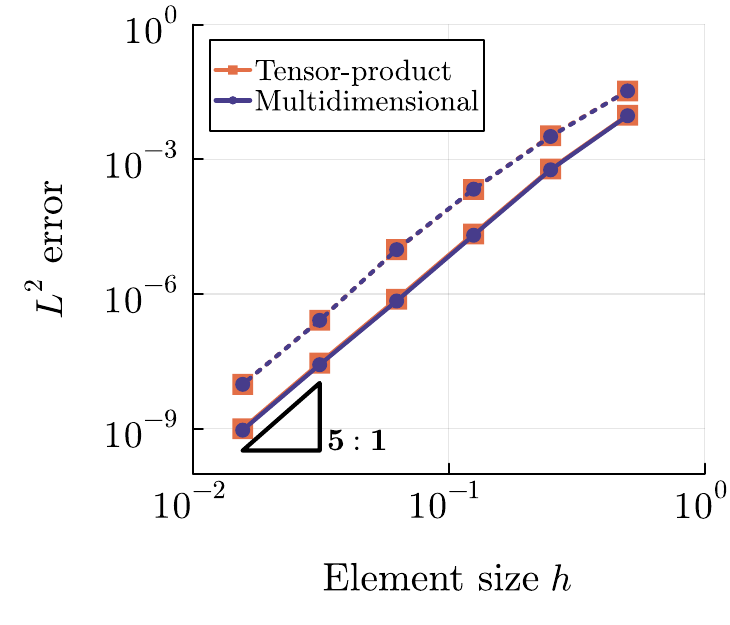}
\caption{$h$-refinement on triangles, $p=4$}
\end{subfigure}~
\begin{subfigure}{0.325\textwidth}
\includegraphics[height=42mm]{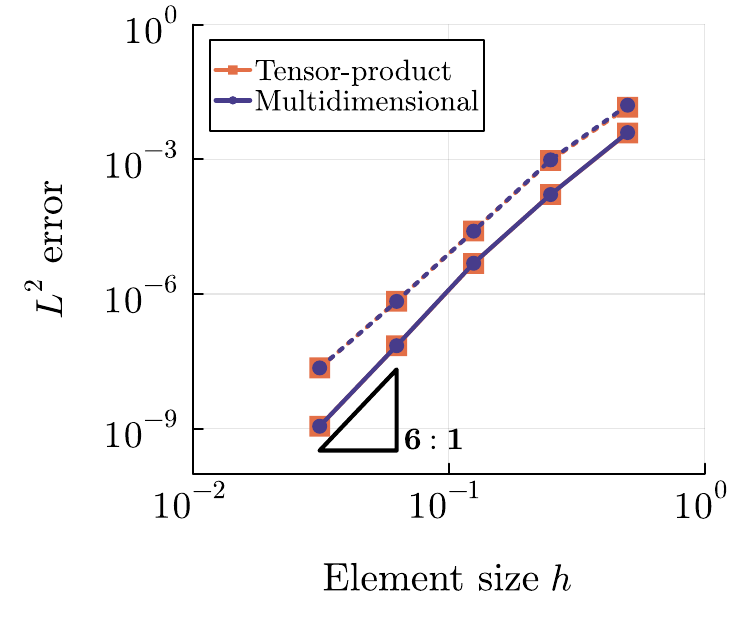}
\caption{$h$-refinement on triangles, $p=5$}
\end{subfigure}~
\begin{subfigure}{0.325\textwidth}
\includegraphics[height=42mm]{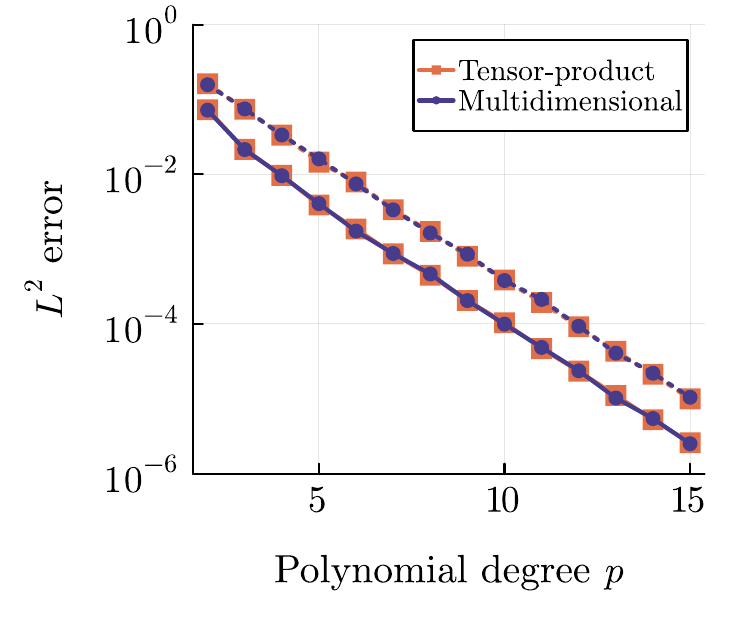}
\caption{$p$-refinement on triangles, $M=4$}
\end{subfigure}

\bigskip
\begin{subfigure}{0.325\textwidth}
\includegraphics[height=42mm]{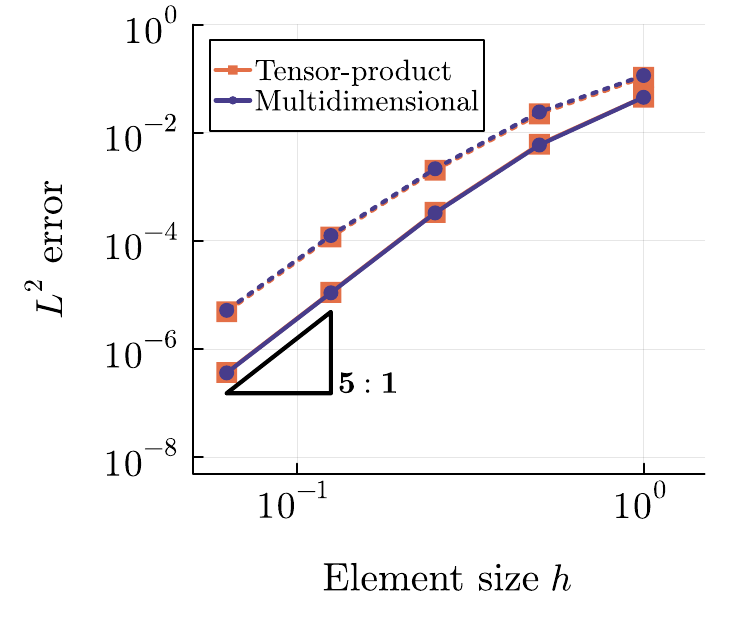}
\caption{$h$-refinement on tetrahedra, $p=4$}
\end{subfigure}~
\begin{subfigure}{0.325\textwidth}
\includegraphics[height=42mm]{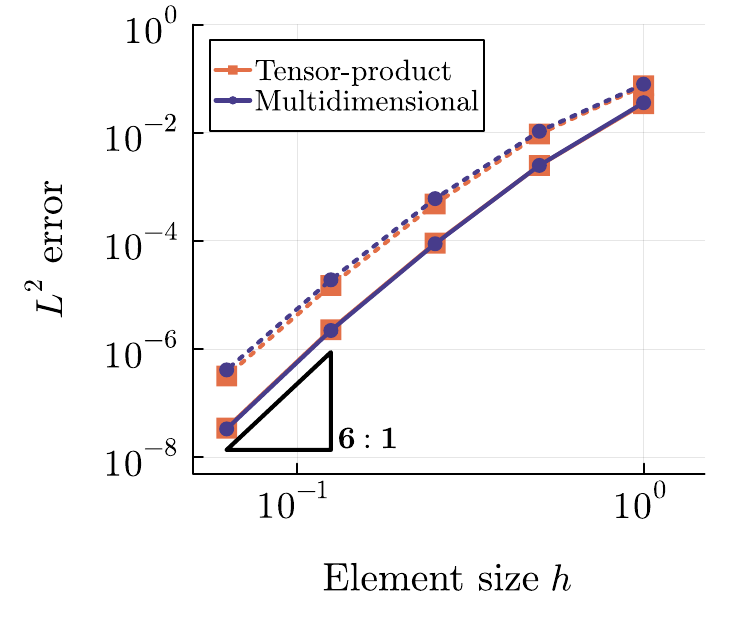}
\caption{$h$-refinement on tetrahedra, $p=5$}
\end{subfigure}~
\begin{subfigure}{0.325\textwidth}
\includegraphics[height=42mm]{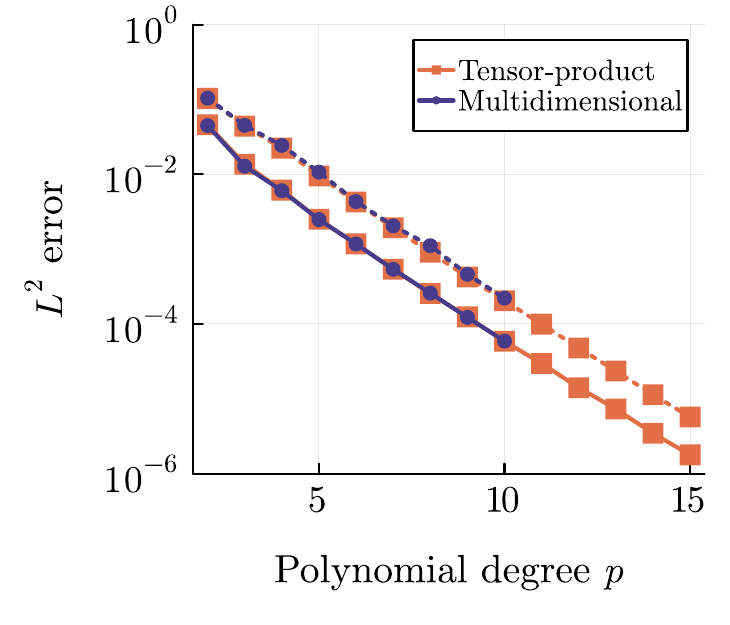}
\caption{$p$-refinement on tetrahedra, $M=4$}
\end{subfigure}

\caption{Convergence with respect to $h$ and $p$ for discretizations of the Euler equations; dashed and solid lines denote density error for entropy-conservative and entropy-stable schemes, respectively}\label{fig:accuracy}
\end{figure}

Convergence is examined with respect to the nominal element size (taken here to be $h \df L/M$) as well as the polynomial degree $p$ for entropy-conservative and entropy-stable DSEMs using the tensor-product and multidimensional SBP operators on triangles and tetrahedra considered in \cref{sec:implementation}. The degree $p$ of the solution expansion in \eqref{eq:poly_approx} is taken to be equal to the degree $q$ of the SBP operators, and we report the $L^2$ norm of the density error, which is computed numerically using a quadrature rule of degree 35. Similar convergence behaviour was observed for the other solution variables. \Cref{fig:accuracy} demonstrates optimal $\mathcal{O}(h^{p+1})$ algebraic convergence under $h$-refinement for the entropy-stable discretizations (i.e.\ those including interface dissipation) as well as exponential convergence under $p$-refinement, where we recall from \cref{sec:implementation} that the multidimensional SBP operators on tetrahedra considered here are only available for degrees up to ten. For a given mesh and polynomial degree, the error norms obtained for the proposed tensor-product discretizations are found to be very close to those obtained for their multidimensional counterparts, which suggests that their algorithmic advantages discussed in \cref{sec:implementation} with respect to the ability to exploit operator sparsity and sum factorization do not come at the expense of accuracy.

\begin{remark}
The numerical results are consistent with the conservation property established in \cref{thm:conservation}, with the time derivative on the left-hand side of \eqref{eq:conservation} remaining close to machine precision for all solution variables at approximately 100 equispaced snapshots taken during each test. Furthermore, the rates of entropy dissipation given by the left-hand side of \eqref{eq:discrete_entropy_bound} are verified for the entropy-conservative and entropy-stable schemes to be zero and non-positive, respectively, at all snapshots (up to roundoff error levels close to machine precision), as is consistent with \cref{thm:entropy_stability}.
\end{remark}

\subsection{Robustness tests}
High-order methods such as DSEMs are prone to numerical stability issues, particularly in the context of under-resolved nonlinear problems, in which high-frequency numerical modes are produced and potentially amplified by the discretization, resulting in non-physical blow-up or negative values of thermodynamic quantities such as pressure or density (i.e.\ corresponding to solutions outside the admissible set $\Upsilon$). Such under-resolution commonly arises in simulations of turbulent fluid flow problems, which are characterized by the cascade of energy from larger eddies to those progressively smaller in size, eventually dissipating as heat due to the viscosity of the fluid (see, for example, Pope \cite[Chapter 6]{pope_turbulent_flows}). As the Euler equations do not model physical viscosity, the only dissipation present is that inherent in the numerical scheme, which is typically small for a high-order method, thus exposing the potentially destabilizing effects of under-resolved eddies. We therefore use simulations of inviscid vortical flows to mimic worst-case scenarios with respect to under-resolved turbulence, where we are interested in assessing whether the simulations run to completion and whether the semi-discrete entropy bounds established in \cref{sec:analysis} are satisfied, rather than in evaluating the accuracy of such simulations. In the two-dimensional case, we consider the Kelvin--Helmholtz instability (KHI) problem described by Rueda-Ram\'irez and Gassner \cite{rueda_ramirez_fv_limiter_21} and used for robustness tests by Chan \etal \cite{chan_entropy_projection_22}, for which we define the smoothed step function
\begin{equation}
\fn{B}(\vec{x}) \df \tanh(15(x_2 - \tfrac{1}{2})) - \tanh(15(x_2 - \tfrac{3}{2}))
\end{equation}
in order to obtain the initial condition, which is given in terms of the primitive variables by
\begin{equation}
\rho_0(\vec{x}) \df \frac{1}{2} + \frac{3}{4}B(\vec{x}), \quad \vec{V}_0(\vec{x}) \df \bigg[\frac{1}{2}\big(B(\vec{x})-1), \, \frac{1}{10}\sin(2\pi x_1)\bigg]^\T, \quad P_0(\vec{x}) \df 1,
\end{equation}
on the domain $\Omega \df (0,2)^2$, with periodic boundary conditions in both directions. As in \cite{chan_entropy_projection_22}, we integrate until a final time of $T = 15$. In three dimensions, we consider an inviscid Taylor--Green vortex (TGV) problem, for which the initial condition is given on the periodic domain $\Omega \df (0,2\pi)^3$ as
\begin{equation}
\begin{multlined}
\rho_0(\vec{x}) \df 1, \quad \vec{V}_0(\vec{x}) \df \big[\sin(x_1)\cos(x_2)\cos(x_3),\, -\cos(x_1)\sin(x_2)\cos(x_3),\, 0\big]^\T,\\
P_0(\vec{x}) \df \frac{1}{\gamma \mathrm{Ma}^2} + \frac{1}{16}\Big(\cos(2x_1) + 2\cos(2x_2) + \cos(2x_1)\cos(2x_3) + \cos(2x_2)\cos(2x_3)\Big),
\end{multlined}
\end{equation}
where $\mathrm{Ma} \in \mathbb{R}^+$ is the nominal Mach number. We run the TGV simulations until a final time of $T = 14$, and, as in Pazner and Persson \cite{pazner_persson_entropy_stable_line_dg_19}, we consider the nearly incompressible case of $\mathrm{Ma} = 0.1$ in addition to the case of $\mathrm{Ma} = 0.7$, where the latter is expected to pose a greater challenge to the robustness of the proposed methods, specifically with respect to positivity preservation. \par
\begin{figure}[t!]
\centering
\begin{subfigure}{0.325\textwidth}
\includegraphics[height=42mm]{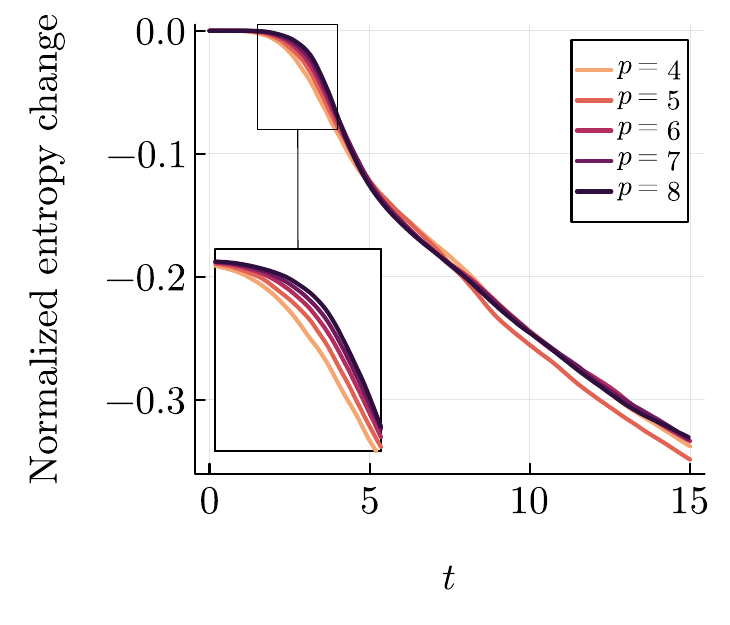}
\caption{Kelvin--Helmholtz instability, tensor-product operators on triangles}
\end{subfigure}~
\begin{subfigure}{0.325\textwidth}
\includegraphics[height=42mm]{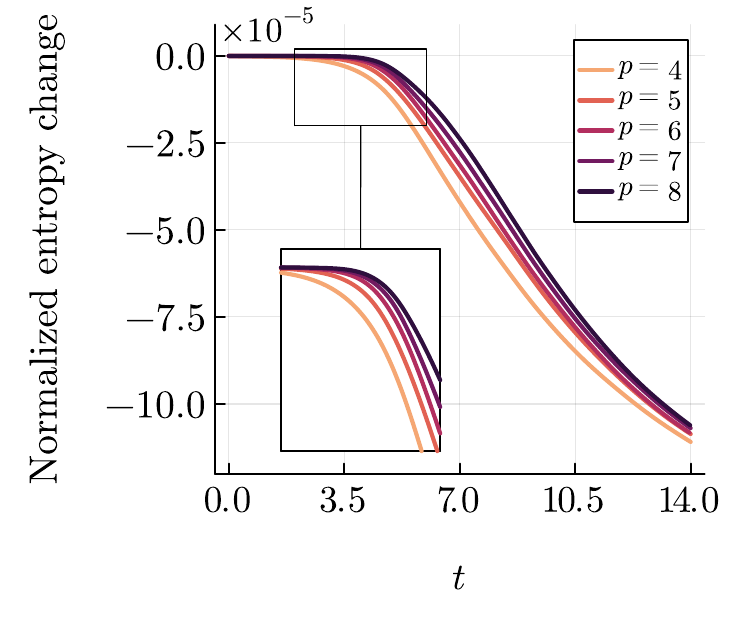}
\caption{Taylor--Green vortex, $\mathrm{Ma} = 0.1$, tensor-product operators on tetrahedra}
\end{subfigure}~
\begin{subfigure}{0.325\textwidth}
\includegraphics[height=42mm]{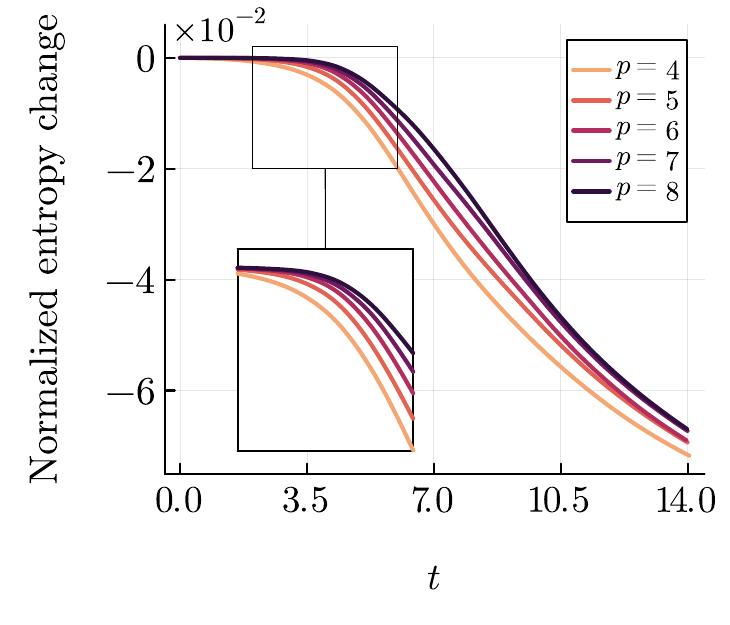}
\caption{Taylor--Green vortex, $\mathrm{Ma} = 0.7$, tensor-product operators on tetrahedra}
\end{subfigure}

\bigskip
\begin{subfigure}{0.325\textwidth}
\includegraphics[height=42mm]{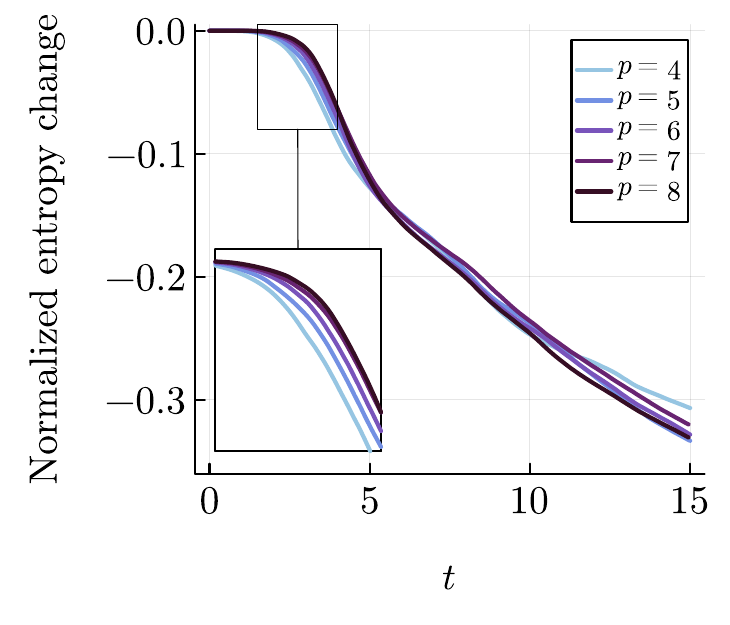}
\caption{Kelvin--Helmholtz instability, multidimensional operators on triangles}
\end{subfigure}~
\begin{subfigure}{0.325\textwidth}
\includegraphics[height=42mm]{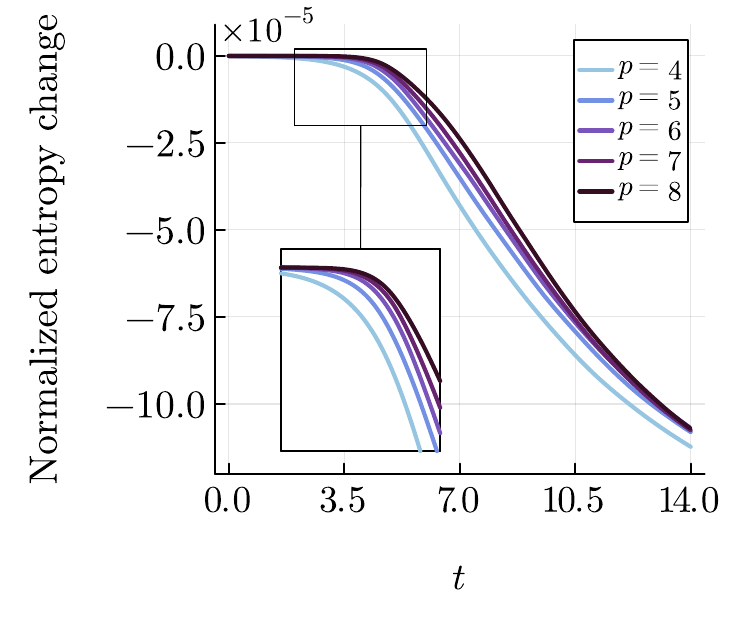}
\caption{Taylor--Green vortex, $\mathrm{Ma} = 0.1$, multidimensional operators on tetrahedra}
\end{subfigure}~
\begin{subfigure}{0.325\textwidth}
\includegraphics[height=42mm]{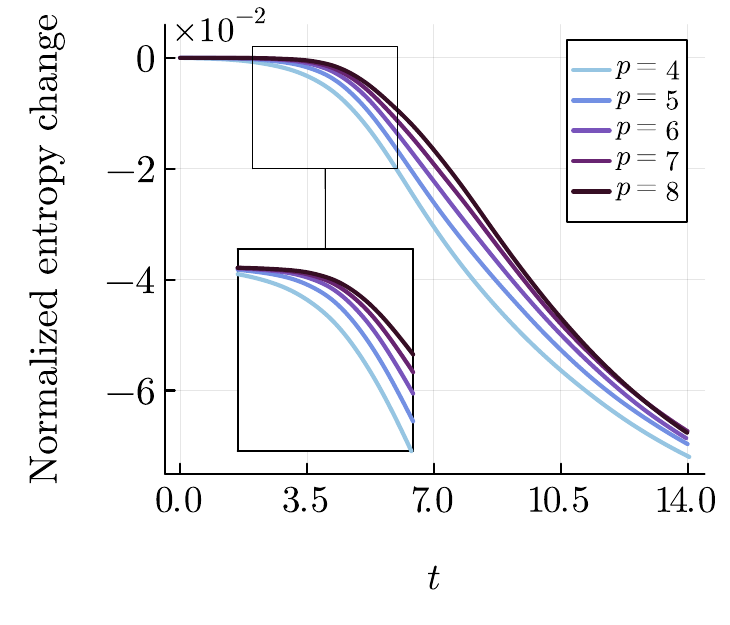}
\caption{Taylor--Green vortex, $\mathrm{Ma} = 0.7$, multidimensional operators on tetrahedra }
\end{subfigure}
\caption{Normalized entropy change for entropy-stable discretizations of the Euler equations on triangles and tetrahedra\label{fig:entropy}} 
\end{figure}
The Euler equations are solved for the above initial conditions using the proposed entropy-conservative and entropy-stable DSEMs for polynomial degrees 4 to 8, taking $M = 16$ for the KHI problem and $M = 4$ for the TGV problem. We integrate in time using the same explicit eighth-order Dormand--Prince method used for the accuracy tests. The time step is taken to be sufficiently small to ensure that the reported instances of instability result purely from the spatial discretization, and we did not find any of the reported instabilities to be remedied by decreasing the time step size. Without any additional stabilization beyond the interface dissipation provided by the numerical flux in \eqref{eq:num_flux}, all simulations ran to completion for the entropy-stable schemes using multidimensional as well as tensor-product operators on triangles and tetrahedra. \Cref{fig:entropy} demonstrates that the entropy is nonincreasing for all time in each of such cases, as expected from \cref{thm:entropy_stability} for periodic boundary conditions. \par 
While the entropy-conservative simulations ran to completion for the TGV with $\mathrm{Ma=0.1}$, incurring a change in entropy close to machine precision, such computations crashed for the TGV with $\mathrm{Ma=0.7}$ as well as for the KHI due to negative densities or pressures. This reflects a well-known limitation of the entropy analysis, namely that entropy stability does not guarantee positivity of thermodynamic quantities. Although not provably positivity preserving, the entropy-stable schemes using the dissipative interface flux in \eqref{eq:num_flux} did not incur negative densities or pressures for any of the tests considered in this work, likely as a consequence of the dissipative term within the numerical flux serving to dampen any oscillations which would otherwise eventually lead to a violation of positivity. As a result, such schemes are highly robust even in the presence of substantially under-resolved solution features. These results are consistent with the observations in \cite{chan_entropy_projection_22}, where the authors demonstrated numerically that entropy-stable schemes which incorporate an entropy projection are often able to avoid negative densities or pressures for challenging under-resolved problems without the need for positivity-preserving limiters. We recognize, however, that such an approach may not be sufficient to preserve positivity for problems with discontinuities, and the extension of subcell limiting techniques such as those in \cite{rueda_ramirez_subcell_limiting_22,yamaleev_positivity_entropy_stable_23,lin_positivity_preservation_euler_23} to the proposed tensor-product discretizations on triangles and tetrahedra is an important topic of future research.

\section{Conclusions}\label{sec:conclusions}
We have developed discretizations of arbitrary order for nonlinear hyperbolic systems of conservation laws which combine the geometric flexibility of curved triangular and tetrahedral elements with the efficiency of tensor-product operators as well as the robustness of a provably entropy-stable formulation. 
The main components of our proposed methodology are outlined below.
\begin{itemize}
\item The numerical solution is represented using an orthonormal PKD polynomial basis on the reference triangle or tetrahedron, which supports the efficient matrix-free evaluation of polynomials at tensor-product quadrature nodes in collapsed coordinates through sum factorization.
\item The flux-differencing volume terms of the DSEM are computed using sparse tensor-product SBP operators in collapsed coordinates, significantly reducing the number of required entropy-conservative two-point flux evaluations relative to discretizations using multidimensional SBP operators based on symmetric quadrature rules, particularly at high polynomial degrees.
\item A weight-adjusted approximation is used to invert the dense local mass matrices arising from the use of curved elements, both in the evaluation of the time derivative and in the entropy projection, allowing for the spatial residual to be obtained in $\mathcal{O}(p^{d+1})$ floating-point operations, comparing favourably to the $\mathcal{O}(p^{2d})$ complexity of existing multidimensional formulations.
\item The discrete metric identities, which must hold to obtain an entropy-stable and free-stream preserving scheme, are satisfied by approximating the metric terms in conservative curl form.
\end{itemize}
This paper provides a comprehensive description of the mathematical and algorithmic aspects of the proposed tensor-product DSEMs on triangles and tetrahedra, discussing their formulation and efficient implementation as well as proving that they are conservative, free-stream preserving, and entropy stable by rewriting them in an equivalent formulation based on hybridized SBP operators. Furthermore, through the solution of the compressible Euler equations on curved triangular and tetrahedral meshes, we numerically verify that the proposed DSEMs are conservative as well as entropy conservative or dissipative (for discretizations with or without interface dissipation, respectively). We also verify the schemes' convergence properties for smooth solutions under $h$-refinement as well as $p$-refinement, providing numerical evidence that the proposed discretizations using tensor-product operators offer very similar accuracy to those using multidimensional SBP operators for such problems. Finally, we demonstrate the robustness of our approach in a suite of challenging under-resolved problems. Future work includes the extension of the proposed methodology to prismatic elements as well as to problems with diffusive terms (e.g.\ the Navier--Stokes equations) and those in non-conservative form (e.g.\ multiphase atmospheric flows). We are also interested in the development of positivity-preserving subcell limiters for the proposed schemes as well as in practical comparisons of accuracy, efficiency, and robustness relative to other entropy-stable DSEMs as well as to those using other techniques such as over-integration to achieve robustness in practice.

\section*{Acknowledgements}
This work was supported by the Natural Sciences and Engineering Research Council of Canada (NSERC), the Ontario Graduate Scholarship, and the University of Toronto. The numerical experiments were performed using the Niagara cluster at the SciNet HPC Consortium \cite{ponce_niagara_19}, which is funded by the Canada Foundation for Innovation, the Government of Ontario, the Ontario Research Fund -- Research Excellence, and the University of Toronto. A list of dependencies for \href{https://github.com/tristanmontoya/StableSpectralElements.jl}{\texttt{StableSpectralElements.jl}} can be found on the solver's GitHub page; we are especially grateful for the core functionality made possible through the use of the \href{https://github.com/jlchan/StartUpDG.jl}{\texttt{StartUpDG.jl}}, \href{https://github.com/JuliaLinearAlgebra/LinearMaps.jl}{\texttt{LinearMaps.jl}}, and \href{https://github.com/SciML/DifferentialEquations.jl}{\texttt{DifferentialEquations.jl}} packages.

\bibliographystyle{elsarticle-num} 
\bibliography{refs.bib}

\providecommand*\hyphen{-}
\begin{thebibliography}{10}
\expandafter\ifx\csname url\endcsname\relax
  \def\url#1{\texttt{#1}}\fi
\expandafter\ifx\csname urlprefix\endcsname\relax\def\urlprefix{URL }\fi
\expandafter\ifx\csname href\endcsname\relax
  \def\href#1#2{#2} \def\path#1{#1}\fi

\bibitem{klockner_dg_gpu_09}
A.~Kl\"ockner, T.~Warburton, J.~Bridge, J.~S. Hesthaven, Nodal discontinuous
  {Galerkin} methods on graphics processors, Journal of Computational Physics
  228~(21) (2009) 7863--7882.

\bibitem{abdi_gpu_cg_dg_17}
D.~S. Abdi, L.~C. Wilcox, T.~Warburton, F.~X. Giraldo, A {GPU}-accelerated
  continuous and discontinuous {Galerkin} non-hydrostatic atmospheric model,
  International Journal of High Performance Computing Applications 33~(1)
  (2017) 81--109.

\bibitem{vermeire_17}
B.~C. Vermeire, F.~D. Witherden, P.~E. Vincent, On the utility of {GPU}
  accelerated high-order methods for unsteady flow simulations: A comparison
  with industry-standard tools, Journal of Computational Physics 334 (2017)
  497--521.

\bibitem{parsani_ssdc_21}
M.~Parsani, R.~Boukharfane, I.~R. Nolasco, D.~C. Del Rey~Fern\'andez,
  S.~Zampini, B.~Hadri, L.~Dalcin, High-order accurate entropy-stable
  discontinuous collocated {Galerkin} methods with the summation-by-parts
  property for compressible {CFD} frameworks: Scalable {SSDC} algorithms and
  flow solver, Journal of Computational Physics 424 (2021) 109844.

\bibitem{mossier_p_adaptive_dg_22}
P.~Mossier, A.~Beck, C.-D. Munz, A $p$-adaptive discontinuous {Galerkin} method
  with $hp$-shock capturing, Journal of Scientific Computing 91~(1) (2022) 4.

\bibitem{tadmor_entropy_stable_fv_87}
E.~Tadmor, The numerical viscosity of entropy stable schemes for systems of
  conservation laws. {I}, Mathematics of Computation 49~(179) (1987) 91--103.

\bibitem{lefloch_entropy_conservative_arbitrary_order_02}
P.~G. LeFloch, J.~M. Mercier, C.~Rohde, Fully discrete, entropy conservative
  schemes of arbitrary order, {SIAM} Journal on Numerical Analysis 40~(5)
  (2002) 1968--1992.

\bibitem{fisher_phd_thesis_12}
T.~C. Fisher, High-order {$L^2$} stable multi-domain finite difference method
  for compressible flows, Ph.D. thesis, Purdue University (2012).

\bibitem{svard_nordstrom_sbpreview_14}
M.~Sv\"{a}rd, J.~Nordstr\"{o}m, Review of summation-by-parts schemes for
  initial-boundary-value problems, Journal of Computational Physics 268 (2014)
  17--38.

\bibitem{delrey_sbp_sat_review_14}
D.~C. Del Rey~Fern\'andez, J.~E. Hicken, D.~W. Zingg, Review of
  summation-by-parts operators with simultaneous approximation terms for the
  numerical solution of partial differential equations, Computers \& Fluids 95
  (2014) 171--196.

\bibitem{fisher_carpenter_ssweno_sbp_13}
T.~C. Fisher, M.~H. Carpenter, High-order entropy stable finite difference
  schemes for nonlinear conservation laws: Finite domains, Journal of
  Computational Physics 252 (2013) 518--557.

\bibitem{fisher_carpenter_telescopingflux_conservation_13}
T.~C. Fisher, M.~H. Carpenter, J.~Nordstr\"{o}m, N.~K. Yamaleev, C.~Swanson,
  Discretely conservative finite-difference formulations for nonlinear
  conservation laws in split form: Theory and boundary conditions, Journal of
  Computational Physics 234 (2013) 353--375.

\bibitem{ismail_roe_ec_flux_09}
F.~Ismail, P.~L. Roe, Affordable, entropy-consistent {Euler} flux functions
  {II}: Entropy production at shocks, Journal of Computational Physics 228~(15)
  (2009) 5410--5436.

\bibitem{roe81}
P.~L. Roe, Approximate {Riemann} solvers, parameter vectors, and difference
  schemes, Journal of Computational Physics 43 (1981) 357--372.

\bibitem{gassner_dgsem_sbp_13}
G.~J. Gassner, A skew-symmetric discontinuous {Galerkin} spectral element
  discretization and its relation to {SBP}-{SAT} finite difference methods,
  {SIAM} Journal on Scientific Computing 35~(3) (2013) A1233--A1253.

\bibitem{carpenter_entropystable_collocation_14}
M.~H. Carpenter, T.~C. Fisher, E.~J. Nielsen, S.~H. Frankel, Entropy stable
  spectral collocation schemes for the {Navier}--{Stokes} equations:
  Discontinuous interfaces, {SIAM} Journal on Scientific Computing 36~(5)
  (2014) B835--B867.

\bibitem{gassner_winters_kopriva_splitform_nodaldg_sbp_16}
G.~J. Gassner, A.~R. Winters, D.~A. Kopriva, Split form nodal discontinuous
  {Galerkin} schemes with summation-by-parts property for the compressible
  {Euler} equations, Journal of Computational Physics 327 (2016) 39--66.

\bibitem{barth_numerical_methods_gasdynamic_systems_99}
T.~J. Barth, Numerical methods for gasdynamic systems on unstructured meshes,
  in: D.~Kr{\"o}ner, M.~Ohlberger, C.~Rohde (Eds.), An Introduction to Recent
  Developments in Theory and Numerics for Conservation Laws, Springer, 1999,
  pp. 195--285.

\bibitem{hiltebrand_mishra_entropystable_dg_exact_int_14}
A.~Hiltebrand, S.~Mishra, Entropy stable shock capturing space-time
  discontinuous {Galerkin} schemes for systems of conservation laws, Numerische
  Mathematik 126~(1) (2014) 103--151.

\bibitem{hughes_franca_mallet_entropy_stable_fem_86}
T.~J.~R. Hughes, L.~P. Franca, M.~Mallet, A new finite element formulation for
  computational fluid dynamics: {I}. symmetric forms of the compressible
  {Euler} and {Navier}--{Stokes} equations and the second law of
  thermodynamics, Computer Methods in Applied Mechanics and Engineering 54~(2)
  (1986) 223--234.

\bibitem{hicken_mdsbp_16}
J.~E. Hicken, D.~C. Del Rey~Fern\'andez, D.~W. Zingg, Multidimensional
  summation-by-parts operators: General theory and application to simplex
  elements, {SIAM} Journal on Scientific Computing 38~(4) (2016) A1935--A1958.

\bibitem{delrey_generalized_framework_14}
D.~C. Del Rey~Fern{\'{a}}ndez, P.~D. Boom, D.~W. Zingg, A generalized framework
  for nodal first derivative summation-by-parts operators, Journal of
  Computational Physics 266 (2014) 214--239.

\bibitem{chen_shu_entropy_stable_dgsbp_17}
T.~Chen, C.-W. Shu, Entropy stable high order discontinuous {Galerkin} methods
  with suitable quadrature rules for hyperbolic conservation laws, Journal of
  Computational Physics 345 (2017) 427--461.

\bibitem{crean_entropystable_sbp_euler_curved_17}
J.~Crean, J.~E. Hicken, D.~C. {Del Rey Fern{\'{a}}ndez}, D.~W. Zingg, M.~H.
  Carpenter, Entropy-stable summation-by-parts discretization of the {Euler}
  equations on general curved elements, Journal of Computational Physics 356
  (2018) 410--438.

\bibitem{chan_discretely_entropy_conservative_dg_sbp_18}
J.~Chan, On discretely entropy conservative and entropy stable discontinuous
  {Galerkin} methods, Journal of Computational Physics 362 (2018) 346--374.

\bibitem{chen_shu_dgsbp_review_19}
T.~Chen, C.-W. Shu, Review of entropy stable discontinuous {Galerkin} methods
  for systems of conservation laws on unstructured simplex meshes, {CSIAM}
  Transactions on Applied Mathematics 1~(1) (2020) 1--52.

\bibitem{orszag_spectral_complex_geometries_80}
S.~A. Orszag, Spectral methods for problems in complex geometries, Journal of
  Computational Physics 37~(1) (1980) 70--92.

\bibitem{montoya_sem_23}
T.~Montoya, D.~W. Zingg, Efficient tensor-product spectral-element operators
  with the summation-by-parts property on curved triangles and tetrahedra,
  {SIAM} Journal on Scientific Computing 46~(4) (2024) A2270--A2297.

\bibitem{sherwin_karniadakis_triangular_sem_95}
S.~J. Sherwin, G.~E. Karniadakis, A triangular spectral element method:
  Applications to the incompressible {Navier}--{Stokes} equations, Computer
  Methods in Applied Mechanics and Engineering 123~(1-4) (1995) 189--229.

\bibitem{sherwin_karniadakis_tetrahedra_hp_fem_96}
S.~J. Sherwin, G.~E. Karniadakis, Tetrahedral $hp$ finite elements: Algorithms
  and flow simulations, Journal of Computational Physics 124~(1) (1996) 14--45.

\bibitem{lomtev_karniadakis_dg_99}
I.~Lomtev, G.~E. Karniadakis, A discontinuous {Galerkin} method for the
  {Navier}--{Stokes} equations, International Journal for Numerical Methods in
  Fluids 29~(5) (1999) 587--603.

\bibitem{kirby_spectral_hp_dg_hybrid_grids_00}
R.~M. Kirby, T.~C. Warburton, I.~Lomtev, G.~E. Karniadakis, A discontinuous
  {Galerkin} spectral/$hp$ method on hybrid grids, Applied Numerical
  Mathematics 33~(1-4) (2000) 393--405.

\bibitem{moxey_matrix_free_triangles_20}
D.~Moxey, R.~Amici, R.~M. Kirby, Efficient matrix-free high-order finite
  element evaluation for simplicial elements, {SIAM} Journal on Scientific
  Computing 42~(3) (2020) C97--C123.

\bibitem{proriol_polynomials_57}
J.~Proriol, Sur une famille de polynomes \`a deux variables orthogonaux dans un
  triangle, Comptes Rendus Hebdomadaires des S\'eances de l'Acad\'emie des
  Sciences 245 (1957) 2459--2461.

\bibitem{koornwinder_orthogonal_polynomials_75}
T.~Koornwinder, Two-variable analogues of the classical orthogonal polynomials,
  in: R.~Askey (Ed.), Theory and Application of Special Functions, Academic
  Press, 1975, pp. 435--495.

\bibitem{dubiner_spectral_triangle_91}
M.~Dubiner, Spectral methods on triangles and other domains, Journal of
  Scientific Computing 6~(4) (1991) 345--390.

\bibitem{chan_weight_adjusted_dg_curvilinear_17}
J.~Chan, R.~J. Hewett, T.~Warburton, Weight-adjusted discontinuous {Galerkin}
  methods: Curvilinear meshes, {SIAM} Journal on Scientific Computing 39~(6)
  (2017) A2395--A2421.

\bibitem{chan_wilcox_entropystable_curvilinear_19}
J.~Chan, L.~C. Wilcox, On discretely entropy stable weight-adjusted
  discontinuous {Galerkin} methods: Curvilinear meshes, Journal of
  Computational Physics 378 (2019) 366--393.

\bibitem{montoya_tensor_product_22}
T.~Montoya, D.~W. Zingg, Stable and conservative high-order methods on
  triangular elements using tensor-product summation-by-parts operators, in:
  Eleventh International Conference on Computational Fluid Dynamics, 2022.

\bibitem{montoya_unifying_21}
T.~Montoya, D.~W. Zingg, A unifying algebraic framework for discontinuous
  {Galerkin} and flux reconstruction methods based on the summation-by-parts
  property, Journal of Scientific Computing 92~(3) (2022) 87.

\bibitem{rueda_ramirez_subcell_limiting_22}
A.~M. Rueda-Ram\'irez, W.~Pazner, G.~J. Gassner, Subcell limiting strategies
  for discontinuous {Galerkin} spectral element methods, Computers {\&} Fluids
  247 (2022) 105627.

\bibitem{yamaleev_positivity_entropy_stable_23}
N.~K. Yamaleev, J.~Upperman, High-order positivity-preserving entropy stable
  schemes for the {3-D} compressible {Navier}--{Stokes} equations, Journal of
  Scientific Computing 95~(1) (2023) 11.

\bibitem{lin_positivity_preservation_euler_23}
Y.~Lin, J.~Chan, I.~Tomas, A positivity preserving strategy for entropy stable
  discontinuous {Galerkin} discretizations of the compressible {Euler} and
  {Navier}--{Stokes} equations, Journal of Computational Physics 475 (2023)
  111850.

\bibitem{friedrichs_lax_systems_conservation_71}
K.~O. Friedrichs, P.~D. Lax, Systems of conservation equations with a convex
  extension, Proceedings of the National Academy of Sciences 68~(8) (1971)
  1686--1688.

\bibitem{kruzhkov_first_order_quasilinear_70}
S.~N. Kru{\v{z}}kov, First order quasilinear equations in several independent
  variables, Mathematics of the {USSR}--Sbornik 10~(2) (1970) 217--243.

\bibitem{lax_shock_waves_and_entropy_71}
P.~D. Lax, Shock waves and entropy, in: Contributions to Nonlinear Functional
  Analysis, Academic Press, 1971, pp. 603--634.

\bibitem{dafermos16}
C.~M. Dafermos, Hyperbolic Conservation Laws in Continuum Physics, Springer,
  2016.

\bibitem{delrey_mdsbp_sat_18}
D.~C. Del Rey~Fern{\'{a}}ndez, J.~E. Hicken, D.~W. Zingg, Simultaneous
  approximation terms for multi-dimensional summation-by-parts operators,
  Journal of Scientific Computing 75~(1) (2018) 83--110.

\bibitem{hesthaven08}
J.~S. Hesthaven, T.~Warburton, Nodal Discontinuous {Galerkin} Methods:
  Algorithms, Analysis, and Applications, Springer, 2008.

\bibitem{karniadakis_sherwin_spectral_hp_element}
G.~E. Karniadakis, S.~J. Sherwin, Spectral/hp Element Methods for Computational
  Fluid Dynamics, 2nd Edition, Oxford University Press, 2005.

\bibitem{pulliam14}
T.~H. Pulliam, D.~W. Zingg, Fundamentals of Computational Fluid Dynamics,
  Springer, 2014.

\bibitem{kopriva09}
D.~A. Kopriva, Implementing Spectral Methods for Partial Differential
  Equations: Algorithms for Scientists and Engineers, Springer, 2009.

\bibitem{kopriva_metric_identities_discontinuous_sem_curved_06}
D.~A. Kopriva, Metric identities and the discontinuous spectral element method
  on curvilinear meshes, Journal of Scientific Computing 26~(3) (2006)
  301--327.

\bibitem{thomas_lombard_gcl_79}
P.~D. Thomas, C.~K. Lombard, Geometric conservation law and its application to
  flow computations on moving grids, {AIAA} Journal 17~(10) (1979) 1030--1037.

\bibitem{chan_delrey_carpenter_gauss_collocation_19}
J.~Chan, D.~C. Del Rey~Fern{\'{a}}ndez, M.~H. Carpenter, Efficient entropy
  stable {Gauss} collocation methods, {SIAM} Journal on Scientific Computing
  41~(5) (2019) A2938--A2966.

\bibitem{kopriva_gassner_dgsem_variable_coeff_14}
D.~A. Kopriva, G.~J. Gassner, An energy stable discontinuous {Galerkin}
  spectral element discretization for variable coefficient advection problems,
  {SIAM} Journal on Scientific Computing 36~(4) (2014) A2076--A2099.

\bibitem{delrey_extension_of_dense_gsbp_tensor_curvilinear_19}
D.~C. Del Rey~Fern{\'{a}}ndez, P.~D. Boom, M.~H. Carpenter, D.~W. Zingg,
  Extension of tensor-product generalized and dense-norm summation-by-parts
  operators to curvilinear coordinates, Journal of Scientific Computing 80
  (2019) 1957--1996.

\bibitem{ranocha_comparison_entropy_conservative_fluxes_18}
H.~Ranocha, Comparison of some entropy conservative numerical fluxes for the
  {Euler} equations, Journal of Scientific Computing 76~(1) (2018) 216--242.

\bibitem{winters_matrix_diss_17}
A.~R. Winters, D.~Derigs, G.~J. Gassner, S.~Walch, A uniquely defined entropy
  stable matrix dissipation operator for high {Mach} number ideal {MHD} and
  compressible {Euler} simulations, Journal of Computational Physics 332 (2017)
  274--289.

\bibitem{chandrashekar_kep_entropy_fv_13}
P.~Chandrashekar, Kinetic energy preserving and entropy stable finite volume
  schemes for compressible {Euler} and {Navier}--{Stokes} equations,
  Communications in Computational Physics 14~(5) (2013) 1252--1286.

\bibitem{ranocha_efficient_implementation_23}
H.~Ranocha, M.~Schlottke-Lakemper, J.~Chan, A.~M. Rueda-Ram{\'{\i}}rez, A.~R.
  Winters, F.~Hindenlang, G.~J. Gassner, Efficient implementation of modern
  entropy stable and kinetic energy preserving discontinuous {Galerkin} methods
  for conservation laws, {ACM} Transactions on Mathematical Software 49~(4)
  (2023) 37.

\bibitem{worku_quadrature_sbp_23}
Z.~A. Worku, J.~E. Hicken, D.~W. Zingg, Quadrature rules on triangles and
  tetrahedra for multidimensional summation-by-parts operators, Preprint,
  arXiv:2311.15576 [math.NA] (2023).

\bibitem{xiao_gimbutas_quadrature_10}
H.~Xiao, Z.~Gimbutas, A numerical algorithm for the construction of efficient
  quadrature rules in two and higher dimensions, Computers {\&} Mathematics
  with Applications 59~(2) (2010) 663--676.

\bibitem{jaskowiec_sukumar_symmetric_cubature_21}
J.~Ja{\'s}kowiec, N.~Sukumar, High-order symmetric cubature rules for
  tetrahedra and pyramids, International Journal for Numerical Methods in
  Engineering 122~(1) (2021) 148--171.

\bibitem{cockburn_hdg_09}
B.~Cockburn, J.~Gopalakrishnan, R.~Lazarov, Unified hybridization of
  discontinuous {Galerkin}, mixed, and continuous {Galerkin} methods for second
  order elliptic problems, SIAM Journal on Numerical Analysis 47~(2) (2009)
  1319--1365.

\bibitem{chan_skewsymmetric_modaldg_sbp_19}
J.~Chan, Skew-symmetric entropy stable modal discontinuous {Galerkin}
  formulations, Journal of Scientific Computing 81~(1) (2019) 459--485.

\bibitem{warburton_unstructured_connectivity_95}
T.~Warburton, S.~J. Sherwin, G.~E. Karniadakis, Unstructured $hp$/spectral
  elements: Connectivity and optimal ordering, in: S.~N. Atluri, G.~Yagawa,
  T.~Cruse (Eds.), Computational Mechanics '95: Theory and Applications,
  Springer, 1995, pp. 433--444.

\bibitem{chan_bencomo_delrey_mortar_based_entropy_stable_21}
J.~Chan, M.~J. Bencomo, D.~C. Del Rey~Fern\'andez, Mortar-based entropy-stable
  discontinuous {Galerkin} methods on non-conforming quadrilateral and
  hexahedral meshes, Journal of Scientific Computing 89~(2) (2021) 51.

\bibitem{harten_entropy_symmetric_form_83}
A.~Harten, On the symmetric form of systems of conservation laws with entropy,
  Journal of Computational Physics 49~(1) (1983) 151--164.

\bibitem{ranocha_phd_thesis_18}
H.~Ranocha, Generalised summation-by-parts operators and entropy stability of
  numerical methods for hyperbolic balance laws, Ph.D. thesis, Technische
  Universit\"at Braunschweig (2018).

\bibitem{ranocha_ec_kep_20}
H.~Ranocha, Entropy conserving and kinetic energy preserving numerical methods
  for the {Euler} equations using summation-by-parts operators, in: S.~J.
  Sherwin, D.~Moxey, J.~Peir{\'o}, P.~E. Vincent, C.~Schwab (Eds.), Spectral
  and High Order Methods for Partial Differential Equations ICOSAHOM 2018,
  Springer, 2020, pp. 525--535.

\bibitem{ranocha_pressure_oscillations_22}
H.~Ranocha, G.~J. Gassner, Preventing pressure oscillations does not fix local
  linear stability issues of entropy-based split-form high-order schemes,
  Communications on Applied Mathematics and Computation 4~(3) (2022) 880--903.

\bibitem{davis_godunov_88}
S.~F. Davis, Simplified second-order {Godunov}-type methods, {SIAM} Journal on
  Scientific and Statistical Computing 9~(3) (1988) 445--473.

\bibitem{chan_warburton_interpolation_nodes_15}
J.~Chan, T.~Warburton, A comparison of high order interpolation nodes for the
  pyramid, SIAM Journal on Scientific Computing 37~(5) (2015) A2151--A2170.

\bibitem{jiang_shu_efficient_implementation_weno_96}
G.-S. Jiang, C.-W. Shu, Efficient implementation of weighted {ENO} schemes,
  Journal of Computational Physics 126~(1) (1996) 202--228.

\bibitem{rackauckas_julia_diffeq_17}
C.~Rackauckas, Q.~Nie, {DifferentialEquations.jl}--a performant and
  feature-rich ecosystem for solving differential equations in {Julia}, Journal
  of Open Research Software 5~(1) (2017) 15.

\bibitem{hairer_93}
E.~Hairer, S.~P. N{\o}rsett, G.~Wanner, Solving Ordinary Differential Equations
  I. Nonstiff Problems, 2nd Edition, Springer, 1993.

\bibitem{pope_turbulent_flows}
S.~B. Pope, Turbulent Flows, Cambridge University Press, 2000.

\bibitem{rueda_ramirez_fv_limiter_21}
A.~M. Rueda-Ram{\'{\i}}rez, G.~J. Gassner, A subcell finite volume
  positivity-preserving limiter for {DGSEM} discretizations of the {Euler}
  equations, in: {WCCM}--{ECCOMAS} Congress, 2021.

\bibitem{chan_entropy_projection_22}
J.~Chan, H.~Ranocha, A.~M. Rueda-Ram{\'{\i}}rez, G.~J. Gassner, T.~Warburton,
  On the entropy projection and the robustness of high order entropy stable
  discontinuous {Galerkin} schemes for under-resolved flows, Frontiers in
  Physics 10 (2022) 898028.

\bibitem{pazner_persson_entropy_stable_line_dg_19}
W.~Pazner, P.-O. Persson, Analysis and entropy stability of the line-based
  discontinuous {Galerkin} method, Journal of Scientific Computing 80~(1)
  (2019) 376--402.

\bibitem{ponce_niagara_19}
M.~Ponce, R.~van Zon, S.~Northrup, D.~Gruner, J.~Chen, F.~Ertinaz, A.~Fedoseev,
  L.~Groer, F.~Mao, B.~C. Mundim, M.~Nolta, J.~Pinto, M.~Saldarriaga,
  V.~Slavnic, E.~Spence, C.-H. Yu, W.~R. Peltier, Deploying a top-100
  supercomputer for large parallel workloads, in: Proceedings of the Practice
  and Experience in Advanced Research Computing on Rise of the Machines
  (Learning), Association for Computing Machinery, 2019.

\end{thebibliography}

\end{document}